\documentclass[10pt,twoside]{siamart1116}

\usepackage[english]{babel}
\usepackage{graphicx,epstopdf,epsfig}
\usepackage{amsfonts,epsfig,fancyhdr,graphics, hyperref,amsmath,amssymb}
\usepackage{amsfonts}
\usepackage{amscd}
\usepackage{amsmath}
\usepackage{amssymb}
\usepackage{epic}
\usepackage{multirow}
\usepackage{makeidx}
\usepackage{mathscinet}
\usepackage{hyperref}
\usepackage{graphicx}
\usepackage{subfigure}
\usepackage{fancyhdr}
\usepackage[english]{babel}
\usepackage[latin1]{inputenc}
\usepackage{srcltx}
\usepackage{epstopdf}
\usepackage{geometry}
\usepackage{marginnote}
\usepackage{mathdots}
\DeclareMathOperator{\heads}{\mathrm{heads}}

\usepackage{blkarray}

\usepackage{color}
\usepackage{mathtools}
\usepackage{enumerate}
\usepackage{verbatim}
\usepackage{arydshln}
\DeclareMathOperator{\csf}{\mathrm{csf}}
\DeclareMathOperator{\rev}{\mathrm{rev}}
\DeclareMathOperator{\diag}{\mathrm{diag}}
\definecolor{Myblue2}{rgb}{0.7,0,1}

\newcommand{\red}{\color{red} }

\newcommand{\HE}{Handling Editor}
\newcommand{\DoS}{Date of Submission}
\newcommand{\DoA}{Date of Acceptance}

\newcommand{\Names}{M. I. Bueno, M. Martin, J. Perez, A. Song, and I. Viviano}
\newcommand{\Title}{Explicit block-structures for block-symmetric Fiedler-like pencils}

\renewcommand{\theequation}{\thesection.\arabic{equation}}
\renewtheorem{theorem}{Theorem}[section]
\newtheorem{remark}[theorem]{Remark}
\newtheorem{example}[theorem]{Example}

\begin{document}

\bibliographystyle{plain}

\setcounter{page}{1}

\thispagestyle{empty}

 \title{\Title\thanks{Received
 by the editors on \DoS.
 Accepted for publication on \DoA. 
 Handling Editor: \HE.}}
 
 \author{
M. I. Bueno\thanks{Department of Mathematics and College of Creative Studies, University of California, Santa Barbara, CA 93106, USA ({\tt mbueno@math.ucsb.edu}). 
The research of M. I. Bueno was partially supported by NSF grant DMS-1358884 and partially supported by Ministerio de Economia y Competitividad of Spain through grants MTM2015-65798-P.}
\and
M. Martin \thanks{Brown University Providence ({\tt madeline\_martin@brown.edu}).
The research of M. Martin was partially supported by NSF grant DMS-1358884.}
\and
J. P\'erez \thanks{Department of Mathematical Sciences, University of Montana, MT, USA ({\tt javier.perez-alvaro@mso.umt.edu.}). The research of J. P\'erez was partially supported by KU Leuven Research Council grant OT/14/074 and the Interuniversity Attraction Pole DYSCO, initiated by the Belgian State Science Policy Office.}
\and
A. Song \thanks{University of California, Santa Barbara ({\tt alexandersong@umail.ucsb.edu}).
The research of A. Song was partially supported by NSF grant DMS-1358884.}
\and
I. Viviano \thanks{Wake Forest University ({\tt viviiv14@wfu.edu}).
The research of  I. Viviano was partially supported by NSF grant DMS-1358884.}}


\markboth{\Names}{\Title}

\maketitle

\begin{abstract}
In the last decade, there has been a continued effort to produce families of strong linearizations of a matrix polynomial $P(\lambda)$, regular and singular, with good  properties, such as, being companion forms, allowing the recovery of eigenvectors of a regular $P(\lambda)$ in an easy way, allowing the computation of the minimal indices of a singular $P(\lambda)$ in an easy way, etc.
 As a consequence of this research, families such as the family of Fiedler pencils, the family of generalized Fiedler pencils (GFP), the family of Fiedler pencils with repetition, and the family of generalized Fiedler pencils with repetition (GFPR) were constructed. 
In particular, one of the goals was to find in these families structured linearizations of structured matrix polynomials. For example, if a matrix polynomial $P(\lambda)$ is symmetric (Hermitian), it is convenient to use linearizations of $P(\lambda)$ that are also symmetric (Hermitian). Both the family of GFP and the family of GFPR  contain block-symmetric linearizations of $P(\lambda)$, which are symmetric (Hermitian) when $P(\lambda)$ is. Now the objective is to determine which of those structured linearizations have the best numerical properties. The main obstacle for this study is the fact that these pencils are defined implicitly as products of so-called elementary matrices.  Recent  papers in the literature had as a goal to  provide an  explicit block-structure for the pencils belonging to the family of Fiedler pencils and any of its further generalizations to solve this problem. In particular, it was shown that all GFP and GFPR, after permuting some block-rows and block-columns, belong to the family of extended block Kronecker pencils, which are defined explicitly in terms of their block-structure. Unfortunately, those permutations that transform a GFP or a GFPR into an extended block Kronecker pencil do not preserve the block-symmetric structure. Thus, in this paper we consider the family of block-minimal bases pencils, which is closely related to the family of  extended block Kronecker pencils, and whose pencils are also defined in terms of their block-structure, as a source of canonical forms for block-symmetric pencils. More precisely,  we present four families of block-symmetric pencils which, under some generic nonsingularity conditions are block minimal bases pencils and strong linearizations of a matrix polynomial. 
 We  show that the block-symmetric GFP and GFPR, after some row and column permutations, belong to the union of these four families.
Furthermore, we show that, when $P(\lambda)$ is a complex matrix polynomial, any block-symmetric GFP and GFPR is permutationally congruent to a pencil in some of these four families.
Hence, these four families of  pencils provide an alternative but explicit approach to the block-symmetric Fiedler-like pencils existing in the literature. 
\end{abstract}




\section{Introduction}

The standard approach to numerically solving a \emph{polynomial eigenvalue problem (PEP)} associated with a matrix polynomial (whose matrix coefficients have entries in a field $\mathbb{F}$) of the form
\begin{equation}\label{eq:polynomial}
P(\lambda)=\sum_{i=0}^k A_i \lambda^i, \quad \mbox{with }A_0,A_1,\hdots,A_k\in\mathbb{F}^{m\times n},
\end{equation}
starts by embedding the coefficients of $P(\lambda)$ into a matrix pencil (that is, a matrix polynomial of grade equal to 1). 
This process is known as \emph{linearization}, and it transforms the given PEP into a \emph{generalized eigenvalue problem (GEP)}.
Then, the obtained GEP  can be solved by using the QZ algorithm \cite{QZ} or the staircase algorithm \cite{staircase,VanDooren83}, for example.

The literature on linearizations is huge as can be seen, for example, by counting all the references in \cite{canonical_Fiedler} concerning this topic. 
The best well-known examples  of linearizations of a matrix polynomial $P(\lambda)$ as in (\ref{eq:polynomial}) are the so-called \emph{Frobenius companion forms} given by
\[
\begin{bmatrix}
\lambda A_k+A_{k-1} & A_{k-2} & \cdots &  A_0 \\ 
-I_n & \lambda I_n \\
& \ddots & \ddots \\
& & -I_n & \lambda I_n 
\end{bmatrix}
\quad \mbox{and} \quad 
\begin{bmatrix}
\lambda A_k+A_{k-1} & -I_m \\
A_{k-2} & \lambda I_m & \ddots  \\
\vdots & & \ddots & -I_m \\
A_0 & &   & \lambda I_m
\end{bmatrix}.
\]

We note that here and throughout the paper, we sometimes omit the  block-entries of a matrix polynomial that are equal to zero as we have done above. The algorithm QZ implemented in Matlab to solve the PEP uses the first Frobenius companion form as a linearization by default.

Frobenius companion forms have  many desirable properties from a numerical point of view, as i) they are constructed from the matrix coefficients of $P(\lambda)$ without performing any arithmetic operations; ii) they are strong linearizations of $P(\lambda)$ regardless of whether $P(\lambda)$ is regular or singular \cite{DTDM10,DTDM12}; iii) the minimal indices of $P(\lambda)$ are related with the minimal indices of the Frobenius companion forms by uniform shifts {\red \cite{singular,DTDM10}}; iv)  eigenvectors of regular matrix polynomials and minimal bases of singular matrix polynomials  are easily recovered from those of the Frobenius companion forms \cite{DTDM10};  and v) solving PEP's by applying a backward stable eigensolver to the Frobenius companion forms is backward stable \cite{canonical,VanDooren83}.
Nonetheless, solving a PEP by solving the GEP associated with a Frobenius companion form presents some significant drawbacks.
For instance, if the matrix polynomial $P(\lambda)$ is symmetric (Hermitian), that is $P(\lambda)^T=P(\lambda)$ ( $\mathbb{F}=\mathbb{C}$ and $P(\lambda)^* = P(\lambda)$), neither of the Frobenius companion forms is symmetric (Hermitian). 
Since the preservation of the structure has been recognized as key for obtaining better (and physically more meaningful)  numerical results \cite{GoodVibrations}, this drawback has motivated an intense research on structure-preserving linearizations; see, for example \cite{Her-GFPR,FPR1,GFPR,FPR2,FPR3,DTDM11,Philip2016,GoodVibrations,Leo,ant-vol11}, to name a few recent references on this topic. There are many  papers in the literature addressing the problem of constructing symmetric (Hermitian) strong linearizations of symmetric (Hermitian) matrix polynomials.
 Most of these papers approach the problem by constructing first {block-symmetric} strong linearizations, as it is done, for example, in \cite{FPR1,GFPR,symmetric-Mackey,GoodVibrations}.
 
Among the block-symmetric linearizations in the literature, it has been shown that, within the vector space   of block-symmetric pencils $\mathbb{DL}(P)$ \cite{symmetric-Mackey,v4space}, the first and last pencils in its standard basis, denoted by $D_1(\lambda, P)$ and $D_k(\lambda, P)$, respectively, have almost optimal behavior in terms of conditioning and backward error when used to compute an eigenvalue $\delta$ of $P(\lambda)$, as long as $|\delta| \geq 1$ if $D_1(\lambda, P)$ is used or $|\delta| \leq 1$ if  $D_k(\lambda, P)$ is used \cite{condition,backward}. A natural question is whether a single block-symmetric linearization can be found with good conditioning and backward error regardless of the modulus of $\delta$ or if any block-symmetric linearizations  outside $\mathbb{DL}(P)$ present a better numerical behavior than $D_1(\lambda, P)$ and $D_k(\lambda, P)$. One possible approach to answering these questions consists in replacing some nonzero blocks of the form $\pm A_i$ in the matrix coefficients  of these pencils (which can be seen as block-matrices whose  blocks are of the form $0$, $\pm I_n$,  and $\pm A_i$) by zero or identity blocks. But in order to do that, it is necessary to identify which of those blocks are essential  to keep a given linearization  a linearization of $P(\lambda)$ as they are replaced by zero or identity blocks. Thus, an explicit block-structure of the well-known block-symmetric pencils in the literature that allows to determine easily if they are a linerization of $P(\lambda )$ or not can be useful, for example, to accomplish this goal. In this paper we focus on the block-symmetric pencils in the families of Fiedler-like pencils presented in \cite{FPR1,GFPR}, which are known as \emph{block-symmetric generalized Fiedler pencils} \emph{(block-symmetric GFP)} and  \emph{block-symmetric generalized} \emph{ Fiedler pencils with repetition} \emph{(block-symmetric GFPR)}.
The block-symmetric GFP are strong linearizations of any  $P(\lambda)$. 
 The block-symmetric GFPR are  strong linearizations of $P(\lambda)$ modulo some generic nonsingularity conditions. Moreover, all block-symmetric GFP and GFPR  are symmetric (Hermitian) when $P(\lambda)$ is.
Furthermore, they share some of the desirable properties of the Frobenius companion forms mentioned above.
 The main disadvantage of these pencils is that they were defined implicitly in terms of products of elementary matrices, which makes it difficult to study their algebraic and numerical properties. Thus, identifying their block structure might solve some of these difficulties.

The family of block minimal bases pencils was recently constructed with the goal of  performing a backward stability analysis of PEP's when solved  by linearization  \cite{canonical}.
These pencils are defined by their explicit block-structure. Moreover, it has been shown that, modulo some generic nonsingularity conditions, Fiedler pencils, generalized Fiedler pencils, Fiedler pencils with repetition (and, thus, the standard basis of the $\mathbb{DL}(P)$ space), and generalized Fiedler pencils with repetition are \emph{permutationally equivalent} to block minimal bases pencils \cite{canonical_Fiedler,canonical}.  
However, none of these results takes into account any extra structural properties that these pencils might possess. 
For example, given a block-symmetric GFPR, the results in \cite{canonical_Fiedler,canonical} do not guarantee that this pencil is \emph{permutationally block-congruent}\footnote{Given two block-symmetric pencils $L_1(\lambda)$ and $L_2(\lambda)$, we say that they are \emph{permutationally block-congruent} if there exists a  block-permutation matrix $Q$  such that $ L_1(\lambda)= Q  L_2(\lambda) Q^{\mathcal{B}}$, where $M^{\mathcal{B}}$ denotes the block-transpose of the matrix $M$} to a block-symmetric  block minimal bases pencil.  The  focus of this paper is not on  constructing new families of block-symmetric  pencils but on identifying a family of block-symmetric pencils, that under some generic nonsingularity conditions are block minimal bases pencils, and  showing that the block-symmetric GFP and the block-symmetric GFPR are  permutationally block-congruent to a pencil in that family. 
 This family of block-symmetric minimal bases pencils can be divided into four subfamilies, two associated with odd degree polynomials and two associated with even degree polynomials.
 Each of these subfamilies is built by applying certain block-congruences to a very simple block-symmetric block minimal bases pencil, the ``skeleton'' or ``generator'' of the family. The ``skeleton'' of each family contains a ``minimal'' block-structure (in the sense that its matrix coefficients contain more zero blocks and less nonzero nonidentity
blocks than any other pencil in the family) that guarantees it being a strong linearization of a given matrix polynomial $P(\lambda)$. 
Hence, this approach allows to identify the block-entries of the block-structure of strong linearizations based on block-symmetric Fiedler-like pencils (including the basis of
$\mathbb{DL}(P)$) that are essential to embed the spectral information of $P(\lambda)$ in the pencil and the block-entries that are not while preserving the block-symmetry. 
We expect these ``skeletons'' to be candidates to have optimal numerical properties among the block-symmetric linearizations in the family they ``generate''.

The rest of the paper is structured as follows.
In section \ref{sec:notation}, we review the basic theory of matrix polynomials, linearizations, minimal bases and dual minimal bases needed throughout the paper.
In Section \ref{sec:BMBP}, we recall the definitions of the family of  block minimal bases pencils and the family of  extended block Kronecker pencils.
By using extended block Kronecker pencils, we introduce in Section \ref{sec:four-families} four families of block-symmetric pencils which are block minimal bases pencils under generic nonsingularity conditions, and contain infinitely many  block-symmetric strong linearizations of a matrix polynomial.
These pencils are explicitly defined in terms of their block-entries.  
 In Section \ref{sec:GFPR-def}, we recall the definitions of block-symmetric GFP and block-symmetric GFPR.
Finally,  in Section \ref{sec:main-GFP}  we give a result that states that  the block-symmetric GFP associated with an odd degree matrix polynomial and any block-symmetric GFPR  is permutationally block-congruent to a pencil belonging to some of the four families introduced in Section \ref{sec:four-families}.
Thus, these four families of pencils provide an alternative and simplified approach to block-symmetric Fiedler-like pencils by providing their explicit block-structure. 
The proof of the result for the block-symmetric GFPR turns out to be quite involved, long and highly technical. One reason for this is that the family of block-symmetric GFPR is infinite and we are stating theorems that hold true for all the pencils in this  family. The other reason, as we said before, is that  these pencils are defined in an implicit way in terms of products of matrices, which makes the work with them quite cumbersome. The implicit definition of these pencils also leads to the use of a very heavy notation. Since the proof of this result is very similar to the proof of Theorem 8.1 in \cite{canonical_Fiedler}, we include its proof in the Appendix. Now that we have an explicit definition of the block-symmetric GFPR in terms of their block entries, all the notation and the original implicit definition can be abandoned. 
What remains is a simpler description of block-symmetric Fiedler-like linearizations as block-symmetric block minimal bases pencils. This explicit definition of the block-symmetric GFPR has already proven to be useful. In \cite{TPodd,TPeven}, it has been used to identify sparse pencils that outperform numerically (in terms of conditioning and backward error) the block-symmetric linearizations $D_1(\lambda, P)$ and $D_k(\lambda, P)$ in the standard basis of $\mathbb{DL}(P)$.

\section{Notation and background}\label{sec:notation}

 Throughout the paper, we use the following notation. 
If $a$ and $b$ are two integers, we define
$$a:b:=\left\{ \begin{array}{ll} a, a+1,\ldots, b, & \textrm{if $a\leq b$,}\\
\emptyset, & \textrm{if $a>b$.}
 \end{array}\right. $$

In this work we consider square matrix polynomials whose matrix coefficients have entries in a field $\mathbb{F}$, that is, matrix polynomials as in \eqref{eq:polynomial} with $m=n$.
The number $k$ in \eqref{eq:polynomial} is called the \emph{grade} of $P(\lambda)$.
The \emph{degree} of $P(\lambda)$ is defined as the largest $d$ such that $A_d\neq 0$.
Notice that the degree is a number intrinsic to $P(\lambda)$, while the grade is an option (larger than or equal to the degree).

A square matrix polynomial $P(\lambda)$ is said to be \emph{regular} if the scalar polynomial $\det (P(\lambda))$ is not the zero polynomial; otherwise $P(\lambda)$ is said to be \emph{singular}.
Furthermore, if $\det P(\lambda)\in\mathbb{F}$, $P(\lambda)$ is called a \emph{unimodular matrix polynomial}.
The \emph{complete eigenstructure} of a regular matrix polynomial consists of its finite and infinite elementary divisors.
For a singular matrix polynomial, the complete eigenstructure consists of its finite and infinite elementary divisors together with its right and left minimal indices.
For more detailed definitions of the complete eigenstructure of matrix polynomials, we refer the reader to \cite[Section 2]{IndexSum}.

By the \emph{polynomial eigenvalue problem (PEP)} associated with a matrix polynomial $P(\lambda)$, we refer to the problem of computing the complete eigenstructure of $P(\lambda)$.
If $P(\lambda)$ is a matrix pencil, the associated PEP is referred to as a \emph{generalized eigenvalue problem (GEP)}.
A \emph{strong linearization} of a regular matrix polynomial $P(\lambda)$ is a matrix pencil $\mathcal{L}(\lambda)$ having the same finite and infinite elementary divisors as $P(\lambda)$; when $P(\lambda)$ is singular, a strong linearization must also have the same numbers of right and left minimal indices as $P(\lambda)$.
Hence, the PEP associated with the polynomial $P(\lambda)$ can be solved by solving the GEP associated with $\mathcal{L}(\lambda)$ provided that the minimal indices of $P(\lambda)$ and $\mathcal{L}(\lambda)$ are related in a simple way.

Given two matrix polynomials $P(\lambda)$ and $Q(\lambda)$ of the same size, we recall the following equivalence relations.
The polynomials $P(\lambda)$ and $Q(\lambda)$ are said to be 
\begin{itemize}
\item[\rm (i)] \emph{unimodularly equivalent} if there are unimodular matrix polynomials $U(\lambda)$ and $V(\lambda)$ such that $Q(\lambda)=U(\lambda)P(\lambda)V(\lambda)$; and
\item[\rm (ii)] \emph{strictly equivalent} if there are nonsingular constant matrices $U$ and $V$ such that $Q(\lambda) = UP(\lambda)V$.
\end{itemize}
 We recall that unimodular equivalence preserves the finite eigenstructure of matrix polynomials, while strict equivalence preserves the whole eigenstructure \cite{gantmacher}. 
 In this work we also use  the following concepts extensively.
\begin{itemize}
\item[\rm (i)] Given an $s\times t$ block matrix $M=[M_{ij}]$ with $n\times n$ block-entries $M_{ij}$, the \emph{block-transpose} matrix of $M$, denoted by $M^\mathcal{B}$, is the $t\times s$ block-matrix whose $(i,j)$ block-entry is $M_{ji}$. 
\item[\rm (ii)] Given a $k\times k$ block matrix $M=[M_{ij}]$ with $n\times n$ block-entries $M_{ij}$, we say that $M$ is \emph{block-symmetric} if $M^\mathcal{B}=M$.
\item[\rm (iii)] A $kn\times kn$ permutation matrix $\Pi$ is called a \emph{block-permutation} matrix if $\Pi=P\otimes I_n$, for some $k\times k$ permutation matrix $P$, where $\otimes$ denotes the Kronecker product of two matrices.
\item[\rm (iv)] We say that the matrix polynomials $P(\lambda)$ and $Q(\lambda)$ are \emph{permutationally equivalent} if there are permutation matrices $\Pi_1$ and $\Pi_2$ such that $Q(\lambda) = \Pi_1 P(\lambda) \Pi_2$.
\item[\rm (v)] We say that the $kn\times kn$ matrix polynomials $P(\lambda)$ and $Q(\lambda)$  are \emph{permutationally block-congruent} if there exists a block-permutation matrix $\Pi$ such that $Q(\lambda)=\Pi P(\lambda) \Pi^\mathcal{B}$.
\end{itemize} 
We notice that permutational equivalence and, thus, permutational block-congruency  are particular instances of strict equivalence.
Hence, the matrix polynomials $P(\lambda)$,  $\Pi P(\lambda) \Pi^\mathcal{B}$ and $\Pi_1 P(\lambda)\Pi_2$ have the same eigenstructure (finite, infinite and singular).
 Furthermore, since the block-entries of any block permutation matrix $\Pi$ are either the zero or the identity matrices, $P(\lambda)$ is block-symmetric if and only if $\Pi P(\lambda) \Pi^\mathcal{B}$ is block-symmetric.

Here and thereafter, we denote by $\mathbb{F}[\lambda]^{m\times n}$ the set of $m\times n$ matrix polynomials, by $\mathbb{F}(\lambda)$ the field of rational functions over $\mathbb{F}$ and by $\mathbb{F}(\lambda)^n$ the set of $n$-tuplas with entries in $\mathbb{F}(\lambda)$.
By $\overline{\mathbb{F}}$ we denote the algebraic closure of $\mathbb{F}$.
Any subspace $\mathcal{W}\subseteq \mathbb{F}(\lambda)^n$ is called a \emph{rational subspace}.
We recall that any  $\mathcal{W}\subseteq \mathbb{F}(\lambda)^n$ has bases consisting entirely of vectors with polynomial entries.

Key for this work are the so-called \emph{minimal bases} and \emph{dual minimal bases}, introduced by Forney \cite{Forney}.
For their definitions, we rely on the concept of  \emph{row-degrees vector} of  an $m \times n$ matrix polynomial  $P(\lambda)$, which is a row vector of length $m$ whose $i$th component is  the maximum of the degrees of the entries in the $i$th row of $P(\lambda)$.
For example, the row-degrees vector of the matrix
\begin{equation}\label{eq:example}
\begin{bmatrix}
1 & \lambda^2 & 1-\lambda \\
0 & 1 & \lambda 
\end{bmatrix}
\end{equation}
is $[ 2, 1]$.
\begin{definition}
Let $\mathcal{W}$ be a  rational subspace of $ \mathbb{F}(\lambda)^n$.
We say that a matrix polynomial $L(\lambda)\in\mathbb{F}[\lambda]^{m\times n}$ is a \emph{minimal basis} of $\mathcal{W}$ if  its rows form a basis for $\mathcal{W}$  and the sum of the entries of its row-degrees vector is minimal among all the possible polynomial bases for $\mathcal{W}$.
Furthermore, the entries of the  row-degrees vector of $L(\lambda)$ are called the \emph{minimal indices} of $\mathcal{W}$. 
\end{definition}
\begin{remark} 
For simplicity,  we  say that ``$L(\lambda)\in \mathbb{F}[\lambda]^{m\times n}$ is a minimal basis'' to mean that ``$L(\lambda)$ is a minimal basis for the subspace of $\mathbb{F}(\lambda)^n$ spanned by its rows''.
\end{remark}

The following characterization of minimal bases is very useful in practice.
\begin{theorem}\label{thm:minimal_basis}
Let $L(\lambda)\in \mathbb{F}[\lambda]^{m\times n}$ and let $[d_1, \ldots, d_m]$ be the row-degrees vector of $L(\lambda)$. 
Then, $L(\lambda)$ is a minimal basis if and only if $L(\lambda_0)$ has full row rank for all $\lambda_0 \in \overline{\mathbb{F}}$ and the $m\times n$ constant matrix whose $(i, j)$th entry is the coefficient of $\lambda^{d_j}$ in the $(i, j)$th entry of $L(\lambda)$  has full row rank.
\end{theorem}

\begin{example}
The matrix polynomial in \eqref{eq:example} is a minimal basis because it clearly has full row rank for every $\lambda_0\in \overline{\mathbb{F}}$ and the matrix $\left[\begin{smallmatrix} 0 & 1 & 0 \\ 0 & 0 & 1 \end{smallmatrix}\right]$  has full row rank as well.
\end{example}

\begin{definition} \label{def:dualminimalbases}
Two matrix polynomials $L(\lambda)\in\mathbb{F}[\lambda]^{m_1\times n}$ and $N(\lambda)\in\mathbb{F}[\lambda]^{m_2\times n}$ are called \emph{dual minimal bases} if $m_1+m_2 = n$, $L(\lambda)N(\lambda)^T=0$, and $L(\lambda)$ and $N(\lambda)$ are both minimal bases.
\end{definition}
\begin{remark}
We will say that ``$N(\lambda)$ is a minimal basis dual to $L(\lambda)$'', or vice versa, when referring to matrix polynomials $L(\lambda)$ and $N(\lambda)$ as those in Definition \ref{def:dualminimalbases}.
\end{remark}

Continuing with the example in \eqref{eq:example}, it is easy to show that the matrix polynomials
\[
\begin{bmatrix}
1 & \lambda^2 & 1-\lambda \\
0 & 1 & \lambda 
\end{bmatrix} \quad \mbox{and} \quad
\begin{bmatrix}
\lambda^3+\lambda-1 & -\lambda & 1
\end{bmatrix}
\]
are dual minimal bases.

In the following proposition, we introduce the most important pair of dual minimal bases used in this work. 
\begin{proposition}{\rm \cite{canonical}}
Let 
\begin{equation}\label{eq:Lk}
L_s(\lambda):=\begin{bmatrix}
-1 & \lambda  \\
& -1 & \lambda \\
& & \ddots & \ddots \\
& & & -1 & \lambda  \\
\end{bmatrix}\in \mathbb{F}[\lambda]^{s\times(s+1)},
\end{equation}
and
\begin{equation}
  \label{eq:Lambda}
  \Lambda_s(\lambda):=
\begin{bmatrix}
      \lambda^{s} & \cdots & \lambda & 1
\end{bmatrix} \in \mathbb{F}[\lambda]^{1\times (s+1)}.
\end{equation}
Then,  for every positive integer $p$, the matrix polynomials $L_s(\lambda)\otimes I_p$ and $\Lambda_s(\lambda)\otimes I_p$ are dual minimal bases.
\end{proposition}

The following proposition concerning dual minimal bases will be useful.
\begin{proposition}\label{BL-dual}
Let $L(\lambda)$ be a minimal basis.
If $B$ is a nonsingular matrix, then $BL(\lambda)$ is also a minimal basis.
Further, if $N(\lambda)$ is any minimal basis dual to $L(\lambda)$, $N(\lambda)$ is also dual to $BL(\lambda)$.
\end{proposition}
\begin{proof}
The proof follows immediately from the characterization of minimal bases in Theorem \ref{thm:minimal_basis}, and the definition of dual minimal bases in Definition \ref{def:dualminimalbases}.
\end{proof}


\section{Block minimal bases pencils and extended block Kronecker pencils}\label{sec:BMBP}

We recall in this section the familis of  block minimal bases pencils and  of extended block Kronecker pencils, and state their main properties used in this work.

\subsection{Block minimal bases pencils}

The  block minimal bases pencils were introduced in \cite{canonical}. 
The definition of block minimal bases pencil involves the concept of minimal basis and pair of dual minimal bases introduced in the previous section.

\begin{definition}\label{def:BMBP}
A matrix pencil 
\begin{equation}\label{BMBP}
C(\lambda)=\left[ \begin{array}{c|c} M(\lambda) & G_2(\lambda)^T\\ \hline G_1(\lambda) & 0 \end{array} \right]
\end{equation}
is called a \emph{block minimal bases pencil} if $G_1(\lambda)$ and $G_2(\lambda)$  are both minimal bases. If, in addition, the  row-degrees vector of $G_1(\lambda)$ (resp. $G_2(\lambda)$)  have all entries equal to 1 and  the entries of the  row-degrees vector of a minimal basis dual to $G_1(\lambda)$ (resp. $G_2(\lambda)$) are all equal, then $C(\lambda)$ is called a \emph{strong block minimal bases pencil}. 
\end{definition}

A fundamental property of any strong block minimal bases pencil of the form \eqref{BMBP} is that it is a strong linearization of some matrix polynomial expressed in terms of the block-entry $M(\lambda)$ and the dual minimal bases of $G_1(\lambda)$ and $G_2(\lambda)$. 
\begin{theorem}\label{thm:linearizationBMBP}{\rm \cite[Theorem 3.3]{canonical}}
Let $C(\lambda)$ be a strong block minimal bases pencil as in \eqref{BMBP}.
Let $N_1(\lambda)$ (resp. $N_2(\lambda)$)  be a  minimal basis dual to $G_1(\lambda)$ (resp. $G_2(\lambda))$ whose row-degrees vector has equal entries.
Let
\begin{equation}\label{eq:Qpolinminbaslin}
Q(\lambda):=N_2(\lambda) M(\lambda) N_1(\lambda)^T.\end{equation}
Then, $C(\lambda)$  is a strong linearization of $Q(\lambda)$, considered as a polynomial of grade  $1 + \deg(N_1(\lambda))+\deg(N_2(\lambda))$. 
\end{theorem}


\subsection{Extended block Kronecker pencils}

Next we recall a family of pencils  that has played an important role in the canonical expression of the GFP and GFPR  in terms of their block-structure \cite{canonical_Fiedler}.
The pencils in this  family are called {\em extended block Kronecker pencils}. 
In their definition,  we use the dual minimal bases $L_s(\lambda)$ and $\Lambda_s(\lambda)$ introduced, respectively, in \eqref{eq:Lk} and \eqref{eq:Lambda}.

\begin{definition} \label{def:Kronecker_pencils}{\rm \cite[Definition 3.5]{canonical_Fiedler}}
 Let $M(\lambda)$ be an arbitrary $(q+1)m \times (p+1)n$ pencil. 
 Let $A\in \mathbb{F}^{np\times np}$ and $B \in \mathbb{F}^{mq\times mq}$ be arbitrary matrices.
Then the matrix pencil
\begin{equation}
  \label{eq:BKP}
  \begin{array}{cl}
  C(\lambda)=
  \left[
    \begin{array}{c|c}
      M(\lambda) & (L_{q}(\lambda)^{T}\otimes I_{m}) B\\\hline
      A(L_{p}(\lambda)\otimes I_{n})&0
      \end{array}
    \right]&
    \begin{array}{l}
      \left. \vphantom{L_{\mu}^{T}(\lambda)\otimes I_{m}} \right\} {\scriptstyle (q+1)m}\\
      \left. \vphantom{L_{\epsilon}(\lambda)\otimes I_{n}}\right\} {\scriptstyle p n}
    \end{array}\\
    \hphantom{C(\lambda)=}
    \begin{array}{cc}
      \underbrace{\hphantom{L_{\epsilon}(\lambda)\otimes I_{n}}}_{(p+1)n}&\underbrace{\hphantom{L_{\mu}^{T}(\lambda)\otimes I_{m}}}_{q m}
    \end{array}
  \end{array}
  \> 
\end{equation}
where $L_p (\lambda)$ and $L_q(\lambda)$ are as  in \eqref{eq:Lk}, is called an extended  \emph{$(p,n,q,m)$-block Kronecker pencil} or, simply, an {\em  extended block Kronecker pencil}. When $A=I_{np}$ and $B=I_{mq}$, then $C(\lambda)$ is called a block Kronecker pencil.
The block $M(\lambda)$ is called the \emph{body of $C(\lambda)$.}
\end{definition}

Note that, if $A$ and $B$ are nonsingular matrices, then $C(\lambda)$ is a (strong) block minimal bases pencil (see Proposition \ref{BL-dual}). 
However, if either $A$ or $B$ is singular, it is not guaranteed that $C(\lambda)$ is a block minimal bases pencil.


One advantage of the extended block Kronecker pencils with $A$ and $B$ nonsingular over more general strong block minimal bases pencils is that  it is easy to give simple characterizations for all the grade-1 solutions $M(\lambda)$ of the equation
\begin{equation}\label{eq:eq_for_P_2} 
(\Lambda_q(\lambda)^T\otimes I_m)M(\lambda)(\Lambda_p(\lambda)\otimes I_n)=P(\lambda),
\end{equation}
for a prescribed matrix polynomial $P(\lambda)$ of grade  $k=p+q+1$.

The following definition will be used in one of such characterizations.

\begin{definition}\label{AS-def}{\rm\cite[Definition 3.7]{canonical_Fiedler}}
Let  $M(\lambda)=\lambda M_1 + M_0\in \mathbb{F}[\lambda]^{(q+1)m \times (p+1)n}$ be a matrix pencil and set $k:=p+q+1$. 
Let us denote by $[M_0]_{ij}$ and $[M_1]_{ij}$ the $(i,j)$th block-entries of $M_0$ and $M_1$, respectively, when $M_0$ and $M_1$ are partitioned as $(q+1)\times (p+1)$ block-matrices with blocks of size $m\times n$. 
We call the \emph{antidiagonal sum of $M(\lambda)$ related to $s\in \{ 0:k\}$} the matrix
\[
\mathrm{AS}(M, s):=\sum_{i+j=k+2-s} [M_1]_{ij} +\sum_{i+j=k+1-s} [M_0]_{ij}.
\]
Additionally, given a matrix polynomial $P(\lambda)=\sum_{i=0}^k A_i \lambda^i \in\mathbb{F}[\lambda]^{m\times n}$, we say that $M(\lambda)$ satisfies the \emph{antidiagonal sum condition (AS condition)} for $P(\lambda)$ if 
\begin{equation}\label{AS-condition}
\mathrm{AS}(M, s)= A_s, \quad s=0:k.
\end{equation}
\end{definition}
The  AS condition has been used in the construction of large classes of linearizations  of a matrix polynomial $P(\lambda)$ easily constructible from the coefficients of $P(\lambda)$; see \cite[Theorem 5.4]{canonical} or \cite[Section 3]{Philip2016}.

\begin{example}
Let $P(\lambda)= \sum_{i=0}^5 A_i \lambda^i$. The matrix pencil 
$$M(\lambda) = \left[ \begin{array}{ccc} A_5 \lambda & 0 &  0 \\
A_4 \lambda & 0 & 0 \\ A_3\lambda & A_2\lambda  &  A_1 \lambda + A_0 \end{array} \right].$$
 satisfies the AS condition for $P(\lambda)$. 
\end{example}

 \begin{theorem}\label{thm:Lambda-dual-pencil-linearization}
Let $P(\lambda)=\sum_{i=0}^k A_i\lambda^i\in\mathbb{F}[\lambda]^{m\times n}$, and let $C(\lambda)$ be an extended block Kronecker pencil as in \eqref{eq:BKP} with $p+q+1=k$ and  with  body $M(\lambda)$.
The following conditions are equivalent.
\begin{itemize}
\item[\rm (a)] The pencil  $M(\lambda)$ satisfies \eqref{eq:eq_for_P_2}.
\item[\rm (b)] The pencil $M(\lambda)$  satisfies the AS condition for $P(\lambda)$.
\item[\rm (c)]  The pencil  $M(\lambda)$ is of the form
\[
M(\lambda) = M_0(\lambda)+C_1(L_p(\lambda)\otimes I_n)+(L_q(\lambda)^T\otimes I_m)C_2,
\]
where $M_0(\lambda)$ is any solution of \eqref{eq:eq_for_P_2} and  $C_1\in\mathbb{F}^{(q+1)m\times pn}$ and $C_2\in\mathbb{F}^{qm \times (p+1)n}$ are arbitrary matrices.
\end{itemize}
%
\end{theorem}
\begin{proof}
The proof that  (a) and (b) are equivalent can be obtained by some simple algebraic manipulations. 
The proof that parts (a) and (c) are equivalent can be found in   \cite{Philip2016} (in a paragraph just before Theorem 1). 
\end{proof}

Now, as an immediate consequence of Theorem \ref{thm:Lambda-dual-pencil-linearization}, we  obtain the following family of strong linearizations of $P(\lambda)$.
 
 \begin{theorem}\label{thm:linearizationBKP_mod}
Let $P(\lambda)$ be a matrix polynomial, and let $p,q$ be nonnegative integers such that $p+q+1=\deg(P(\lambda))$. 
Let $M_0(\lambda)$ be a pencil satisfying the AS condition for $P(\lambda)$.
Then, 
any pencil of the form  
 \begin{equation}\label{special-pencil}
   \left[
    \begin{array}{c|c}
      M_0(\lambda)+C_1(L_p(\lambda)\otimes I_n)+(L_q(\lambda)^T\otimes I_m) C_2
      &(L_{q}(\lambda)^{T}\otimes I_{m}) B_2\\\hline
      B_1(L_{p}(\lambda)\otimes I_{n}) &0
      \end{array} \right],
 \end{equation}
 where $C_1\in\mathbb{F}^{(q+1)m\times pn}$ and $C_2\in\mathbb{F}^{qm\times (p+1)n}$ are arbitrary matrices, and $B_1\in\mathbb{F}^{pn\times pn}$ and $B_2\in\mathbb{F}^{qm\times qm}$ are arbitrary nonsingular matrices,  is a strong linearization of $P(\lambda)$.
\end{theorem}
\begin{proof}
The result is an immediate consequence of Theorems \ref{thm:linearizationBMBP} and \ref{thm:Lambda-dual-pencil-linearization}, together with the fact that, when the matrices $B_1$ and $B_2$ are nonsingular, the extended block Kronecker pencil  \eqref{special-pencil} is a strong block minimal bases pencil  (see Proposition \ref{BL-dual}). 
\end{proof}
 
 \begin{remark}\label{rem:thm3.9}
Observe that any pencil of the form (\ref{special-pencil}) is an extended block Kronecker pencil whose body satisfies the AS condition for $P(\lambda)$ since, given two matrix pencils $M_1(\lambda)$ and $M_2(\lambda)$,  $\mathrm{AS}(M_1+M_2, s) = \mathrm{AS}(M_1, s) + \mathrm{AS}(M_2, s)$. 
Moreover, the pencil in (\ref{special-pencil}) can be expressed as follows:
 $$\left[ \begin{array}{cc} I & C_1\\ 0 & B_1 \end{array} \right] \left[ \begin{array}{c|c} M_0(\lambda) & L_q(\lambda)^T \otimes I_m \\ \hline L_p(\lambda) \otimes I_n & 0 \end{array} \right] \left[ \begin{array}{cc} I & 0\\  C_2 & B_2 \end{array} \right].$$
 \end{remark}

Theorem \ref{thm:linearizationBKP_mod} will be key to provide a simple canonical block-structure for block-symmetric Fiedler-like pencils  under  permutational block-congruence operations.
The description of these block-structures is the main goal of the following section.

\section{The four families of block-symmetric minimal bases pencils}\label{sec:four-families}

We introduce in this section four types of block-symmetric pencils associated with a matrix polynomial $P(\lambda)$, which are block minimal bases pencils, under some generic nonsingularity conditions,  and we give their explicit block structure.
We will show later that the block-symmetric Fiedler-like pencils known in the literature belong to one of these families, modulo a permutational block-congruence.  

Since block-symmetric Fiedler-like pencils  are only defined for square matrix polynomials, here and thereafter, we restrict our study to square matrix polynomials $P(\lambda)\in\mathbb{F}[\lambda]^{n\times n}$.
As the size of $P(\lambda)$ is always going to be denoted by $n$, there is no risk of confusion if we introduce the notation
\[
K_s(\lambda):= L_s(\lambda)\otimes I_n,
\] 
with $L_s(\lambda)$ as in (\ref{eq:Lk}). 
We note that
\begin{equation}\label{TvsB}
K_s(\lambda)^T = K_s(\lambda)^{\mathcal{B}} \quad \mbox{and} \quad  (\Lambda_s(\lambda)\otimes I_n)^T=(\Lambda_s(\lambda)\otimes I_n)^\mathcal{B},
\end{equation}
with $\Lambda_s(\lambda)$ as in (\ref{eq:Lambda}),  when $K_s(\lambda)$ is seen as an $s \times (s+1)$ block matrix with blocks of size $n\times n$ and $\Lambda_s(\lambda) \otimes I_n$ is seen as a $1 \times (s+1)$ block matrix with blocks of size $n\times n$.  Moreover, if $B$ is an $s \times s$ block matrix, then
\begin{equation}\label{BKsB}
 (B K_s(\lambda))^{\mathcal{B}} = K_s(\lambda)^{\mathcal{B}} B^{\mathcal{B}} = K_s(\lambda)^T B^{\mathcal{B}}.
 \end{equation}

Additionally, we introduce the block-symmetric pencil
\begin{equation}\label{eq:M(lambda,P)}
M (\lambda;Q) := 
\begin{bmatrix}
\lambda Q_d+Q_{d-1} \\ & \ddots \\ & & \lambda Q_1+Q_0
\end{bmatrix}\in\mathbb{F}[\lambda]^{\frac{n(d+1)}{2}\times \frac{n(d+1)}{2}}
\end{equation}
associated with a matrix polynomial $Q(\lambda)=\sum_{i=0}^d Q_i\lambda^i\in\mathbb{F}[\lambda]^{n\times n}$ of odd degree $d$, which will play a fundamental role in what follows. 
Notice that $M(\lambda;Q)$ satisfies the AS condition for  $Q(\lambda)$.

Associated with the matrix polynomial $P(\lambda)=\sum_{i=0}^k A_i\lambda^i \in\mathbb{F}[\lambda]^{n\times n}$, we define the following matrix polynomials 
\begin{align}
\label{eq:Pd}&P^{k-1}(\lambda):= A_{k-1}\lambda^{k-1}+\cdots+\lambda A_1+A_0,\\
\label{eq:Pell}&P^{k-1}_{k-1}(\lambda):=A_{k-1}\lambda^{k-2}+\cdots+A_2\lambda+A_1, \quad \mbox{and}\\
\label{eq:Pu}&P_{k-1}(\lambda):=A_k\lambda^{k-1}+\cdots+\lambda A_2+A_1,
\end{align}
which will be used in the definition  of the four families of block-symmetric  pencils introduced in this section. 
Note that $P^{k-1}(\lambda)$ is a truncation of degree $k-1$ of  $P(\lambda)$ while $P_{k-1}(\lambda)$ is the so-called $(k-1)$th Horner shift polynomial associated with $P(\lambda)$.

\subsection{The first fundamental block-structure}
We introduce here the first of the families of block-symmetric  pencils.
Let  $P(\lambda)\in\mathbb{F}[\lambda]^{n\times n}$ be a matrix polynomial of odd degree $k$, and let $s:=(k-1)/2$.
We start by defining the pencil
\begin{equation}\label{eq:O1}
   \mathcal{O}^P_1(\lambda) := \left[
        \begin{array}{c|c}
            M(\lambda;P)
            & K_s(\lambda)^T\\
            \hline
            K_s(\lambda) & 0
        \end{array}
    \right]\in\mathbb{F}[\lambda]^{nk\times nk}
\end{equation}
 where $M(\lambda;P)$ is defined in \eqref{eq:M(lambda,P)}.
 By Definition \ref{def:Kronecker_pencils}, the  pencil $\mathcal{O}^P_1(\lambda)$ is an $(s,n,s,n)$-  block Kronecker pencil and a strong block minimal bases pencil.
Furthermore, taking into account \eqref{TvsB}, it is clearly block-symmetric.
 Notice additionally that,  by Theorem \ref{thm:linearizationBKP_mod},  $\mathcal{O}^P_1(\lambda)$ is a strong linearization of $P(\lambda)$ because $M(\lambda;P)$ satisfies the AS condition for $P(\lambda)$.
Thus, the pencil $\mathcal{O}^P_1(\lambda)$ is a block-symmetric strong linearization of $P(\lambda)$.

 We can obtain many more block-symmetric strong linearizations of $P(\lambda)$ by considering pencils obtained by applying the block-congruence
\begin{equation}\label{eq:equivalence for O1}
\left[ \begin{array}{cc} I_{(s+1)n} & C \\ 0 & B \end{array} \right ] \left [ \begin{array}{c|c} M(\lambda;P) & K_s(\lambda)^T \\ \hline K_s(\lambda) & 0 \end{array} \right]\left[ \begin{array}{cc} I_{(s+1)n} & 0 \\ C^{\mathcal{B}} &   B^\mathcal{B}\end{array}
\right],
\end{equation}
where   $B=[B_{ij}]$ is an $s\times s$ block matrix and  $C=[C_{ij}]$ is an  $(s+1)\times s$ block-matrix, with $n\times n$ block-entries $B_{ij}$ and $C_{ij}$, respectively.
The pencil \eqref{eq:equivalence for O1} motivates the first fundamental block-structure family associated with the matrix polynomial $P(\lambda)$. 
\begin{definition}\label{def:first_family}
Let $P(\lambda)=\sum_{i=0}^k A_i\lambda^i\in\mathbb{F}[\lambda]^{n\times n}$ be a matrix polynomial of odd degree $k$,  and let $s=(k-1)/2$.
The {\em first fundamental block-structure family}, denoted by $\langle \mathcal{O}_1^P \rangle $, is the set of pencils of the form
\begin{equation}\label{eq:first_family}
        \left[\begin{array}{c|c}
            M(\lambda;P) + C K_{s}(\lambda)+ K_{s}^T(\lambda)C^\mathcal{B} 
            & K_{s}(\lambda)^T  B^\mathcal{B} \\
            \hline
            BK_{s}(\lambda) & 0
        \end{array}\right],
\end{equation}
 where $M(\lambda;P)$ is defined in \eqref{eq:M(lambda,P)}, and $B=[B_{ij}]$ and $C=[C_{ij}]$  are, respectively,   some arbitrary $s\times s$ block matrix and  $(s+1)\times s$ block matrix, with $n\times n$ block-entries $B_{ij}$ and $C_{ij}$.
 \end{definition}

\begin{remark}\label{O1-struct}
The matrix pencil in \eqref{eq:first_family}, which is also the pencil in (\ref{eq:equivalence for O1}), can  be  expressed as follows: 
\begin{align*}
&\left[ \begin{array}{cc} I_{(s+1)n} & C \\ 0 & B \end{array} \right ] \left [ \begin{array}{c|c} M(\lambda;P) & K_s(\lambda)^T \\ \hline K_s(\lambda) & 0 \end{array} \right]\left[ \begin{array}{cc} I_{(s+1)n} & C \\ 0 & B \end{array} \right ]^{\mathcal{B}},
\end{align*}
where the block transpose is applied on the matrix $ \left[ \begin{array}{cc} I_{(s+1)n} & C \\ 0 & B \end{array} \right ]$ when considered a $k\times k$ block matrix. That is, every pencil in $\langle \mathcal{O}_1^P \rangle$ is block congruent to $\mathcal{O}_1^P$ and, therefore, block-symmetric. 
\end{remark}

By \eqref{TvsB} and Definition \ref{def:Kronecker_pencils}, any pencil in the family $\langle \mathcal{O}_1^P \rangle $ is a block-symmetric $(s,n,s,n)$-extended block Kronecker pencil. 
 Moreover, if $B$ and $B^{\mathcal{B}}$ are nonsingular, each pencil in this family is a strong block minimal bases pencil, which leads to the following theorem.

 \begin{theorem}\label{thm:first family}
Let $P(\lambda)=\sum_{i=0}^k A_i\lambda^i\in\mathbb{F}[\lambda]^{n\times n}$ be a matrix polynomial of odd degree $k$, let $s=(k-1)/2$, and let $\mathcal{L}(\lambda) \in \langle \mathcal{O}_1^P \rangle$, that is, $\mathcal{L}(\lambda)$ is of the form (\ref{eq:first_family}).
If $B$ and $B^\mathcal{B}$ are  nonsingular, then the pencil $\mathcal{L}(\lambda)$ is a block-symmetric strong linearization of $P(\lambda)$.
Moreover, if $P(\lambda)$ and all the block-entries $B_{ij}$ are symmetric (resp. Hermitian), then the pencil $\mathcal{L}(\lambda)$ is symmetric (resp. Hermitian).
\end{theorem}
\begin{proof}
The fact that $\mathcal{L}(\lambda)$ is a strong linearization of $P(\lambda)$ when $B$ and $B^\mathcal{B}$ are nonsingular is an immediate consequence of Theorem \ref{thm:linearizationBKP_mod}.
The pencil $\mathcal{L}(\lambda)$ is block-symmetric as a consequence of \eqref{TvsB}, together with the fact that $M(\lambda;P)$ is block-symmetric.
The fact that $\mathcal{L}(\lambda)$ is symmetric (resp. Hermitian) when $P(\lambda)$ and all the block-entries $B_{ij}$ of $B$ are symmetric (resp. Hermitian) follows easily from the facts that $M(\lambda;P)$ is symmetric (resp. Hermitian) when $P(\lambda)$ is symmetric (resp. Hermitian), and that $B^\mathcal{B}=B^T$ ($B^\mathcal{B}=B^*$) and $C^\mathcal{B}=C^T$ ($C^\mathcal{B}=C^*$) when all the block-entries $B_{ij}$ and $C_{ij}$ are symmetric (resp. Hermitian).
\end{proof}

\begin{example}\label{D1DkO1}
As mentioned in the introduction, the best well-known block-symmetric pencils in the literature are those in the vector space $\mathbb{DL}(P)$. The pencils in the standard basis of this space are block-symmetric GFPR of special importance. Let $P(\lambda)$ be a matrix polynomial of odd degree $k$ and let $m$ be an odd positive integer. Then, as we will show in Theorem \ref{thm:main_GFPR}, the $m$th pencil $D_m(\lambda, P)$  in the standard basis of  $\mathbb{DL}(P)$, which is a GFPR with parameter $h=k-m$,  is permutationally block-congruent to a pencil  in $\langle \mathcal{O}_1^P\rangle.$  This holds, in particular, for $D_1(\lambda, P)$ and $D_k(\lambda, P)$.
\end{example}

\subsection{The second fundamental block-structure}
 We introduce in this section  the second fundamental  family of block-symmetric pencils.
This family is also associated with odd-degree matrix polynomials, but describing its block-structure is more involved.

Let  $P(\lambda)=\sum_{i=0}^k A_i\lambda^i$ be an $n\times n$ matrix polynomial of odd degree $k$, and let $s:=(k-1)/2$.
First, we define the pencil
\begin{equation}\label{eq:O2}
 \mathcal{O}_2^P(\lambda) := \left[
        \begin{array}{c:c:c|c}
  -A_k & \begin{matrix} \lambda A_k & 0 \end{matrix}&0& 0\\
            \hdashline
            \begin{matrix} \lambda A_k \\ 0 \end{matrix} & M(\lambda;P_{k-1}^{k-1}) & \begin{matrix} 0 \\ A_0 \end{matrix} & K_{s-1}(\lambda)^T \\
            \hdashline
            0&\begin{matrix} 0 & A_0 \end{matrix} & -\lambda A_0 & 0 \\  \hline
            0 &K_{s-1}(\lambda) & 0 & 0
        \end{array}
    \right],
\end{equation}
where  $P_{k-1}^{k-1}(\lambda)$ is defined in \eqref{eq:Pell} and $M(\lambda;P_{k-1}^{k-1})$ is defined in \eqref{eq:M(lambda,P)}.
 Notice that the pencil $\mathcal{O}_2^P(\lambda)$ is a block-symmetric block minimal bases pencil. 
However, this pencil is not an extended block-Kronecker pencil.

Next we give an example  to clarify the block-structure of  the pencil $\mathcal{O}^P_2(\lambda)$.
\begin{example}
Let $P(\lambda)=\sum_{i=0}^7 A_i\lambda^i\in\mathbb{F}[\lambda]^{n\times n}$.
Then,
\[
  \mathcal{O}^P_2(\lambda) = \left[\begin{array}{c:ccc:c|cc}
-A_7 & \lambda A_7 & 0 & 0 & 0 & 0 & 0 \\ \hdashline
\lambda A_7 & \lambda A_6+A_5 & 0 &0 & 0 & -I_n & 0 \\
0 & 0 & \lambda A_4+A_3 & 0 & 0 & \lambda I_n & -I_n \\
0 & 0 & 0 & \lambda A_2+A_1 & A_0 & 0 & \lambda I_n \\ \hdashline
0 & 0 & 0 & A_0 & -\lambda A_0 & 0 & 0 \\ \hline
0 & -I_n & \lambda I_n & 0 & 0 & 0 & 0\\
0 & 0 & -I_n & \lambda I_n & 0 & 0 & 0
  \end{array}\right].
\]
 Notice that, if we denote by $\Pi_2$ the block-permutation matrix that permutes the first block-column of $\mathcal{O}_2^P$ with the block-columns in positions 2--5,  we have
\begin{align*}
  \mathcal{O}_2^P(\lambda)\Pi_2 = &\left [ \begin{array}{c|c} M(\lambda) & K_3(\lambda)^TB_2\\ \hline B_1 K_3(\lambda) & 0 \end{array} \right ] := \\
  & \left[\begin{array}{cccc|ccc}
 \lambda A_7 & 0 & 0 & 0& -A_7 & 0 & 0 \\ 
 \lambda A_6+A_5 & 0 &0 & 0 &\lambda A_7 & -I_n & 0 \\
 0 & \lambda A_4+A_3 & 0 & 0 & 0 &\lambda I_n & -I_n \\
 0 & 0 & \lambda A_2+A_1 & A_0 & 0 &0 & \lambda I_n \\ \hline 
 0 & 0 & A_0 & -\lambda A_0 & 0 &0 & 0 \\ 
-I_n & \lambda I_n & 0 & 0 & 0 & 0 & 0\\
 0 & -I_n & \lambda I_n & 0 & 0 &0 & 0
  \end{array}\right],
 \end{align*}
  where 
  $$B_1= \left[ \begin{array}{ccc} 0 &  0 & -A_0 \\ I_n & 0 & 0 \\ 0 & I_n & 0  \end{array} \right]  \quad \mbox{and} \quad  B_2 = \left[ \begin{array}{ccc} A_7 & 0 & 0 \\0 & I_n & 0 \\ 0 & 0 & I_n  \end{array} \right] .$$
Thus, although $\mathcal{O}_2^P(\lambda)$ is not an extended block Kronecker pencil, it is only a column-permutation away from being so. 
 It is easy to see that the body $M(\lambda)$ of $ \mathcal{O}_2^P(\lambda)\Pi_2$ satisfies the AS condition for $P(\lambda)$.
 Hence, by Theorem \ref{thm:linearizationBKP_mod} and Remark \ref{rem:thm3.9},  the pencil $ \mathcal{O}_2^P(\lambda)\Pi_2$, and therefore $\mathcal{O}_2^P(\lambda)$,  is a strong linearization of $P(\lambda)$ if $A_0$ and $A_k$ are nonsingular matrices.
\end{example}

The procedure used in the previous example can be generalized to matrix polynomials of any odd-degree $k$. 
Denoting by $\Pi_2$ the block-permutation matrix that permutes the first block-column of  $\mathcal{O}^P_2(\lambda)$, defined in \eqref{eq:O2}, with the block-columns in positions  $2$ through $s+2=\frac{k+3}{2}$, we obtain 
\begin{equation}\label{eq:O2mod}
\mathcal{O}_2^P(\lambda)\Pi_2:=\left[
        \begin{array}{c:c|c:c}
 \begin{matrix} \lambda A_k & 0 \end{matrix}&0 & -A_k &0\\
          \hdashline  
    M(\lambda;P_{k-1}^{k-1}) & \begin{matrix} 0 \\ A_0 \end{matrix} &         \begin{matrix} \lambda A_k \\ 0 \end{matrix} & K_{s-1}(\lambda)^T \\
            \hline
            \begin{matrix} 0 & A_0 \end{matrix} & -\lambda A_0 & 0 & 0 \\  \hdashline  
   
            K_{s-1}(\lambda) & 0 & 0 & 0
        \end{array}
    \right],
\end{equation}
which is an $(s,n,s,n)$-extended block Kronecker pencil. 
Furthermore, if $A_0$ and $A_k$ are nonsingular, from Theorem \ref{thm:linearizationBKP_mod}, Remark \ref{rem:thm3.9},  and the fact that 
\[
\left[ \begin{array}{c:c}
\begin{matrix} \lambda A_k & 0 \end{matrix} & 0 \\ \hdashline
M(\lambda;P_{k-1}^{k-1}) & \begin{matrix} 0 \\ A_0 \end{matrix}
\end{array} \right],
\] 
satisfies  the AS condition for $P(\lambda)$,  it is immediately obtained that the pencil in \eqref{eq:O2mod} is a strong linearization of $P(\lambda)$.
In summary, the pencil $\mathcal{O}_2^P(\lambda)$ is a block-symmetric strong linearization of $P(\lambda)$ if $A_0$ and $A_k$ are nonsingular.
Moreover, $\mathcal{O}_2^P(\lambda)$ is symmetric (resp. Hermitian) whenever $P(\lambda)$ is symmetric (resp. Hermitian).

Motivated by the block-structure of the pencil \eqref{eq:O2mod} and by Theorem \ref{thm:linearizationBKP_mod}, we now consider a subfamily of extended block Kronecker pencils constructed from $\mathcal{O}_2^P(\lambda)\Pi_2$. 
Note that, among all the possible operations that would transform $\mathcal{O}_2^P(\lambda)\Pi_2$ into another extended block Kronecker pencil, we are only applying  some that will preserve the block-symmetry once the $(s+2)$th block column is permuted back to the original position, that is, the first block-column.
More precisely, we begin by considering  pencils of the form
\begin{equation}\label{gen-21}
 \left[ \begin{array}{cccc} I_{n} & 0 & 0 & B\\ 0 & I_{sn} & 0 & C\\ 0 & 0 & I_n & D \\ 0 & 0 & 0 & E\end{array} \right]
 \mathcal{O}_2^P(\lambda)\Pi_2 
      \left[ \begin{array}{cccc} I_{sn} & 0 & 0 & 0\\ 0 & I_n & 0 & 0\\ 0 & 0 & I_n & 0 \\ C^{\mathcal{B}} & D^{\mathcal{B}} & B^{\mathcal{B}} &  E^\mathcal{B}\end{array} \right],
    \end{equation}
 for some arbitrary $1\times (s-1)$, $s\times (s-1)$, $1\times (s-1)$ and $(s-1)\times (s-1)$ block matrices $B=[B_{ij}]$, $C=[C_{ij}]$, $D=[D_{ij}]$ and $E=[E_{ij}]$, with $n\times n$ block-entries $B_{ij}$, $C_{ij}$, $D_{ij}$ and $E_{ij}$.
 
Then, permuting  the $(s+2)$th block-column of the above pencil back to the first position, we get
\begin{equation*}
  \left[ \begin{array}{cccc} I_{n} & 0 & 0 & B\\ 0 & I_{sn} & 0 & C\\ 0 & 0 & I_{n} & D \\ 0 & 0 & 0 & E\end{array} \right]
 \mathcal{O}_2^P(\lambda)
\left[ \begin{array}{cccc} 0 &  0 & I_n & 0\\ I_{sn} & 0 & 0 & 0\\  0 & I_n & 0 &0  \\ C^{\mathcal{B}} & D^{\mathcal{B}} & B^{\mathcal{B}} & E^\mathcal{B} \end{array} \right]\Pi_2^\mathcal{B},
    \end{equation*}
    which equals
   \begin{equation}\label{gen-2}\left[ \begin{array}{cccc} I_{n} & 0 & 0 & B\\ 0 & I_{sn} & 0 & C\\ 0 & 0 & I_n & D \\ 0 & 0 & 0 & E\end{array} \right]
    \mathcal{O}_2^P(\lambda)
\left[ \begin{array}{cccc} I_n & 0 & 0 & B\\ 0 & I_{sn} & 0 & C\\ 0 & 0 & I_n & D \\ 0 & 0 & 0 & E \end{array} \right]^{\mathcal{B}}. 
\end{equation}
In this way,  we obtain the block-structure \eqref{gen-2} defining the second  fundamental family of block-structures associated with the matrix polynomial $P(\lambda)$.
\begin{definition}\label{def:first_family}
Let $P(\lambda)=\sum_{i=0}^k A_i\lambda^i\in\mathbb{F}[\lambda]^{n\times n}$ be a matrix polynomial of odd degree $k$,  let $s=(k-1)/2$.
{\rm The second fundamental block-structure family}, denoted by $\langle \mathcal{O}_2^P \rangle $, is the set of pencils of the form  (we are omitting the dependence on $\lambda$ in the pencil $K_{s-1}(\lambda)$ for lack of space) 
\begin{equation}\label{eq:second_family}
  \left[
        \begin{array}{c:c:c|c}
  -A_k & \begin{bmatrix} \lambda A_k & 0 \end{bmatrix} + BK_{s-1}&0& 0\\
            \hdashline
            \begin{bmatrix} \lambda A_k \\ 0 \end{bmatrix} + K_{s-1}^TB^\mathcal{B} & M(\lambda;P_{k-1}^{k-1}) + CK_{s-1}+K_{s-1}^TC^\mathcal{B}& \begin{bmatrix} 0 \\ A_0 \end{bmatrix} +  K_{s-1}^T D^\mathcal{B} & K_{s-1}^T E^\mathcal{B} \\
            \hdashline
            0&\begin{bmatrix} 0 & A_0 \end{bmatrix} + DK_{s-1} & -\lambda A_0 & 0 \\  \hline
            0 &EK_{s-1} & 0 & 0
        \end{array}
    \right],
\end{equation}
 where  $P_{k-1}^{k-1}(\lambda)$ is defined in \eqref{eq:Pell} and  $M(\lambda;P_{k-1}^{k-1})$ is defined in \eqref{eq:M(lambda,P)}, for some arbitrary  $1\times (s-1)$ block-matrix $B=[B_{ij}]$, $s\times (s-1)$ block-matrix $C=[C_{ij}]$, $1\times (s-1)$ block-matrix $D=[D_{ij}]$, and $(s-1)\times (s-1)$ block-matrix $E=[E_{ij}]$, with $n\times n$ block-entries $B_{ij}$, $C_{ij}$, $D_{ij}$ and $E_{ij}$, respectively.
\end{definition}

We note that every pencil in $\langle \mathcal{O}_2^P \rangle$ is a block minimal bases pencil if $E$ and $E^{\mathcal{B}}$ are nonsingular matrices.



The following theorem gives sufficient conditions for pencils in the family $\langle \mathcal{O}_2^P \rangle $ to be strong linearizations of an odd-degree matrix polynomial $P(\lambda)$.

 \begin{theorem}\label{thm:second family}
Let $P(\lambda)=\sum_{i=0}^k A_i\lambda^i\in\mathbb{F}[\lambda]^{n\times n}$ be a matrix polynomial of odd degree $k$, let $s=(k-1)/2$, and consider a pencil $\mathcal{L}(\lambda)\in \langle \mathcal{O}_2^P \rangle$, that is, a pencil  of the form (\ref{eq:second_family}).
If $A_0$, $ A_k$, $E$ and $E^\mathcal{B}$  are nonsingular, then $\mathcal{L}(\lambda)$ is a block-symmetric strong linearization of $P(\lambda)$.
Furthermore, if $P(\lambda)$, and all the block-entries of $B$, $C$, $D$ and $E$ are symmetric (resp. Hermitian), then $\mathcal{L}(\lambda)$  is symmetric (resp. Hermitian).
\end{theorem}
\begin{proof}
When $A_0$ and $A_k$ are nonsingular, the extended block Kronecker pencil \eqref{eq:O2mod} and, thus, the pencil $\mathcal{O}_2^P(\lambda)$ are strong linearizations of $P(\lambda)$. 
In addition, we see from \eqref{gen-2} that if $E$ and $E^\mathcal{B}$ are nonsingular, then the pencil $\mathcal{L}(\lambda)$ is strictly equivalent to the pencil $\mathcal{O}_2^P(\lambda)$.
Therefore, in this case, $\mathcal{L}(\lambda)$ is a strong linearization of $P(\lambda)$. 
The pencil $\mathcal{L}(\lambda)$ is block-symmetric as a consequence of \eqref{TvsB}, together with the fact that $M(\lambda;P_{k-1}^{k-1})$ is block-symmetric. 
The fact that $\mathcal{L}(\lambda)$ is symmetric (resp. Hermitian) when $P(\lambda)$ and all the block-entries of $B$, $C$, $D$ and $E$ are symmetric (resp. Hermitian)  follows easily from the following facts.
First, $M(\lambda;P_{k-1}^{k-1})$ is symmetric (resp. Hermitian) when $P(\lambda)$ is symmetric (resp. Hermitian).
Secondly, we have $B^\mathcal{B}=B^T$ ($B^\mathcal{B}=B^*$), $C^\mathcal{B}=C^T$ ($C^\mathcal{B}=C^*$), $D^\mathcal{B}=D^T$ ($D^\mathcal{B}=D^*$) and $E^\mathcal{B}=E^T$ ($E^\mathcal{B}=E^*$) when all the block-entries of $B$, $C$, $D$ and $E$ are symmetric (resp. Hermitian).
\end{proof}

\begin{example}
 Let $P(\lambda)$ be a matrix polynomial of odd degree $k$ and let $m$ be an even positive integer. Then, as we will show in Theorem \ref{thm:main_GFPR}, the $m$th pencil $D_m(\lambda, P)$  in the standard basis of the vector space $\mathbb{DL}(P)$, which is a GFPR with parameter $h=k-m$,  is permutationally block congruent to a pencil in $\langle \mathcal{O}_2^P\rangle.$ 
\end{example}

\subsection{The third fundamental block-structure}
The third fundamental family of block-symmetric pencils  is defined for matrix polynomials of even degree.
So, let  $P(\lambda)$ be a matrix polynomial of even degree $k$, and let $s:=(k-2)/2$.
First, we define the pencil
\begin{equation}\label{eq:E1}
   \mathcal{E}_1^P(\lambda) := \left[\begin{array}{c|c:c}
   M(\lambda;P_{k-1}) & \begin{matrix} 0 \\ A_0 \end{matrix} & K_s(\lambda)^T \\ \hdashline
   \begin{matrix} 0 & A_0 \end{matrix} &- \lambda A_0 & 0\\ \hline
  K_s(\lambda) & 0 &  0
   \end{array}\right],
\end{equation}
where $P_{k-1}(\lambda)$ is defined in \eqref{eq:Pu} and  $M(\lambda;P_{k-1})$ is defined in \eqref{eq:M(lambda,P)}.
The pencil $ \mathcal{E}_1^P(\lambda)$ is an extended $(s,n,s+1,n)$-block Kronecker pencil\footnote{It can be seen as an extended $(s+1,n,s,n)$-block Kronecker pencil as well.}, with the solid lines indicating one of its natural partitions. 

Note that the body of $\mathcal{E}_1^P(\lambda)$, regardless of the chosen partition (see \cite[Theorem 3.10]{canonical}), satisfies the AS condition for $P(\lambda)$.
Thus, $\mathcal{E}_1^P(\lambda)$ is a strong linearization of $P(\lambda)$, provided that $A_0$ is nonsingular.
 Furthermore, the pencil $\mathcal{E}_1^P(\lambda)$ is block-symmetric, and it is symmetric (resp. Hermitian) when $P(\lambda)$ is symmetric (resp. Hermitian).

\begin{example}
Let $P(\lambda)=\sum_{i=0}^6 A_i\lambda^i\in\mathbb{F}[\lambda]^{n\times n}$. Then,
$$\mathcal{E}_1^P(\lambda)= \left [\begin{array}{ccc|c:cc} \lambda A_6 + A_{5} &0 &0 &0 &-I_n & 0\\ 
0 & \lambda A_4 + A_3 & 0 & 0 & \lambda I_n & - I_n\\ 0 & 0 & \lambda A_2 +A_1 & A_0 & 0 & \lambda I_n \\ \hdashline 0 & 0 & A_0 & -\lambda A_0 & 0 & 0 \\
\hline
-I_n & \lambda I_n & 0 & 0 & 0 & 0\\
0 & -I_n & \lambda I_n & 0 & 0 &0  \end{array} \right ].$$
\end{example}

 Motivated by the block-structure of the pencil  $\mathcal{E}_1^P(\lambda)$, we introduce the third fundamental family of block-structures by applying the following block-congruence
\begin{equation}\label{gen-3}
\left[ \begin{array}{ccc} I_{(s+1)n} & 0 & B\\ 0 & I_{n} & C \\ 0 & 0 & D \end{array} \right ]  \mathcal{E}_1^P(\lambda) \left[ \begin{array}{ccc} I_{(s+1)n} & 0 & 0 \\ 0 & I_{n} & 0 \\  
B^\mathcal{B} & C^\mathcal{B} & D^\mathcal{B} \end{array}  \right ],
\end{equation}
where $B=[B_{ij}]$ is a $(s+1)\times s$  block-matrix, $C=[C_{ij}]$ is an $1\times s$ block-matrix  and $D=[D_{ij}]$ is an $s\times s$ block-matrix, with $n\times n$ block-entries $B_{ij}$, $C_{ij}$ and $D_{ij}$, respectively.

\begin{definition}\label{def:third_family}
Let $P(\lambda)=\sum_{i=0}^k A_i\lambda^i\in\mathbb{F}[\lambda]^{n\times n}$ be a matrix polynomial of even degree $k$, and let $s=(k-2)/2$.
{\rm The third fundamental block-structure family}, denoted by $\langle \mathcal{E}_1^P \rangle $, is the set of pencils of the form
\begin{equation}\label{eq:third_family}
\left[\begin{array}{c|c:c}
   M(\lambda;P_{k-1})+BK_s(\lambda)+K_s(\lambda)^TB^\mathcal{B} & \begin{bmatrix} 0 \\ A_0 \end{bmatrix} + K_s(\lambda)^T C^\mathcal{B} & K_s(\lambda)^T D^\mathcal{B} \\ \hdashline
   \begin{bmatrix} 0 & A_0 \end{bmatrix} + CK_s(\lambda) & -\lambda A_0 & 0\\ \hline
  DK_s(\lambda)& 0 &  0
   \end{array}\right],
\end{equation}
where $P_{k-1}(\lambda)$ is defined in \eqref{eq:Pu} and  $M(\lambda;P_{k-1})$ is defined in \eqref{eq:M(lambda,P)}, for some arbitrary $(s+1)\times s$ block-matrix $B=[B_{ij}]$, $1\times s$ block-matrix $C=[C_{ij}]$ and  $s\times s$ block-matrix $D=[D_{ij}]$, with $n\times n$ block-entries $B_{ij}$, $C_{ij}$ and $D_{ij}$, respectively.
\end{definition}

Note that the pencils in $\langle \mathcal{E}_1^P \rangle$ are block minimal bases pencils  if  $A_0$, $D$ and $D^{\mathcal{B}}$ are nonsingunlar.

The following theorem gives sufficient conditions for the pencils in the family $\langle \mathcal{E}_1^P \rangle$ to be strong linearizations of the even-degree matrix polynomial $P(\lambda)$.
\begin{theorem}\label{thm:third family}
 Let $P(\lambda)$ be a matrix polynomial of even degree $k$, let $s=(k-2)/2$, and consider a pencil $\mathcal{L}(\lambda)\in \langle \mathcal{E}_1^P \rangle$ of the form (\ref{eq:third_family}).
If $A_0$, $D$ and $D^\mathcal{B}$ are nonsingular, then $\mathcal{L}(\lambda)$ is a block-symmetric strong linearization of $P(\lambda)$.
Moreover, if $P(\lambda)$ and all the block-entries $B_{ij}$, $C_{ij}$ and $D_{ij}$ are symmetric (resp. Hermitian), then $\mathcal{L}(\lambda)$ is symmetric (resp. Hermitian).
\end{theorem}
\begin{proof}
If $A_0$ is nonsingular, the pencil  $\mathcal{E}_1^P(\lambda)$ is a strong linearization of $P(\lambda)$.
Additionally, if $D$ and $D^\mathcal{B}$ are nonsingular, the pencil $\mathcal{L}(\lambda)$ is strictly equivalent to  $\mathcal{E}_1^P(\lambda)$ (see \eqref{gen-3}).
Thus, if $A_0$, $D$ and $D^\mathcal{B}$ are nonsingular, $\mathcal{L}(\lambda)$ is a strong linearization of $P(\lambda)$. 
The fact that $\mathcal{L}(\lambda)$ is block-symmetric follows readily from \eqref{TvsB} and the fact that $ M(\lambda;P_{k-1})$ is block-symmetric. 
Finally, the fact that $\mathcal{L}(\lambda)$ is symmetric (resp. Hermitian) when $P(\lambda)$ and all the block-entries $B_{ij}$, $C_{ij}$ and $D_{ij}$ are symmetric (resp. Hermitian)  follows easily from the following facts.
First, $M(\lambda;P_{k-1})$ is symmetric (resp. Hermitian) when $P(\lambda)$ is symmetric (resp. Hermitian).
Secondly, we have $B^\mathcal{B}=B^T$ ($B^\mathcal{B}=B^*$), $C^\mathcal{B}=C^T$ ($C^\mathcal{B}=C^*$) and $D^\mathcal{B}=D^T$ ($D^\mathcal{B}=D^*$) when all the block-entries $B_{ij}$, $C_{ij}$ and $D_{ij}$ are symmetric (resp. Hermitian). 
\end{proof}

\begin{example}
Let $P(\lambda)$ be a matrix polynomial of even degree $k$ and let $m$ be an odd positive integer. Then, as we will show in Theorem \ref{thm:main_GFPR}, the $m$th pencil $D_m(\lambda, P)$  in the standard basis of the vector space $\mathbb{DL}(P)$, which is a block-symmetric GFPR with parameter $h=k-m$,  is permutationally block congruent to a pencil in  $\langle \mathcal{E}_1^P\rangle.$  This holds true, in  particular, for $D_1(\lambda, P)$. 
\end{example}

\subsection{The fourth fundamental block-structure}
The fourth fundamental block-structure family is also associated with even-degree matrix polynomials.
So, let  $P(\lambda)$ be a matrix polynomial of even degree $k$, and let $s:=(k-2)/2$.
First, we define the pencil
\begin{equation}\label{eq:E2}
   \mathcal{E}_2^P(\lambda) :=
    \left[\begin{array}{c:c|c}
   -A_k & \begin{matrix} \lambda A_k & 0 \end{matrix} & 0 \\ \hdashline
   \begin{matrix} \lambda A_k \\ 0 \end{matrix} & M(\lambda;P^{k-1}) &   K_s(\lambda)^T \\ \hline
   0 &   K_s(\lambda) & 0
     \end{array}\right],
\end{equation} 
where  $P^{k-1}(\lambda)$ is defined in \eqref{eq:Pd}  and $M(\lambda;P^{k-1})$ is defined in \eqref{eq:M(lambda,P)}.  By applying a block column-permutation $\Pi_4$ to  $\mathcal{E}_2^P(\lambda)$, we obtain the pencil 
\[
   \mathcal{E}_2^P(\lambda)\Pi_4:=
    \left[\begin{array}{c|c:c}
    \begin{matrix} \lambda A_k & 0 \end{matrix} & -A_k  & 0 \\ \hdashline
   M(\lambda;P^{k-1}) & \begin{matrix} \lambda A_k \\ 0 \end{matrix} & K_s(\lambda)^T \\ \hline
     K_s(\lambda) & 0 & 0
     \end{array}\right],
\]
which is an extended $(s,n,s+1,n)$-block Kronecker pencil.
Notice that the body of the pencil $\mathcal{E}_2^P(\lambda)\Pi_4$ satisfies the AS condition for $P(\lambda)$.
Hence, $\mathcal{E}_2^P(\lambda)\Pi_4$ and, thus, $\mathcal{E}_2^P(\lambda)$, are strong linearizations of $P(\lambda)$ if $A_k$ is nonsingular. 
Moreover, the pencil $\mathcal{E}_2^P(\lambda)$ is block-symmetric, and it is symmetric (resp. Hermitian) provided that $P(\lambda)$ is symmetric (resp. Hermitian).

\begin{example}
Let $P(\lambda)=\sum_{i=0}^6  A_i\lambda^i\in\mathbb{F}[\lambda]^{n\times n}$. Then,
$$\mathcal{E}_2^P(\lambda)= \left [\begin{array}{cccc|cc} - A_6  & \lambda A_6 &0 &0 & 0 & 0\\ 
\lambda A_6 & \lambda A_5 + A_4 & 0 & 0 & - I_n & 0\\ 0 & 0 & \lambda A_3 +A_2 & 0 & \lambda I_n & - I_n \\ 0 & 0 & 0 & \lambda A_1+A_0 & 0 & \lambda I_n \\
\hline
0 &-I_n & \lambda I_n & 0 & 0 & 0 \\
0 & 0 & -I_n & \lambda I_n & 0 & 0  \end{array} \right ].$$
By permuting the first block-column with the block-columns in positions 2-4, we get
$$\mathcal{E}_2^P(\lambda)\Pi_4= \left [\begin{array}{ccc|ccc}  \lambda A_6 & 0 &0 &-A_6 &  0 & 0\\ 
  \lambda A_5 + A_4 & 0 & 0 &\lambda A_6 &  - I_n & 0\\  0 & \lambda A_3 +A_2 & 0 &0 & \lambda I_n & - I_n \\  0 & 0 & \lambda A_1+A_0 &0 & 0 & \lambda I_n \\
\hline
-I_n & \lambda I_n & 0 & 0 & 0 &0 \\
 0 & -I_n & \lambda I_n & 0 & 0 &0 \end{array} \right ],$$
which is clearly an extended block Kronecker pencil.
\end{example}

Inspired by the  block-structure of  $\mathcal{E}_2^P(\lambda)\Pi_4$, we consider extended block Kronecker pencils of the form 
\begin{equation}\label{eq:equivalence for E2}
\left[ \begin{array}{ccc} I_n & 0 & C \\ 0 & I_{(s+1)n} & 
B\\ 0 & 0 & D \end{array}\right] \left[\begin{array}{c|c:c}
    \begin{matrix} \lambda A_k & 0 \end{matrix} & -A_k  & 0 \\ \hdashline
   M(\lambda;P^{k-1}) & \begin{matrix} \lambda A_k \\ 0 \end{matrix} & K_s(\lambda)^T \\ \hline
     K_s(\lambda) & 0 & 0
     \end{array}\right] \left[ \begin{array}{ccc} I_{(s+1)n} & 0 & 0\\ 0 & I_n & 0 \\ B^{\mathcal{B}} & C^{\mathcal{B}} &  D^\mathcal{B} \end{array}\right]
     \end{equation}
for arbitrary matrices $C \in \mathbb{F}^{n \times sn}$, $B\in \mathbb{F}^{(s+1)n \times  sn}$  and  $D\in \mathbb{F}^{sn \times sn}$,
or, equivalently,
\[
   \left[\begin{array}{c|c:c}
    \begin{bmatrix} \lambda A_k & 0 \end{bmatrix} + CK_s(\lambda) & -A_k & 0 \\ \hdashline
    M(\lambda;P^{k-1}) + BK_s(\lambda)+K_s(\lambda)^TB^\mathcal{B}&  \begin{bmatrix} \lambda A_k \\ 0 \end{bmatrix} + K_s(\lambda)^TC^\mathcal{B} &  K_s(\lambda)^T D^\mathcal{B} \\ \hline
   DK_s(\lambda) & 0 & 0
     \end{array}\right].
\]

Reversing the block-permutation we did originally, we obtain the block-structure defining the fourth family of block-structures.
\begin{definition}\label{def:fourth_family}
Let $P(\lambda)=\sum_{i=0}^k A_i\lambda^i\in\mathbb{F}[\lambda]^{n\times n}$ be a matrix polynomial of even degree $k$, and let $s=(k-2)/2$.
{\rm The fourth fundamental block-structure family}, denoted by $\langle \mathcal{E}_2^P \rangle $, is the set of pencils of the form
\begin{equation}\label{eq:fourth_family} 
   \left[\begin{array}{c:c|c}
   -A_k & \begin{bmatrix} \lambda A_k & 0 \end{bmatrix} + CK_s(\lambda) & 0 \\ \hdashline
   \begin{bmatrix} \lambda A_k \\ 0 \end{bmatrix} + K_s(\lambda)^TC^\mathcal{B} & M(\lambda;P^{k-1}) + BK_s(\lambda)+K_s(\lambda)^TB^\mathcal{B}&   K_s(\lambda)^T  D^\mathcal{B} \\ \hline
   0 &   DK_s(\lambda) & 0
     \end{array}\right],
\end{equation}
where  $P^{k-1}(\lambda)$ is defined in \eqref{eq:Pd}  and $M(\lambda;P^{k-1})$ is defined in \eqref{eq:M(lambda,P)}, 
for an arbitrary $(s+1)\times s$ block-matrix $B=[B_{ij}]$, $1\times n$ block-matrix $C=[C_{ij}]$, and  $s\times s$ block-matrix $D=[D_{ij}]$, with $n\times n$ block entries $B_{ij}$, $C_{ij}$ and $D_{ij}$, respectively.

\end{definition}

Note that all pencils in $\langle \mathcal{E}_2^P \rangle$ are block minimal bases pencils if $A_k$,  $D$, and $D^{\mathcal{B}}$ are nonsingular.

The following theorem gives necessary and sufficient conditions for pencils in the family $\langle \mathcal{E}_2^P \rangle$ to be strong linearizations of the even-degree matrix polynomial $P(\lambda)$. 
\begin{theorem}
 Let $P(\lambda)$ be a matrix polynomial of even degree $k$, let $s=(k-2)/2$, consider a pencil $\mathcal{L}(\lambda)\in \langle \mathcal{E}_2^P \rangle $ of the form (\ref{eq:fourth_family}). 
If $A_k $, $D$ and $D^\mathcal{B}$ are nonsingular, then $\mathcal{L}(\lambda)$ is a block-symmetric strong linearization of $P(\lambda)$.
Moreover, if $P(\lambda)$ and all the block-entries $B_{ij}$, $C_{ij}$ and $D_{ij}$ are symmetric (resp. Hermitian), then $\mathcal{L}(\lambda)$ is symmetric (resp. Hermitian).
\end{theorem}
\begin{proof}
If $A_k$ is nonsingular, the pencil $\mathcal{E}_2^P(\lambda)$ is a strong linearization of $P(\lambda)$.
In addition, notice from \eqref{eq:equivalence for E2} that if $D$ and $D^\mathcal{B}$ are nonsingular, the pencil $\mathcal{L}(\lambda)$ is strictly equivalent to $\mathcal{E}_2^P(\lambda)$.
Thus, if $A_k$, $D$ and $D^\mathcal{B}$ are nonsingular, then $\mathcal{L}(\lambda)$ is a strong linearization of $P(\lambda)$. 
The fact that $\mathcal{L}(\lambda)$ is block-symmetric follows easily from \eqref{BKsB} and the fact that $M(\lambda;P^{k-1})$ is block-symmetric. 
Finally, notice the following two facts. 
First, if $P(\lambda)$ is symmetric (resp. Hermitian) so is $M(\lambda;P^{k-1})$.
Secondly, when all the block-entries of $B$, $C$ and $D$ are symmetric (resp. Hermitian.), we have $B^\mathcal{B} = B^T$ (resp. $B^\mathcal{B} = B^*$), $C^\mathcal{B} = C^T$ (resp. $C^\mathcal{B} = C^*$) and $D^\mathcal{B} = D^T$ (resp. $D^\mathcal{B} = D^*$).
Hence, if $P(\lambda)$ and all the block-entries $B_{ij}$, $C_{ij}$ and $D_{ij}$ are symmetric (resp. Hermitian), then $\mathcal{L}(\lambda)$ is symmetric (resp. Hermitian).
\end{proof}


\begin{example}
 Let $P(\lambda)$ be a matrix polynomial of even degree $k$ and let $m$ be an even positive integer. Then, as we will show in Theorem \ref{thm:main_GFPR}, the $m$th pencil $D_m(\lambda, P)$  in the standard basis of the vector space $\mathbb{DL}(P)$, which is a block-symmetric GFPR with parameter $h=k-m$,  is permutationally block congruent to a pencil in  $\langle \mathcal{E}_2^P\rangle.$ This holds true, in particular, for $D_k(\lambda, P)$.
 \end{example}


 \section{ Block-symmetric generalized Fiedler pencils and block-symmetric generalized Fiedler pencils with repetition}\label{sec:GFPR-def}

We introduce in this section the block-symmetric generalized Fiedler pencils and the family of block-symmetric generalized Fiedler pencils with repetition.
We start with some concepts and basic results needed for those definitions. 

\subsection{The index tuple notation and matrix assignments}\label{sec:index-tuples}

We start by introducing the fundamental definition of an index tuple and some related notions.

 \begin{definition}{\rm \cite[Definition 3.1]{GFPR}}
We call an \emph{index tuple}  a finite ordered sequence of integer numbers.
Each of these integers is called an \emph{index} of the tuple.
The number of indices in an index tuple $\mathbf{t}$ is called its \emph{length} and is denoted by $|\mathbf{t}|$. 
For integers $a$ and $b$, we call the tuple $(a:b)$ a \emph{string}.

\end{definition}

We will use the following notation for some important basic operations with tuples.
If $\mathbf{t}=(t_1, \ldots, t_r)$ is an index tuple, we denote  $-\mathbf{t}:=(-t_1,\hdots,-t_r)$, and, when $a$ is an integer, we denote $a+\mathbf{t}:=(a+t_1, a+t_2, \ldots, a+t_r)$. 
We call the \emph{reversal index tuple} of $\mathbf{t}$  the index tuple $\rev(\mathbf{t}):=(t_r, \hdots, t_2,  t_1).$
Additionally, given index tuples $\mathbf{t}_1,\hdots,\mathbf{t}_s$,  we denote by $(\mathbf{t}_1,\hdots,\mathbf{t}_s)$ the index tuple obtained by concatenating the indices in the index tuples $\mathbf{t}_1,\hdots,\mathbf{t}_s$ in the indicated order.

An important property of  index tuples used to define the block-symmetric GFPR is the so-called Successor Infix Property, which we introduce in the following definition.
\begin{definition}
\label{defSIP}{\rm \cite[Definition 7]{ant-vol11}} Let ${\mathbf{t}}=(i_{1},i_{2},\hdots,i_{r})$
be an index tuple of either all nonnegative integers or all negative integers.
 Then, ${\mathbf{t}}$ is said to satisfy the \emph{Successor
Infix Property (SIP)} if for every pair of indices $i_{a},i_{b}\in{\mathbf{t}%
}$, with $1\leq a<b\leq r$, satisfying $i_{a}=i_{b}$, there exists at least
one index $i_{c}=i_{a}+1$ with $a<c<b$.
\end{definition}
\begin{remark}\label{remark:SIP}
We note the following basic properties of tuples satisfying the SIP.
Any subtuple of consecutive indices of a tuple satisfying the SIP also satisfies the SIP. 
The reversal of any tuple satisfying the SIP also satisfies
the SIP.
If the tuple $\mathbf{t}$ has no repeated indices, then $\mathbf{t}$ satisfies the SIP.
\end{remark}

The following  definitions are motivated by the construction of block-symmetric GFPR in Section \ref{sec:GFPR-def}.
For more details, we refer the reader to \cite[Section 4]{Her-GFPR} and \cite{FPR3}.



\begin{definition}{\rm \cite[Definition 4.3]{Her-GFPR}}
Let $h$ be a nonnegative integer, and let $p=0$ if $h$ is even, and $p=1$ is $h$ is odd.
Then, we call the index tuple
\[
\mathbf{w}_h:= (h-1:h,h-3:h-2,\hdots,p+1:p+2,0:p)
\]
the \emph{admissible tuple associated with the integer $h\geq 0$}.
\end{definition}

Notice that the tuple $\mathbf{w}_h$ is a permutation of the tuple $(0:h)$.

\begin{definition}{\rm \cite[Definition 4.3]{Her-GFPR}}
Let $h$ be a nonnegative integer, and let $\mathbf{w}_h$ be the admissible tuple associated with $h$.
Then, the symmetric complement of $\mathbf{w}_h$ is the tuple 
\begin{itemize}
\item $\mathbf{c}_h:=(h-1,h-3,\hdots,2,0)$ if $h$ is odd;
\item $\mathbf{c}_h:=(h-1,h-3,\hdots,1)$ if $h>0$ is even;
\item $\mathbf{c}_h:=\emptyset$ is $h=0$.
\end{itemize}
\end{definition}

\begin{lemma}\label{wcSIP} {\rm \cite[Lemma 3.11]{FPR1} }
Let $h$ be a nonnegative integer, let $\mathbf{w}_h$ be the admissible tuple associated with $h$, and let $\mathbf{c}_h$ be the symmetric complement of $\mathbf{w}_h$. Then, the index tuple $(\mathbf{w}_h, \mathbf{c}_h)$ satisfies the SIP.
\end{lemma}


The matrix coefficients of  the block-symmetric GFPR (and that we review in this section) are products of  elementary block-matrices, whose definition we recall next.
\begin{definition}{\rm \cite{GFPR}}
Let $k\geq 2$ be an integer and let $B$ be an arbitrary $n\times n$ matrix. 
We call \emph{elementary matrices} the  following  $k\times k$ block-matrices partitioned into blocks of size $n\times n$:
\begin{equation*}
\,M_{0}(B):=\left[
\begin{tabular}
[c]{c|c}%
$I_{(k-1)n}$ & $0$\\\hline
$0$ & $B$%
\end{tabular}
\ \right]  ,\quad M_{-k}(B):=\left[
\begin{tabular}
[c]{c|c}%
$B$ & $0$\\\hline
$0$ & $I_{(k-1)n}$%
\end{tabular}
\ \right]  , \label{m0-mk}%
\end{equation*}

\begin{equation}
M_{i}(B):=\left[
\begin{tabular}
[c]{c|cc|c}%
$I_{(k-i-1)n}$ & $0$ & $0$ & $0$\\\hline
$0$ & $B$ & $I_{n}$ & $0$\\
$0$ & $I_{n}$ & $0$ & $0$\\\hline
$0$ & $0$ & $0$ & $I_{(i-1)n}$%
\end{tabular}
\ \right]  ,\ \ \ i=1:k-1, \label{mis}%
\end{equation}
\[
M_{-i}(B):=\left[
\begin{tabular}
[c]{c|cc|c}%
$I_{(k-i-1)n}$ & $0$ & $0$ & $0$\\\hline
$0$ & $0$ & $I_{n}$ & $0$\\
$0$ & $I_{n}$ & $B$ & $0$\\\hline
$0$ & $0$ & $0$ & $I_{(i-1)n}$%
\end{tabular}\ \ \ \right] \ \ \ i=1:k-1,
\]
and
\[
M_{-0}(B):=M_0(B)^{-1} \quad \mbox{and} \quad M_{k}(B):=M_{-k}(B)^{-1}.
\]
assuming that $B$ is nonsingular.
\end{definition}
Notice that   the notation $-0$ does not have the usual meaning, that is, in this case $-0 \neq 0$.

\begin{remark}
Notice that, for $i=1:k-1$, the elementary matrices $M_i(B)$ and $M_{-i}(B)$ are nonsingular for any $B$.
Furthermore, $(M_i(B))^{-1}=M_{-i}(-B)$. 
On the other hand, the matrices $M_{0}(B)$ and $M_{-k}(B)$ are nonsingular if and only if $B$ is nonsingular.
\end{remark}

\begin{definition}{\rm \cite[Definition 4.6]{Her-GFPR}}
Let $\mathbf{t}=(i_1,i_2,\hdots,i_r)$ be an index tuple with indices contained in $\{-k:k-1\}$ and let $\mathcal{Z}=(Z_1,\hdots,Z_r)$ be a list of $r$ arbitrary $n\times n$ matrices.
We define
\[
M_\mathbf{t}(\mathcal{Z}):=M_{i_1}(Z_1)M_{i_2}(Z_2)\cdots M_{i_r}(Z_r),
\]
and say that $\mathcal{Z}$ is a \emph{matrix assignment for $\mathbf{t}$}. 
 If $\mathbf{t}$ (and therefore $\mathcal{Z}$) is empty, then $M_\mathbf{t}(\mathcal{Z}) := I_{kn}$.
 The matrix assignment $\mathcal{Z}$ for $\mathbf{t}$ is said to be \emph{nonsingular} if the matrices assigned to the positions in $\mathbf{t}$ occupied by the $ 0$ and $- k$ indices are nonsingular.
 If the matrices in $\mathcal{Z}$ are symmetric (resp. Hermitian), then $\mathcal{Z}$ is said to be a \emph{symmetric (resp. Hermitian) matrix assignment} for $\mathbf{t}$.
\end{definition}

Given an ordered list of $n\times n$ arbitrary matrices $\mathcal{Z}=(Z_1,\hdots,Z_r)$, we denote by $\rev(\mathcal{Z})$ the list of matrices $(Z_r,\hdots,Z_1)$.

Given a  matrix polynomial $P(\lambda)=\sum_{i=0}^k A_i \lambda^i\in\mathbb{F}[\lambda]^{n\times n}$, we will use the following abbreviated notation:
\[
M_i^P := M_i(-A_i), \quad i=0:k-1,
\]
and
\[
M_{-i}^P:=M_{-i}(A_i), \quad i=1:k.
\]
When the polynomial $P(\lambda)$ is understood from the context, we simply write $M_i$ and $M_{-i}$, instead of $M_i^P$ and $M_{-i}^P$ to simplify the notation.


\subsection{Block-symmetric GFP}

Here we recall the block-symmetric strong linearizations of a matrix polynomial $P(\lambda)$ in the family of generalized Fiedler pencils (GFP).
The following GFP was introduced in \cite[Theorem 3.1]{GFP}:
 \begin{equation*}
\mathcal{T}_P(\lambda):=\lambda M^P_{-1,-3,\ldots,-k+2,-k}-M^P_{0,2,\ldots,k-1},
\end{equation*}
if $k$ is odd, and 
\[
\mathcal{T}_P(\lambda):=\lambda M^P_{-1,-3,\ldots,-k+1}-M^P_{0,2,\ldots,k},
\]
if $k$ is even and the leading coefficient $A_k$ is nonsingular. 
The pencil $\mathcal{T}_P(\lambda)$ is explicitly given by
\begin{equation}\label{GFPT}
\mathcal{T}_P(\lambda)= \left[ \begin{array}{ccccccc} \lambda A_k + A_{k-1} & -I_n &&&&& \\
-I_n & 0 & \lambda I_n &&&&\\  & \lambda I_n & \lambda A_{k-2} + A_{k-3} & -I_n &&&\\ & & -I_n &\ddots &&& \\
&&&& -I_n & & \\
&&& -I_n & 0 & \lambda I_n &\\
&&&& \lambda I_n & \lambda A_1+ A_0
\end{array} \right],
\end{equation}
when $k$ is odd, and by

\[
\mathcal{T}_P(\lambda)= \left[ \begin{array}{ccccccc} -A_k^{-1}  &  \lambda I_n &&&&& \\
\lambda I_n & \lambda A_{k-1}+A_{k-2}  & - I_n &&&&\\  & - I_n & 0 &  \lambda I_n &&&\\ & & \lambda I_n &\ddots &&& \\
&&&& -I_n & & \\
&&& -I_n & 0 & \lambda I_n &\\
&&&& \lambda I_n & \lambda A_1+ A_0
\end{array} \right],
\]
when $k$ is even and $A_k$ is nonsingular.  
We note  that this pencil is not a companion form since one of its matrix coefficients contains a block equal to $A_k^{-1}$. 
Notice that $\mathcal{T}_P(\lambda)$ is block-symmetric, regardless of the parity of $k$. 
Some small variations of these pencils can be found in \cite{4m-alt}. 

\subsection{Block-symmetric GFPR}

Next, we recall a subfamily of GFPR comprised  of block-symmetric pencils.
\begin{definition}\label{def:sym-GFPR}{\rm \cite{Her-GFPR}}
Let $P(\lambda)=\sum_{i=0}^kA_i\lambda^i\in\mathbb{F}[\lambda]^{n\times n}$ of degree $k$, and let $h$ be an integer such that $0\leq h<k$.
Let $\mathbf{w}_h$ and $k+\mathbf{v}_h$ be the admissible tuples associated with $h$ and $k-h-1$, respectively, and let $\mathbf{c}_h$ and $\mathbf{c}_{k-h-1}$ be the symmetric complements of $\mathbf{w}_h$ and $k+\mathbf{v}_h$, respectively.
Let $\mathbf{t}_w$ and $k+\mathbf{t}_v$ be index tuples with indices from $\{0:h-1\}$ and $\{0:k-h-2\}$, respectively, such that $(\mathbf{t}_w,\mathbf{w}_h,\mathbf{c}_h,\rev(\mathbf{t}_w))$ and $(\mathbf{t}_v,\mathbf{v}_h,-k+\mathbf{c}_{k-h-1},\rev(\mathbf{t}_v))$ satisfy the SIP.
Let $\mathcal{Z}_w$ and $\mathcal{Z}_v$ be matrix assignments for $\mathbf{t}_w$ and $\mathbf{t}_v$, respectively.
Then, the pencil 
\begin{equation}\label{symGFPR}
M_{\mathbf{t}_w,\mathbf{t}_v}(\mathcal{Z}_w,\mathcal{Z}_v)(\lambda M^P_{\mathbf{v}_h}-M^P_{\mathbf{w}_h})M^P_{-k+\mathbf{c}_{k-h-1},\mathbf{c}_h}M_{\rev(\mathbf{t}_w),\rev(\mathbf{t}_v)}(\rev(\mathcal{Z}_w),\rev(\mathcal{Z}_v))
\end{equation}
is a  \emph{block-symmetric generalized Fiedler pencil with repetition (block-symmetric GFPR)} and we denote it  by $L_P(h,\mathbf{t}_w,\mathbf{t}_v,\mathcal{Z}_w,\mathcal{Z}_v)$.
If the matrix assignments $\mathcal{Z}_w$ and $\mathcal{Z}_v$ are chosen so that $M_{\mathbf{t}_w,\mathbf{t}_v}(\mathcal{Z}_w,\mathcal{Z}_v)=M_{\mathbf{t}_w,\mathbf{t}_v}^P$, then the block-symmetric GFPR $L_P(h,\mathbf{t}_w,\mathbf{t}_v,\mathcal{Z}_w,\mathcal{Z}_v)$ is a  \emph{block-symmetric Fiedler pencil with repetition (block-symmetric FPR)}, which we denote by $L_P(h,\mathbf{t}_w,\mathbf{t}_v)$.
\end{definition}

Theorem \ref{thm:GFPR linearization} establishes when a block-symmetric GFPR associated with a  matrix polynomial $P(\lambda)$ is a strong linearization of $P(\lambda)$.
\begin{theorem}{\rm \cite[Theorem 4.9]{Her-GFPR}}\label{thm:GFPR linearization}
Let $P(\lambda)=\sum_{i=0}^kA_i\lambda^i\in\mathbb{F}[\lambda]^{n\times n}$ be a matrix polynomial, and let $\mathcal{L}(\lambda)$ be the block-symmetric GFPR  defined in \eqref{symGFPR}.
If the following three conditions hold
\begin{itemize}
\item[\rm (i)] $\mathcal{Z}_w$ and $\mathcal{Z}_v$ are nonsingular matrix assignments for $\mathbf{t}_w$ and $\mathbf{t}_v$, respectively,
\item[\rm (ii)] $A_0$ is nonsingular if $h$ is odd, and
\item[\rm (iii)] $A_k$ is nonsingular if $k-h$ is even,
\end{itemize}
then, the pencil $\mathcal{L}(\lambda)$ is a strong linearization of $P(\lambda)$.
\end{theorem}

Theorem \ref{thm:GFPR symmetric}  gives sufficient conditions for a block-symmetric GFPR associated with a symmetric (resp. Hermitian) matrix polynomial $P(\lambda)$ to be symmetric (resp. Hermitian).
\begin{theorem}{\rm \cite{Her-GFPR}}\label{thm:GFPR symmetric}
Let $P(\lambda)$ be a symmetric (resp. Hermitian) matrix polynomial, and let $\mathcal{L}(\lambda)$ be the block-symmetric GFPR  defined in \eqref{symGFPR}.
If the matrix assignments $\mathcal{Z}_w$ and $\mathcal{Z}_v$ are symmetric (resp. Hermitian), then $\mathcal{L}(\lambda)$ is symmetric (resp. Hermitian).
\end{theorem}

\section{Block-symmetric GFP and GFPR as block-symmetric block minimal bases pencils}\label{sec:main-GFP}

 In this section, we start by showing that the block-symmetric pencil $\mathcal{T}_P(\lambda)$ associated with an odd-degree matrix polynomial $P(\lambda)$ is permutationally block-congruent to the pencil $\mathcal{O}_1^P(\lambda)$. 
The  case when $P(\lambda)$ has even degree is not considered, since $\mathcal{T}_P(\lambda)$ is not a companion  form, that
is, if the matrix coefficients of $\mathcal{T}_P(\lambda)$ are seen as block matrices, one of them contains a block  that is not of the form $0$, $\pm I_n$ or $\pm A_i$.
In fact, since  one of the matrix coefficients of  $\mathcal{T}_P(\lambda)$ contains a block-entry that is the inverse of $A_k$, the interest of this pencil in applications is very limited.

\begin{theorem}\label{th;TP}
Let $P(\lambda)=\sum_{i=0}^k A_i\lambda^i\in\mathbb{F}[\lambda]^{n\times n}$ be an odd-degree matrix polynomial.
Let $\mathcal{T}_P(\lambda)$ be the block-symmetric GFP associated with $P(\lambda)$ defined in \eqref{GFPT}. 
Let $\mathbf{c}$ be the permutation of $\{1:k\}$ given by $(1, 3, 5, \ldots, k, 2, 4, \ldots, k-1)$.
Then,
\[
(\Pi_\mathbf{c}^n)^\mathcal{B}\mathcal{T}_P(\lambda)\Pi_\mathbf{c}^n = \mathcal{O}_1^P(\lambda).
\]
In other words, modulo block-permutations, the block-symmetric GFP $\mathcal{T}_P(\lambda)$ belongs to the family $\langle
\mathcal{O}_1^P \rangle$.
\end{theorem}
\begin{proof}
Using the explicit expression for $\mathcal{T}_P(\lambda)$ presented in Section \ref{sec:GFPR-def}, the result is easily checked by performing the matrix product $(\Pi_\mathbf{c}^n)^\mathcal{B}\mathcal{T}_P(\lambda)\Pi_\mathbf{c}^n$. 
\end{proof}

Now we give the main result for  block-symmetric GFPR associated with a matrix polynomial $P(\lambda)$, that  is, we state that, up to permutations of block-rows and block-columns, every block-symmetric GFPR is  in one of the four block-symmetric families  introduced in Section \ref{sec:four-families}. 
This result is stated in Theorem \ref{thm:main_GFPR}.
Its proof is included in the Appendix.
\begin{theorem}\label{thm:main_GFPR}
Let $P(\lambda)=\sum_{i=0}^k A_i\lambda^i\in\mathbb{F}[\lambda]^{n\times n}$ and let $s=(k-1)/2$ if $k$ is odd, or $s=(k-2)/2$ is $k$ is even.
Let $L_P(h,\mathbf{t}_w,\mathbf{t}_v,\mathcal{Z}_w,\mathcal{Z}_v)$ be the block-symmetric GFPR associated with $P(\lambda)$ given in (\ref{symGFPR}).
Then, there exists a permutation $\mathbf{c}$ of $\{1:k\}$  such that
\[
(\Pi_\mathbf{c}^n)^{\mathcal{B}} L_P(h,\mathbf{t}_w,\mathbf{t}_v,\mathcal{Z}_w,\mathcal{Z}_v)\Pi_\mathbf{c}^n\in\left\{ \begin{array}{l} 
\langle \mathcal{O}_1^P \rangle \mbox{ if $k$ is odd and $h$ is even},\\
\langle \mathcal{O}_2^P \rangle \mbox{ if $k$ and $h$ are odd},\\
\langle \mathcal{E}_1^P \rangle \mbox{ if $k$ is even and $h$ is odd,}\\
\langle \mathcal{E}_2^P \rangle \mbox{ if $k$ and $h$ are even}.
\end{array} \right.
\]
Furthermore, if the following conditions hold
\begin{itemize}
\item[\rm (i)] $\mathcal{Z}_w$ and $\mathcal{Z}_v$ are nonsingular matrix assignments for $\mathbf{t}_w$ and $\mathbf{t}_v$, respectively,
\item[\rm (ii)] $A_0$ is nonsingular if $h$ is odd, and
\item[\rm (iii)] $A_k$ is nonsingular if $k-h$ is even,
\end{itemize}
then $(\Pi_\mathbf{c}^n)^{\mathcal{B}} L_P(h,\mathbf{t}_w,\mathbf{t}_v,\mathcal{Z}_w,\mathcal{Z}_v)\Pi_\mathbf{c}^n$ is a strong linearization of $P(\lambda)$ and the following statements hold:
\begin{itemize}
\item[\rm (a)] If $(\Pi_\mathbf{c}^n)^{\mathcal{B}} L_P(h,\mathbf{t}_w,\mathbf{t}_v,\mathcal{Z}_w,\mathcal{Z}_v)\Pi_\mathbf{c}^n\in\langle \mathcal{O}_1^P \rangle$ is as in \eqref{eq:first_family}, then $B$ and $B^\mathcal{B}$ are nonsingular. 
\item[\rm (b)] If $(\Pi_\mathbf{c}^n)^{\mathcal{B}} L_P(h,\mathbf{t}_w,\mathbf{t}_v,\mathcal{Z}_w,\mathcal{Z}_v)\Pi_\mathbf{c}^n\in\langle \mathcal{O}_2^P \rangle$ is as in \eqref{eq:second_family}, then $E$ and $E^\mathcal{B}$ are nonsingular.
\item[\rm (c)] If $(\Pi_\mathbf{c}^n)^{\mathcal{B}} L_P(h,\mathbf{t}_w,\mathbf{t}_v,\mathcal{Z}_w,\mathcal{Z}_v)\Pi_\mathbf{c}^n\in\langle \mathcal{E}_1	^P \rangle$ is as in \eqref{eq:third_family}, then $D$ and $D^\mathcal{B}$ are nonsingular.
\item[\rm (d)] If $(\Pi_\mathbf{c}^n)^{\mathcal{B}} L_P(h,\mathbf{t}_w,\mathbf{t}_v,\mathcal{Z}_w,\mathcal{Z}_v)\Pi_\mathbf{c}^n\in\langle \mathcal{E}_2^P \rangle$ is as in \eqref{eq:fourth_family}, then $D$ and $D^\mathcal{B}$ are nonsingular.
\end{itemize}
\end{theorem}

The following example illustrates the result for a particular block-symmetric GFPR associated with an odd degree matrix polynomial. 
\begin{example}\label{FPR7}
Let $P(\lambda)=\sum_{i=0}^7 A_i \lambda^i$ be an $n\times n$ matrix polynomial of degree $7$. Consider the block-symmetric GFPR 
\begin{center}
   \begin{equation*}
   L_P(\lambda):=L_P(k-1, \emptyset,\emptyset)=  \begin{bmatrix}
    \lambda A_7 + A_6 & A_5 & -I_n &&&& \\
    A_5 & -\lambda A_5 + A_4 & \lambda I_n & A_3 & -I_n && \\
    -I_n & \lambda I_n & 0 & 0 & 0 && \\
    & A_3 & 0 & -\lambda A_3 + A_2 & \lambda I_n & A_1 & -I_n \\
    & -I_n & 0 & \lambda I_n & 0 & 0 & 0 \\
    &&& A_1 & 0 & -\lambda A_1 + A_0 & \lambda I_n \\
    &&& -I_n & 0 & \lambda I_n & 0 \\
    \end{bmatrix}.
    \end{equation*}
\end{center}
Let  $\mathbf{c}=(1,2,4,6,3,5,7)$. Then,
$$(\Pi_\mathbf{c}^n)^{\mathcal{B}} L_P(\lambda)\Pi_\mathbf{c}^n =\left[ \begin{array}{cccc|ccc}
    \lambda A_7 + A_6 & A_5 & 0 &0 &-I_n &0 & 0 \\
    A_5 & -\lambda A_5 + A_4 &  A_3  & 0 & \lambda I_n & -I_n & 0 \\
    0 & A_3 & - \lambda A_3 + A_2  & A_1 & 0 &\lambda I_n & -I_n \\
     0 & 0  & A_1 & -\lambda A_1 + A_0 &  0&  0 & \lambda I_n \\
     \hline
     -I_n & \lambda I_n & 0 & 0& 0 & 0 & 0 \\
    0 &-I_n &\lambda I_n & 0 & 0 & 0 & 0 \\
    0 & 0& -I_n & \lambda I_n & 0 & 0 & 0 \\
    \end{array}\right]. $$
We note that $(\Pi_\mathbf{c}^n)^{\mathcal{B}} L_P(\lambda) \Pi_\mathbf{c}^n\in \langle \mathcal{O}_1^P \rangle$ since 
 \[
 (\Pi_\mathbf{c}^n)^{\mathcal{B}} L_P(\lambda)\Pi_\mathbf{c}^n= \left [ \begin{array}{cc} I_{4n} & C \\ 0 & I_{3n}\end{array} \right] \mathcal{O}_1^P \left [ \begin{array}{cc} I_{4n} & C \\ 0 & I_{3n}\end{array} \right]^{\mathcal{B}},
 \]
where 
$$C = \left[ \begin{array}{ccc} 0 & 0 & 0 \\ -A_5 & 0 & 0 \\ 0 & -A_3 & 0 \\ 0 & 0 & -A_1 \end{array} \right].$$ Moreover, $(\Pi_\mathbf{c}^n)^{\mathcal{B}} L_P(\lambda) \Pi_\mathbf{c}^n$ is a strong linearization of every matrix polynomial $P(\lambda)$.

\end{example}

\begin{remark}
We note that, when $k$ is odd, by Example \ref{D1DkO1} and Theorem \ref{th;TP}, the three pencils $D_1(\lambda, P)$, $D_k(\lambda, P)$ and $\mathcal{T}_P(\lambda)$ are permutationally congruent to some pencil in  $\langle \mathcal{O}_1^P \rangle$. In fact,  $\mathcal{T}_P(\lambda)$ is essentially $\mathcal{O}_1^P(\lambda)$, after permuting some block-rows and some block-columns. Thus, $\mathcal{T}_P(\lambda)$ could be seen, in layman's terms, as the ``skeleton'' of $D_1(\lambda, P)$ and $D_k(\lambda, P)$, that is, the least information that can be retained from these pencils without  stopping from being a linearization of $P(\lambda)$. Hence, $\mathcal{T}_P(\lambda)$ is an  ideal candidate to outperform numerically the combined use of $D_1(\lambda, P)$ and $D_k(\lambda, P)$ in the solution of the block-symmetric polynomial eigenvalue problem. This problem is studied in \cite{TPodd}.
\end{remark}

The following example gives the pencils in $\langle \mathcal{O}_1^P \rangle$  permutationally block congruent to $D_1(\lambda, P)$, $D_k(\lambda, P)$ and $\mathcal{T}_P(\lambda)$, when $k=3$. 

\begin{example} Let $k=3$. Then, 
$D_1(\lambda, P)$ is permutationally block congruent to the pencil 
$$\left[ \begin{array}{cc|c} \lambda A_3 +A_2 & A_1 & A_0\\ A_1 & -\lambda A_1 +A_0 & -\lambda A_0\\ \hline A_0 & - \lambda A_0 & 0 \end{array} \right ]\in \langle \mathcal{O}_1^P \rangle.$$
The pencil $D_k(\lambda, P)$ is permutationally block congruent to the pencil 
$$\left[ \begin{array}{cc|c} \lambda A_3 -A_2 &  \lambda A_2 & -A_3 \\  \lambda A_2 & \lambda A_1 - A_0 & \lambda A_3 \\ \hline -A_3  & \lambda A_3 & 0\end{array} \right ] \in \langle \mathcal{O}_1^P \rangle.$$
The pencil $\mathcal{T}_P(\lambda)$ is permutationally block congruent to the pencil 
$$\mathcal{O}_1^P(\lambda)=\left[ \begin{array}{cc|c} \lambda A_3+ A_2 & 0 & -I_n \\ 0 & \lambda A_1 +A_0 & \lambda I_n \\ \hline
-I_n & \lambda I_n & 0 \end{array} \right ] =  \mathcal{O}_1^P(\lambda).$$

\end{example}

\section{Conclusions and future work}

In this paper we have introduced four families of block-symmetric pencils that, under some generic nonsingular conditions, are block-symmetric block minimal bases pencils and strong linearizations of a matrix polynomial $P(\lambda)$. Furthermore, we have shown that every block-symmetric GFP and block-symmetric GFPR is permutationally block-congruent to a pencil in the union of these four families, which provides an alternative approach to the implicit definition of these pencils as products of elementary matrices by providing their block structure. The importance of this result resides in the expectation that the explicit block structure of the block-symmetric GFP and GFPR will provide a venue to explore their numerical properties such as conditioning of eigenvalues and backward error of approximate eigenpairs. In particular, our objective is to find linearizations of $P(\lambda)$ in these famllies with optimal condition number and backward error that can replace the combined used of $D_1(\lambda, P)$ and $D_k(\lambda, P)$ when $P(\lambda)$ is symmetric or Hermitian, as  suggested in the current literature.   

\appendix 
\section{(Proof of Theorem \ref{thm:main_GFPR})}

Here we include the proof of Theorem \ref{thm:main_GFPR}. We start with some extra concepts and results that are necessary for this proof.

\subsection{Auxiliary notation for pencils that are   block-permutationally equivalent to pencils  in $\langle \mathcal{O}_1^P\rangle \cup \langle \mathcal{E}_1^P\rangle$}

In order to prove that all block-symmetric GFPR pencils  are  permutationally block-congruent to a pencil  in $\langle \mathcal{O}_1^P\rangle \cup \langle \mathcal{O}_2^P\rangle \cup \langle \mathcal{E}_1^P\rangle \cup \langle \mathcal{E}_2^P\rangle $,  we introduce in this section some useful concepts.

\begin{definition}{\rm \cite{canonical_Fiedler}}\label{block-perm}
Let $k,n\in \mathbb{N}$. 
Let $\mathbf{c}=(c_1, c_2, \ldots, c_k)$ be a permutation of the set $\{1:k\}$. 
Then, we call the block-permutation matrix associated with $(\mathbf{c}, n)$, and denote it by $\Pi^n_{\mathbf{c}}$, the $k\times k$ block-matrix whose $(c_i,i)$th block-entry is $I_n$, for $i=1:k$, and having $0_n$ in every other block-entry. 
In particular, we denote by $\boldsymbol{\mathrm{id}}=(1:k)$ the identity permutation. 
\end{definition}

When the scalar $n$ is clear in the context, we  write $\Pi_{\mathbf{c}}$   instead of $\Pi_{\mathbf{c}}^n$ to simplify the notation.

An important block-permutation matrix for this paper is the so-called \emph{block standard involutory permutation matrix (block-sip matrix)}.
Such block-permutation matrix is
\begin{equation}\label{eq:sip matrix}
R_k := \Pi^n_{(k:1)} = \begin{bmatrix}
0 & \cdots & I_n \\ \vdots & \iddots & \vdots \\ I_n & \cdots & 0
\end{bmatrix}\in\mathbb{F}^{kn\times kn}.
\end{equation}

 In Definition \ref{def:permutationally_eq}, we introduce some useful notation for pencils that are block-permutationally equivalent to pencils in $\langle \mathcal{O}_1^P\rangle$.
\begin{definition}\label{def:permutationally_eq}
Let  $k$ be an odd positive integer, let $L(\lambda)\in\mathbb{F}[\lambda]^{kn\times kn}$ be a $k\times k$ block-pencil with block-entries of size $n\times n$,  and let $s=(k-1)/2$.  
Assume that there exist  block-permutation matrices  $\Pi_{\boldsymbol\ell}^n$ and $\Pi_{\mathbf{r}}^n$  such that $C(\lambda):=(\Pi_{\boldsymbol\ell}^n)^{\mathcal{B}} L(\lambda) \Pi_{\mathbf{r}}^n\in \langle \mathcal{O}_1^P\rangle$.
\begin{itemize}
\item[\rm (a-1)]
We call the upper-left $(s+1)\times (s+1)$ block submatrix of $C(\lambda)$ the \emph{body of $L(\lambda)$ }relative to $(\boldsymbol\ell,\mathbf{r})$.
\item[\rm (a-2)] We call the \emph{body block-rows} (resp. \emph{body block-columns}) of $L(\lambda)$ relative to $(\boldsymbol\ell,{\mathbf{r}})$ the block-rows (resp. block-columns) of $L(\lambda)$ that, after the permutations, occupy the first $s+1$ block-rows (resp. block-columns) of $C(\lambda)$. 
\item [\rm (b)] We call the \emph{wing block-rows} (resp. 
\emph{wing block-columns}) of $L(\lambda)$ relative to $(\boldsymbol\ell,\mathbf{r})$ the block-rows (resp. block-columns) of $L(\lambda)$ that are not body block-rows (resp. body block-columns) relative to $(\boldsymbol\ell,{\mathbf{r}})$. 
\end{itemize}
\end{definition}

 The following example illustrates the concepts introduced in  Definition \ref{def:permutationally_eq}.

 \begin{example}\label{ex:permutationally_eq}
Let us consider the following pencil
\[
L(\lambda)=
\begin{bmatrix}
\lambda A_5+A_4 & A_3 & -I_n & 0 & 0 \\
A_3 & A_2-\lambda A_3 & \lambda I_n & A_1 & -I_n \\
-I_n & \lambda I_n & 0 & 0 & 0 \\
0 & A_1 & 0 & A_0-\lambda A_1 & \lambda I_n \\
0 & -I_n & 0 & \lambda I_n & 0 
\end{bmatrix},
\]
which is a block-symmetric GFPR associated with the matrix polynomial $P(\lambda)=\sum_{i=0}^5 A_i\lambda^i\in\mathbb{F}[\lambda]^{n\times n}$ (this type of block-symmetric GFPR is called the  simple FPR with parameter $k-1$ in \cite{Her-GFPR}).
Consider the permutation  $\mathbf{c}=(1,2,4,3,5)$ of $\{1:5\}$.
Then, 
\[
(\Pi_{\mathbf{c}}^n)^{\mathcal{B}} L(\lambda) \Pi_{\mathbf{c}}^n =
\left[\begin{array}{ccc|cc}
\lambda A_5+A_4 & A_3 & 0 & -I_n & 0 \\
A_3 & A_2-\lambda A_3 & A_1 & \lambda I_n & -I_n \\
0 & A_1 & A_0-\lambda A_1 & 0 & \lambda I_n \\ \hline
-I_n & \lambda I_n & 0 & 0 & 0 \\
0 & -I_n & \lambda I_n & 0 & 0 
\end{array}\right]\in \langle \mathcal{O}_1^P\rangle.
\]
Hence,
\begin{itemize}
\item the first, second, and forth block-rows and block-columns of $L(\lambda)$ are, respectively, its body block-rows and body block-columns relative to $(\mathbf{c},\mathbf{c})$; and
\item the third and fifth block-rows and block-columns of $L(\lambda)$ are, respectively, its wing block-rows and wing block-columns relative to $(\mathbf{c},\mathbf{c})$.
\end{itemize}
 \end{example}

In Definition \ref{def:permutationally_eq2}, we introduce some useful notation for pencils that are block-permutationally equivalent to pencils in $\langle \mathcal{E}_1^P\rangle$.
\begin{definition}\label{def:permutationally_eq2}
Let  $k$ be an even positive integer, let $L(\lambda)\in\mathbb{F}[\lambda]^{kn\times kn}$ be a $k\times k$ block-pencil with block-entries of size $n\times n$,   and let $s=(k-2)/2$.  
Assume that there exist  block-permutation matrices  $\Pi_{\boldsymbol\ell}^n$ and $\Pi_{\mathbf{r}}^n$  such that $C(\lambda):=(\Pi_{\boldsymbol\ell}^n)^{\mathcal{B}} L(\lambda) \Pi_{\mathbf{r}}^n\in \langle \mathcal{E}_1^P\rangle$.
\begin{itemize}
\item[\rm (a-1)]
We call the upper-left $(s+1)\times (s+1)$ block submatrix of $C(\lambda)$ the \emph{body of $L(\lambda)$ }relative to $(\boldsymbol\ell,\mathbf{r})$.
\item[\rm (a-2)] We call the \emph{body block-rows} (resp. \emph{body block-columns}) of $L(\lambda)$ relative to $(\boldsymbol\ell,{\mathbf{r}})$ the block-rows (resp. block-columns) of $L(\lambda)$ that, after the permutations, occupy the first $s+1$ block-rows (resp. block-columns) of $C(\lambda)$. 
\item[\rm (b)] We call the \emph{exceptional block-row} (resp. \emph{exceptional block-column}) of $L(\lambda)$ relative to $(\boldsymbol\ell,{\mathbf{r}})$ the block-row (resp. block-column) of $L(\lambda)$ that, after the permutations, occupies the $s+2$ block-row (resp. block-column) of $C(\lambda)$.
\item [\rm (c)] We call the \emph{wing block-rows} (resp. 
\emph{wing block-columns}) of $L(\lambda)$ relative to $(\boldsymbol\ell,\mathbf{r})$ the block-rows (resp. block-columns) of $L(\lambda)$ that are not body block-rows nor exceptional block-rows (resp. body block-columns nor exceptional block-columns) relative to $(\boldsymbol\ell,{\mathbf{r}})$. 
\end{itemize}
\end{definition}

 We illustrate the concepts introduced in Definition \ref{def:permutationally_eq2} in the following example.
 \begin{example}\label{ex:permutationally_eq2}
Let us consider the following pencil
\[
L(\lambda)=
\begin{bmatrix}
\lambda A_6+A_5 & A_4 & -I_n & 0 & 0 & 0 \\
A_4 & A_3-\lambda A_4 & \lambda I_n & A_2 & -I_n & 0 \\
-I_n & \lambda I_n & 0 & 0 & 0 & 0 \\
0 & A_2 & 0 & A_1-\lambda A_2 & \lambda I_n & A_0 \\
0 & -I_n & 0 & \lambda I_n & 0 & 0 \\
0 & 0 & 0 & A_0 & 0 & -\lambda A_0
\end{bmatrix},
\]
which is a block-symmetric GFPR associated with the matrix polynomial $P(\lambda)=\sum_{i=0}^6 A_i\lambda^i\in\mathbb{F}[\lambda]^{n\times n}$ (this type of block-symmetric GFPR is called the  simple FPR with parameter $k-1$ in \cite{Her-GFPR}). 
Consider the permutation  $\mathbf{c}=(1,2,4,6,3,5)$ of $\{1:6\}$.
Then, 
\[
(\Pi_{\mathbf{c}}^n)^{\mathcal{B}} L(\lambda) \Pi_{\mathbf{c}}^n =
\left[\begin{array}{ccc|c:cc}
\lambda A_6+A_5 & A_4 & 0 & 0 & -I_n & 0 \\
A_4 & A_3-\lambda A_4 & A_2 & 0 & \lambda I_n & -I_n \\
0 & A_2 & A_1-\lambda A_2 & A_0 & 0 & \lambda I_n \\ \hdashline
0 & 0 & A_0 & -\lambda A_0 & 0 & 0 \\ \hline
-I_n & \lambda I_n & 0 & 0 & 0 & 0 \\
0 & -I_n & \lambda I_n & 0 & 0 & 0 \\
\end{array}\right]\in \langle \mathcal{E}_1^P\rangle.
\]
Hence,
\begin{itemize}
\item the first, second, and fourth block-rows and block-columns of $L(\lambda)$ are, respectively, its body block-rows and body block-columns relative to $(\mathbf{c},\mathbf{c})$; 
\item the third and fifth block-rows and block-columns of $L(\lambda)$ are, respectively, its wing block-rows and wing block-columns relative to $(\mathbf{c},\mathbf{c})$; and
\item the sixth block-row and block-column of $L(\lambda)$ are, respectively, its exceptional block-row and exceptional block-column relative to $(\mathbf{c},\mathbf{c})$.

\end{itemize}
\end{example}

\subsection{Auxiliary definitions and lemmas for index tuples}
In order to prove the auxiliary results needed in the proof of Theorem  \ref{thm:main_GFPR}, we need to introduce some definitions and results associated with the concept of  index tuple, introduced in Section \ref{sec:index-tuples}.

We start by introducing  a canonical form for index tuples satisfying the SIP (recall Definition \ref{defSIP}).
To do this, we need the following three definitions.
\begin{definition}
We say that two nonnegative indices $i$ and $j$ in an index tuple commute  if  $|i-j|\neq 1$.
\end{definition}

\begin{definition}{\rm \cite[Definition 3.4]{GFPR}}\label{def:equivalence_tuples}
Given two index tuples $\mathbf{t}$ and $\mathbf{t}'$ of nonnegative indices, we say that $\mathbf{t}$ is equivalent to $\mathbf{t}'$ (and write $\mathbf{t} \sim \mathbf{t'}$), if $\mathbf{t}= \mathbf{t'}$ or $\mathbf{t}'$ can be obtained from $\mathbf{t}$ by interchanging a finite number of times two distinct commuting indices in adjacent positions, that is, indices $t_i$ and $t_{i+1}$ such that  $| t_i - t_{i+1}| \neq 1$ and $t_i \neq t_{i+1}$.
\end{definition}

Notice that the relation $\sim$ introduced in Definition \ref{def:equivalence_tuples} is an equivalence relation.
Note, in addition, that the SIP is invariant under this relation.

\begin{remark}\label{commutativity}
It is easy to check that the commutativity relation
\begin{equation}
M_{i}(B_1)M_{j}(B_2)=M_{j}(B_2)M_{i}(B_1)\label{commut}%
\end{equation}
holds for any $n\times n$ matrices $B_1$ and $B_2$ if   $||i|-|j||\neq 1$ and  $|i|\neq |j|$. 
These commutativity relations readily imply that the product of elementary matrices is invariant under the equivalence relation introduced in Definition \ref{def:equivalence_tuples}, i.e.,  given an index tuple $\mathbf{t}$ and a matrix assignment $\mathcal{Z}$ for $\mathbf{t}$, if $\mathbf{t} \sim \mathbf{t'}$, then $M_{\mathbf{t}}(\mathcal{Z})= M_{\mathbf{t'}}(\tilde{\mathcal{Z}})$, where $\tilde{\mathcal{Z}}$ is the matrix assignment for $\mathbf{t}'$ obtained from $\mathbf{t}$ by assigning to each index in $\mathbf{t}'$ the matrix that was assigned by $\mathcal{Z}$ to the corresponding index in $\mathbf{t}$.
\end{remark}

 \begin{definition}
\label{FSC}{\rm \cite[Theorem 1]{ant-vol11}} Let ${\mathbf{t}}$ be an index tuple
with indices from $\{0:h\}$, $h \geq 0$. 
Then ${\mathbf{t}}$ is said to be in \emph{column standard form}  if
\[
{\mathbf{t}}=\left(  a_{s}:b_s,a_{s-1}:b_{s-1},\hdots,a_{2}:b_{2}%
,a_{1}:b_{1}\right),
\]
with $h\geq b_{s}>b_{s-1}>\cdots>b_{2}>b_{1}\geq0$ and $0\leq a_{j}\leq b_{j}%
$, for all $j=1:s$. We call  each subtuple of consecutive indices
$(a_{i}:b_{i})$ a \emph{string} of ${\mathbf{t}}$.
\end{definition}

The relation between the SIP, the equivalence relation of tuples, and the column standard form is stated in the following lemma.
\begin{lemma}\label{SIP-csf}{\rm \cite[Theorem 2]{ant-vol11}}
Let $\mathbf{t}$ be an index tuple. 
\begin{itemize}
\item[\rm (i)]  If the indices of $\mathbf{t}$ are all  nonnegative integers, then ${\mathbf{t}}$ satisfies the SIP if and only if $\,{\mathbf{t}}$ is equivalent to a tuple in column standard form.
\item[\rm (ii)] If the indices of $\mathbf{t}$ are all negative integers and $a$ is the minimum index in $\mathbf{t}$, then $\mathbf{t}$ satisfies the SIP if and only if $-a+\mathbf{t}$ is equivalent to a tuple in column standard form.
\end{itemize}
\end{lemma}

Two tuples in column standard form are equivalent if and only if they coincide.
This motivates the following definition.
\begin{definition}{\rm \cite[Definition 3.9]{GFPR}}\label{def:unique_csf}
The unique index tuple in column standard form equivalent to an index tuple $\mathbf{t}$ of nonnegative integers satisfying the SIP is called the  \emph{column standard form of $\mathbf{t}$}  and is denoted by $\mathrm{csf}(\mathbf{t})$.
\end{definition}

In the next two definitions, we pay special attention to some indices of  tuples of nonnegative integers satisfying the SIP that will play a key role in the proofs of the main results in this paper.
\begin{definition}{\rm \cite[Definition 4.12]{canonical_Fiedler}}\label{def;heads}
For an arbitrary index tuple $\mathbf t$ satisfying the SIP with $\csf(\mathbf t) = \left(  a_{s}:b_s,\hdots
,a_{1}:b_{1}\right)$, we define $\heads(\mathbf {t}) := \{b_i ~|~ 1 \leq i \leq s\}$. 
Furthermore, we denote by  $\mathfrak{h}(\mathbf{t})$ the cardinality $s$ of $\heads(\mathbf{t})$. 
\end{definition}

Given an index tuple $\mathbf{t}$, note that $\mathfrak{h}(\mathbf{t})$ gives not only the cardinality of the set $\heads(\mathbf {t})$, but also the number of strings in $\mathrm{csf}(\mathbf{t})$.

\begin{definition}{\rm \cite[Definition 4.13]{canonical_Fiedler}}\label{tuple-type}
Given an index tuple $\mathbf t$ and an index $x$ such that $(\mathbf{t}, x)$ satisfies the SIP,  we say that $x$ is of \emph{Type~I} relative to $\mathbf t$ if   $\mathfrak{h}(\mathbf t, x)= \mathfrak{h}(\mathbf t)$, and of \emph{Type~II} otherwise. 
That is, $x$ is of Type I relative to $\mathbf{t}$ if $\csf(\mathbf{t}, x)$ has the same number of heads (and, therefore of strings) as $\csf(\mathbf{t})$, and of \emph{Type~II} relative to $\mathbf{t}$ otherwise.
\end{definition}

The next result relates the SIP property with the set $\mathrm{heads}(\mathbf{t})$.
\begin{lemma}\label{lemma:txSIP} {\rm \cite[Lemma 4.13]{canonical}}
Let $h$ be a positive integer and let $\mathbf{t}$ be an index tuple with indices from $\{0:h-1\}$.
Let $(a:b)$ be a string with indices from $\{0:h-2\}$.
Then, the tuple $(\mathbf{t},a:b)$ satisfies the SIP property if and only if $\mathbf{t}$ satisfies the SIP and $c\notin\mathrm{heads}(\mathbf{t})$, for all $c\in (a:b)$. 
\end{lemma}

\begin{proposition}\label{prop:SIPheads}{\rm \cite[Lemma 4.14]{canonical}}
Let $\mathbf{t}=(a_s:b_s, \ldots, a_2:b_2, a_1:b_1)$ be a nonempty index tuple in column standard form with indices from $\{0:k-1\}$, for $k\geq 1$. Let $x$ be an index in $\{0:b_s-1\}$ such that $(\mathbf{t},x)$ satisfies the SIP. Then, $x$ is of Type I relative to $\mathbf{t}$ if and only if $x-1\in \heads(\mathbf{t})$. In particular, $x=0$ is always an index of Type II relative to $\mathbf{t}$, for every nonempty index tuple $\mathbf{t}$ in column standard form with nonnegative indices. 
\end{proposition}

\begin{remark}\label{rem: Type I and II heads}
Using the notation of Proposition \ref{prop:SIPheads}, we note that this proposition implies that $\heads(\mathbf t, 
\allowbreak x) = (\heads(\mathbf t) \setminus \{b_j\}) \cup \{x\}$ when $x$ is an index of Type I relative to $\mathbf t$ and $x=b_j+1$.  Furthermore,  if $x$ is of Type II relative to $\mathbf t$, then $\heads(\mathbf t,x) = \heads(\mathbf t) \cup \{x\}$. 
\end{remark}

\subsection{Proof of Theorem \ref{thm:main_GFPR}} 

Before we give  the proof of Theorem \ref{thm:main_GFPR}, 
 we need some auxiliary results for block-symmetric GFPR associated with even and odd matrix polynomials.

\subsubsection{Auxiliary results for the odd degree case}

We first prove that  Theorem \ref{thm:main_GFPR} holds for the block-symmetric GFPR called the simple FPR with parameter $k-1$ in \cite{Her-GFPR}.
\begin{theorem}\label{aux;simple}
Let $P(\lambda)=\sum_{i=0}^k A_i\lambda^i\in\mathbb{F}[\lambda]^{n\times n}$ with odd degree $k$, let $s=(k-1)/2$, and let $F_k(\lambda)=(\lambda M^P_{-k}-M^P_{\mathbf{w}_{k-1}})M^P_{\mathbf{c}_{k-1}}$.
Let $\mathbf{c}$ be the permutation of $\{1:k\}$ given by $(1,2,4,6,\ldots, k-1, 3, 5, \ldots, k)$. Then,
\begin{equation}\label{eq:main_FPR}
( \Pi_\mathbf{c}^n)^\mathcal{B} F_k(\lambda) \Pi_\mathbf{c}^n = 
\left[\begin{array}{c|c}
M(\lambda;P) + CK_s(\lambda)+K_s(\lambda)^T C^\mathcal{B} & K_s(\lambda)^T \\ \hline
K_s(\lambda) & 0
\end{array}\right] \in\langle \mathcal{O}_1^P \rangle,
\end{equation}
for some matrix $C$.
Moreover, the following statements hold.
\begin{itemize}
\item[\rm (a)] The wing block-columns of  $(\Pi_\mathbf{c}^n)^{\mathcal{B}} F_{k}(\lambda)$ relative to $(\mathbf{id}, \mathbf{c})$  are  of the form $-e_i \otimes I_n + \lambda e_{i+1} \otimes I_n$, for $1\leq i \leq s$,  and are  located in  positions $k-j$, where $j\in \{0:k-2\}$ and $(\mathbf{w}_{k-1},\mathbf{c}_{k-1}, j)$ satisfies the SIP.

\item[\rm (b)] The wing block-rows of $F_{k}(\lambda)\Pi_\mathbf{c}^n$ relative to $(\mathbf{c},\mathbf{id})$ are of the form $-e_i^T \otimes I_n + \lambda e_{i+1}^T \otimes I_n$, for $1\leq i \leq s$,  and are  located in  positions $k-j$, where $j\in \{0:k-2\}$ and $(\mathbf{w}_{k-1},\mathbf{c}_{k-1}, j)$ satisfies the SIP.

\item[\rm (c)]  The first block-row and the first block-column of $F_{k}(\lambda)$ are, respectively, the first body block-row and the first body block-column of $F_{k}(\lambda)$ relative to $(\mathbf{c}, \mathbf{c} )$.
 Moreover, the block-entry of  $( \Pi_\mathbf{c}^n)^\mathcal{B} F_k(\lambda) \Pi_\mathbf{c}^n$ in position $(1,1)$ equals $\lambda A_k + A_{k-1}$.
\end{itemize}
Furthermore, the pencil $( \Pi_\mathbf{c}^n)^\mathcal{B} F_k(\lambda) \Pi_\mathbf{c}^n$ is a strong linearization of $P(\lambda)$. 
\end{theorem}

\begin{proof}
We begin by recalling the block-structure of $F_k(\lambda)$ when $k$ is odd \cite[Section 8]{Her-GFPR}.
Notice that $F_3(\lambda)$ is partitioned to show that it is an extended block Kronecker pencil.
\[
F_3(\lambda) = \left[ \begin{array}{cc|c} \lambda A_3+A_2 & A_1 & -I_n \\ A_1 & -\lambda A_1 +A_0 & \lambda I_n \\ \hline -I_n & \lambda I_n & 0 \end{array} \right]
\]
For $k\geq 5$ odd, we have that $F_k(\lambda)$ is of the form
\begin{equation}\label{Fodd}
\begin{bmatrix}
		\lambda A_k + A_{k-1} & A_{k-2} & -I_n &&&&&&\\
		A_{k-2} & A_{k-3} - \lambda A_{k-2} & \lambda I_n & A_{k-4} & -I_n &&&&\\
		-I_n & \lambda I_n & 0 & 0 &&&&&\\
		& A_{k-4} & 0 & \ddots & \ddots &&&&&\\
		& -I_n && \ddots &&&&&&\\
		&&&&& A_2 - \lambda A_3 & \lambda I_n & A_1 & -I_n \\
		&&&&& \lambda I_n & 0 & 0 & 0\\
		&&&&& A_1 & 0 & A_0 -\lambda A_1 & \lambda I_n \\
		&&&&& -I_n & 0 & \lambda I_n & 0\\
	\end{bmatrix},
\end{equation}
where the empty spaces denote zero blocks.

For $k=3$ the theorem can be easily checked. So let us assume $k\geq 3$. 
From the explicit block-structure of the pencil $F_k(\lambda)$ above, it is easy to see that  part (c) holds and $( \Pi_\mathbf{c}^n)^\mathcal{B} F_k(\lambda) \Pi_\mathbf{c}^n$ is of the form \eqref{eq:main_FPR}, with
%
\[
C=\left[ \begin{array}{cccc} 0 & 0 & \cdots & 0 \\ -A_{k-2} & 0 & \cdots & 0 \\ 0 & -A_{k-4} & \cdots & 0\\ \vdots & \vdots & \ddots & \vdots\\ 0 & 0 & \cdots & -A_1 \end{array} \right].
\]

Parts (a) and (b) follow from checking directly that the wing block-columns and the wing block-rows of $F_k(\lambda)$ relative to $(\mathbf{c}, \mathbf{c})$ are in positions $k-j\in\{3,5,\hdots,k\}$, which, in turn, by Lemmas \ref{wcSIP} and \ref{lemma:txSIP}, correspond to those values of $j\in \{0:k-2\}$ such that $(\mathbf{w}_{k-1},\mathbf{c}_{k-1}, j)$ satisfies the SIP.  
Note that 
$\mathrm{csf}(\mathbf{w}_{k-1}, \mathbf{c}_{k-1})= (k-2:k-1, k-4:k-2, k-6:k-4, \ldots,  1:3, 0:1).$
\end{proof}

The second step towards proving Theorem \ref{thm:main_GFPR} consists in showing that this theorem holds  for general block-symmetric GFPR with parameter $k-1$. 
We also prove  some structural information concerning the block-rows and block-columns of these particular GFPR.
This structural information will allow us to prove this result using an induction argument. 
\begin{theorem}\label{thm:aux_GFPR1}
Let $P(\lambda)=\sum_{i=0}^k A_i\lambda^i\in\mathbb{F}[\lambda]^{n\times n}$ with odd degree $k$, let $s=(k-1)/2$, and let $L_P(k-1,\mathbf{t}_w,\emptyset,\mathcal{Z}_w,\emptyset)$ be a block-symmetric GFPR associated with $P(\lambda)$, that is, a pencil of the form
\[
M_{\mathbf{t}_w}(\mathcal{Z}_w)(\lambda M^P_{-k}-M^P_{\mathbf{w}_{k-1}})M^P_{\mathbf{c}_{k-1}}M_{\rev(\mathbf{t}_w)}(\rev(\mathcal{Z}_w)).
\]
Then, there exists a permutation $\mathbf{c}$ of $\{1:k\}$ such that
\begin{equation}\label{eq:aux_GFPR1}
 (\Pi_\mathbf{c}^n)^\mathcal{B} L_P(k-1,\mathbf{t}_w,\emptyset, \mathcal{Z}_w,\emptyset) \Pi_\mathbf{c}^n =  \left[\begin{array}{c|c}
            M(\lambda;P) + C K_{s}(\lambda)+ K_{s}^T(\lambda)C^\mathcal{B} 
            & K_{s}(\lambda)^T  B^\mathcal{B} \\
            \hline
            BK_{s}(\lambda) & 0
        \end{array}\right]\in \langle \mathcal{O}_1^P \rangle.
 \end{equation}
Moreover, the following statements hold:
\begin{itemize}
\item[\rm (a)] The wing block-columns of  $(\Pi_\mathbf{c}^n)^{\mathcal{B}} L_P(k-1,\mathbf{t}_w,\emptyset, \mathcal{Z}_w,\emptyset)$ relative to $(\mathbf{id}, \mathbf{c})$ that are  of the form $-e_i \otimes I_n + \lambda e_{i+1} \otimes I_n$, for $1\leq i \leq s$,  are  located in  positions $k-j$, where $j\in \{0:k-2\}$ and $(\mathbf{t}_w,\mathbf{w}_{k-1},\mathbf{c}_{k-1},\rev(\mathbf{t}_w), j)$ satisfies the SIP.

\item[\rm (b)] The wing block-rows of $L_P(k-1,\mathbf{t}_w,\emptyset,\mathcal{Z}_w,\emptyset)\Pi_\mathbf{c}^n$ relative to $(\mathbf{c}, \mathbf{id})$ that are of the form $-e_i^T \otimes I_n + \lambda e_{i+1}^T \otimes I_n$, for $1\leq i \leq s$,   are  located in  positions $k-j$, where $j\in \{0:k-2\}$ and $(\mathbf{t}_w,\mathbf{w}_{k-1},\mathbf{c}_{k-1},\rev(\mathbf{t}_w), j)$ satisfies the SIP.

\item[\rm (c)]  The first block-row and the first block-column of $L_P(k-1,\mathbf{t}_w,\emptyset,\mathcal{Z}_w,\emptyset)$ are, respectively, the first body block-row and the first body block-column  of $L_P(k-1,\mathbf{t}_w,\emptyset,\mathcal{Z}_w,\emptyset)$ relative to $(\mathbf{c},\mathbf{c})$.
 Moreover, the block-entry of $ (\Pi_\mathbf{c}^n)^\mathcal{B} L_P(k-1,\mathbf{t}_w,\emptyset,\allowbreak \mathcal{Z}_w,\emptyset) \Pi_\mathbf{c}^n$ in position $(1,1)$ equals $\lambda A_k + A_{k-1}$.
\end{itemize}
Furthermore, if $\mathcal{Z}_w$ is a nonsingular matrix assignment for $\mathbf{t}_w$, then $B$ and $B^\mathcal{B}$ are nonsingular, and $(\Pi_\mathbf{c}^n)^\mathcal{B} L_P(k-1,\mathbf{t}_w,\emptyset, \mathcal{Z}_w,\emptyset) \Pi_\mathbf{c}^n$ is a strong linearization of $P(\lambda)$.
\end{theorem}
\begin{proof}
We prove the result by induction on the number of indices in $\mathbf{t}_w$. 
When the tuple $\mathbf{t}_w$ is empty, then $L_P(k-1,\emptyset
,\emptyset,\emptyset,\emptyset)$ is the simple FPR with parameter $k-1$ associated with $P(\lambda)$ and the result follows by Theorem \ref{aux;simple}.

Assume that the result holds for tuples $\mathbf{t}_w$ with at most $\ell$ indices, with $\ell\geq  0$.
Let $\mathbf{t}_w = (\mathbf{t}'_w,x)$ be a tuple with $\ell+1$ indices, let $ \mathcal{Z}_w = (\mathcal{Z}'_w,Z_x)$ be a matrix assignment for $\mathbf{t}_w$,  let $\mathcal{L}(\lambda) := L_P(k-1,\mathbf{t}_w,\emptyset,\mathcal{Z}_w,\emptyset)$ and $\mathcal{L}'(\lambda): = L_P(k-1,\mathbf{t}'_w,\emptyset,\mathcal{Z}'_w,\emptyset)$. 
Clearly, we have $\mathcal{L}(\lambda) = M_x(Z_x)\mathcal{L}'(\lambda)M_x(Z_x)$.
 Since $\mathbf{t}'_w$ has $\ell$ indices, by the inductive hypothesis, there exists a block-permutation matrix $\Pi_{\boldsymbol\sigma}$ such that
 \[
  (\Pi_{\boldsymbol\sigma})^\mathcal{B} \mathcal{L}'(\lambda) \Pi_{\boldsymbol\sigma} = 
\left[\begin{array}{c|c}
M'(\lambda) & K_s(\lambda)^T(B')^\mathcal{B} \\ \hline
B'K_s(\lambda) & 0
\end{array}\right]\in\langle\mathcal{O}_1^P\rangle,
\]
for some matrix $B'$ and where $M'(\lambda)= M(\lambda;P)+ C' K_s(\lambda)+ K_s^T(C')^{\mathcal{B}}$ for some matrix  $C'$. 
Moreover, properties (a), (b), and (c) hold for $\mathcal{L}'(\lambda)$.
Additionally, notice that if $\mathcal{Z}_w$ is a nonsingular matrix assignment for $\mathbf{t}_w$, then $\mathcal{Z}^\prime$ is a nonsingular matrix assignment for $\mathbf{t}^\prime_w$. 
Hence, if $\mathcal{Z}_w$ is a nonsingular matrix assignment for $\mathbf{t}_w$, then the matrices $B^\prime$ and $(B^\prime)^\mathcal{B}$ are both nonsingular, and  $(\Pi_{\boldsymbol\sigma})^\mathcal{B} \mathcal{L}'(\lambda) \Pi_{\boldsymbol\sigma}$ is a strong linearization of $P(\lambda)$. 

By the definition of block-symmetric GFPR, the index tuple $(x,\mathbf{t}'_w,\allowbreak \mathbf{w}_{k-1},\mathbf{c}_{k-1},\rev(\mathbf{t}'_w),x)$  satisfies the SIP, which implies that the tuple $(\mathbf{t}'_w,\allowbreak \mathbf{w}_{k-1},\mathbf{c}_{k-1},\rev(\mathbf{t}'_w),x)$ also satisfies the SIP (recall Remark \ref{remark:SIP}).
Hence, since $x\in \{0:k-2\}$ by definition of $\mathbf{t}_w$, by parts (a) and (b)  applied to $\mathcal{L}'(\lambda)$, the $(k-x)$th block-column of $\mathcal{L}'(\lambda)$ is one of its wing block-columns relative to $(\boldsymbol\sigma, \boldsymbol\sigma)$, and the $(k-x)$th block-row of $\mathcal{L}'(\lambda)$ is one of its wing block-rows relative to $(\boldsymbol\sigma, \boldsymbol\sigma)$.
Now we have to consider two cases.

Case I: Assume, first, that $x=0$.
In this case, the action of pre- and post-multiplying $\mathcal{L}'(\lambda)$ by $M_0(Z_0)$ consists in multiplying the $k$th block-row and the $k$th block-column of $\mathcal{L}'(\lambda)$, which are, respectively,  a wing block-row and a  wing block- column  of $\mathcal{L}'(\lambda)$ relative to $(\boldsymbol\sigma, \boldsymbol\sigma)$ by  the matrix $Z_0$.
Thus, we obtain
  \[
  (\Pi_{\boldsymbol\sigma})^\mathcal{B} \mathcal{L}(\lambda) \Pi_{\boldsymbol\sigma}= 
\left[\begin{array}{c|c}
M'(\lambda) & K_s(\lambda)^TB^\mathcal{B} \\ \hline
BK_s(\lambda) & 0
\end{array}\right]\in\langle\mathcal{O}_1^P\rangle,
\] 
with  $B=\diag (I_{rn},Z_0,I_{tn})B'$  and $B^\mathcal{B}=(B')^\mathcal{B}\diag (I_r,Z_0,I_t)$, for some block-identity matrices $I_{rn}$ and $I_{tn}$.
Therefore, \eqref{eq:aux_GFPR1} holds with $\Pi_{\mathbf{c}}^n=\Pi_{\boldsymbol\sigma}^n$ and $C = C'$. 
Moreover, if $\mathcal{Z}_w$ is a nonsingular matrix assignment for $\mathbf{t}_w$, then $Z_0$ is nonsingular.
Thus, in this situation, the matrices $B$ and $B^\mathcal{B}$ are nonsingular, and, by Theorem \ref{thm:first family}, the pencil $(\Pi_{\mathbf{c}})^\mathcal{B} \mathcal{L}(\lambda) \Pi_{\mathbf{c}}$ is a strong linearization of $P(\lambda)$.

Now, we prove parts (a) and  (b)  for the case $x=0$.
Since $\Pi_{\mathbf{c}}^n=\Pi_{\boldsymbol\sigma}^n$, the wing block-columns and  the wing block-rows of $\mathcal{L}'(\lambda)$ and $\mathcal{L}(\lambda)$ relative to $(\mathbf{c},\mathbf{c})$ are located in the same positions.
Furthermore, the wing block-columns (resp. block-rows) of $\mathcal{L}'(\lambda)$ and $\mathcal{L}(\lambda)$ other than those in the $\mathbf{c}^{-1}(k)$th position are equal.
Since $x=0$ is  a Type II index relative to $(\mathbf{t}'_w,\mathbf{w}_{k-1},\mathbf{c}_{k-1}, \rev(\mathbf{t}'_w))$, by Remark \ref{rem: Type I and II heads}, we also have
\[
\mathrm{heads}(\mathbf{t}'_w,\mathbf{w}_{k-1},\mathbf{c}_{k-1}, \rev(\mathbf{t}'_w),0)=
\mathrm{heads}(\mathbf{t}'_w,\mathbf{w}_{k-1},\mathbf{c}_{k-1}, \rev(\mathbf{t}'_w))\cup\{0\}.
\]
Thus, parts (a) and (b) follow from Lemma \ref{lemma:txSIP}, provided that we check that the $k$th block-row (resp. block-column) of $L(\lambda)\Pi_{\mathbf{c}}^n$ (resp. $(\Pi_{\mathbf{c}}^n)^\mathcal{B}L(\lambda)$)  is not, generically, of the form $-e_{i}\otimes I_n+\lambda e_{i+1}\otimes I_n$ (resp. $-e_{i}^T\otimes I_n+\lambda e_{i+1}^T\otimes I_n$), for some $1\leq i\leq s$.
Indeed, the induction hypothesis implies that the $k$th block-row of $\mathcal{L}'(\lambda)\Pi_{\boldsymbol\sigma}^n$ is of the form  $-e_{i}\otimes I_n+\lambda e_{i+1}\otimes I_n$, which, in turn, implies that the $k$th block-row of $L(\lambda)\Pi_\mathbf{c}^n$ is  of the form $-e_{i}\otimes Z_0+\lambda e_{i+1}\otimes Z_0$. 
A similar argument holds for the $k$th block-column.
Thus, parts (a) and (b) are established.
Part (c) follows from the induction hypothesis together with the fact that the first block-rows and first block-columns of $\mathcal{L}(\lambda)$ and $\mathcal{L}'(\lambda)$ are, clearly, equal.

Case II: We assume, now, that $x\neq 0$.
We consider two sub-cases, namely, $x$ is a Type I or a Type II index relative to the tuple $(\mathbf{t}'_w,\mathbf{w}_{k-1},\mathbf{c}_{k-1}, \rev(\mathbf{t}'_w))$.

Assume, first, that $x$ is a Type II index.
By Proposition \ref{prop:SIPheads} and Lemma \ref{lemma:txSIP}, $(\mathbf{t}'_w,\mathbf{w}_{k-1},\mathbf{c}_{k-1}, \allowbreak \rev(\mathbf{t}'_w),x-1)$ satisfies the SIP.
This, in turn, implies that the $(k-x+1)$th block-row (resp. block-column) of $\mathcal{L}'(\lambda)$ is one of its wing block-row (resp. block-column) relative to $(\boldsymbol\sigma, \boldsymbol\sigma)$. 
Additionally, notice that  pre- and post-multiplying the pencil $\mathcal{L}'(\lambda)$ by $M_x(Z_x)$ affects only the $(k-x)$th and $(k-x+1)$th block-rows (resp. block-columns), which are both wing block-rows (resp. wing block-columns) of $\mathcal{L}'(\lambda)$ relative to $(\boldsymbol\sigma, \boldsymbol\sigma)$ by (a) and (b). 
 Then,  we can write 
\begin{equation}\label{partitionL'}
\mathcal{L}'(\lambda)= \begin{blockarray}{ccccc}
 & k-x & k-x+1 &  \\
\begin{block}{[c|cc|c]c}
  L_{1}'(\lambda) & \ell_1'(\lambda)^\mathcal{B} & \ell_2'(\lambda)^\mathcal{B} & L_{2}'(\lambda)^\mathcal{B} &  \\ \cline{1-4}
  \ell_1'(\lambda) & \phantom{\Big{(}} 0_n \phantom{\Big{(}} & 0_n & \ell_3'(\lambda) & k-x \\
  \ell_2'(\lambda) & \phantom{\Big{(}} 0_n \phantom{\Big{(}} & 0_n & \ell_4'(\lambda) & k-x+1 \\ \cline{1-4}
  L_{2}'(\lambda) & \phantom{\Big{(}} \ell_3'(\lambda)^\mathcal{B} \phantom{\Big{(}} & \ell_4'(\lambda)^\mathcal{B} & L_2(\lambda) &  \\ 
\end{blockarray}.
 \end{equation}
Note that, with the partition of $\mathcal{L}'(\lambda)$ given in (\ref{partitionL'}), the block of $\mathcal{L}'(\lambda)$ in position $(2,2)$ is zero due to the fact that, since the $(k-x)$th and $(k-x+1)$th block-rows and block-columns of $\mathcal{L}'(\lambda)$ are, respectively, wing block-rows and wing block-columns of $\mathcal{L}'(\lambda)$ relative to $(\boldsymbol{\sigma}, \boldsymbol{\sigma})$, after being permuted the last $s$ block-entries of each of them are zero.

This implies that the pencil $\mathcal{L}(\lambda)=M_x(Z_x) \mathcal{L}'(\lambda) M_x(Z_x)$ can be partitioned as
\[
\begin{blockarray}{ccccc}
 & k-x & k-x+1 &  \\
\begin{block}{[c|cc|c]c}
  L_{1}'(\lambda) & (Z_x\ell_1'(\lambda)+ \ell_2'(\lambda))^\mathcal{B} & \ell_1'(\lambda)^\mathcal{B} & L_{2}'(\lambda)^\mathcal{B} &  \\ \cline{1-4}
  Z_x \ell_1'(\lambda)+ \ell_2'(\lambda) & \phantom{\Big{(}} 0_n \phantom{\Big{(}}& 0_n & Z_x\ell_3'(\lambda)+\ell_4'(\lambda) & k-x \\
    \ell_1'(\lambda) & \phantom{\Big{(}} 0_n \phantom{\Big{(}} & 0_n & \ell_3'(\lambda) & k-x+1 \\ \cline{1-4}
  L_{2}'(\lambda) & (Z_x\ell_3'(\lambda)+\ell_4'(\lambda))^\mathcal{B} & \phantom{\Big{(}} \ell_3'(\lambda)^\mathcal{B} \phantom{\Big{(}} &  L_2(\lambda) &  \\ 
\end{blockarray},
 \]
 where we have used the fact that the only non-zero block-entries of $\ell_1'(\lambda)$ and $\ell_3'(\lambda)$ are equal to $I_n$ or $\lambda I_n$ (this follows from parts (a) and (b) of the inductive hypothesis).
Then, we easily obtain
 \[
 (\Pi_{\boldsymbol \sigma}^n)^\mathcal{B}\mathcal{L}(\lambda)\Pi_{\boldsymbol \sigma}^n= 
\left[\begin{array}{c|c}
M'(\lambda) & K_s(\lambda)^TB^\mathcal{B} \\ \hline
BK_s(\lambda) & 0
\end{array}\right]\in\langle\mathcal{O}_1^P\rangle,
\] 
with $M'(\lambda)= M(\lambda;P) + C K_s(\lambda) + K_s^T(\lambda)C^{\mathcal{B}}$ and  for some matrices $B$  and $B^{\mathcal{B}}$ of the form 
\[
B = 
\begin{bmatrix}
I_{rn} \\
& 0 & \cdots & I_n \\
& & I_{tn}\\
& I_n & \cdots & Z_x \\
& & & & I_{un}
\end{bmatrix}B^\prime, \quad \mbox{and} \quad 
B^\mathcal{B}=(B^\prime)^\mathcal{B}
\begin{bmatrix}
I_{rn} \\
& 0 & & I_n \\
& \vdots & I_{tn} & \vdots \\
 & I_n & & Z_x \\
 & & & & I_{un}
\end{bmatrix} 
\]
or of the form
\[
B = 
\begin{bmatrix}
I_{rn} \\
& Z_x & \cdots & I_n \\
& & I_{tn}\\
& 0 & \cdots & I_n \\
& & & & I_{un}
\end{bmatrix}B^\prime, \quad \mbox{and} \quad 
B^\mathcal{B}=(B^\prime)^\mathcal{B}
\begin{bmatrix}
I_{rn} \\
& Z_x & \cdots & I_n \\
& & I_{tn}\\
& 0 & \cdots & I_n \\
& & & & I_{un}
\end{bmatrix},
\]
for some block-identity matrices $I_{rn}$, $I_{tn}$ and $I_{un}$.
Therefore, \eqref{eq:aux_GFPR1} holds with $\Pi_\mathbf{c}^n=\Pi_{\boldsymbol\sigma}^n$ and $C = C'$. 
 Furthermore, if $\mathcal{Z}_w$ is a nonsingular matrix assignment for $\mathbf{t}_w$, then $B$ and $B^\mathcal{B}$ are nonsingular, since $B^\prime$ and $(B^\prime)^\mathcal{B}$ are. Recall that the elementary matrices $M_i(X)$, with $i\neq 0, k$ are nonsingular for all $X$.  
In this situation, the pencil $(\Pi_{\boldsymbol \sigma}^n)^\mathcal{B}\mathcal{L}(\lambda)\Pi_{\boldsymbol \sigma}^n$ is a strong linearization of $P(\lambda)$ by Theorem \ref{thm:first family}. 

Now, we prove parts (a), (b) and (c).
Recall that the $(k-x)$th and $(k-x+1)$th block-rows (resp. block-columns) of $\mathcal{L}'(\lambda)$ are two of its wing block-rows (resp. wing block-columns) relative to $(\boldsymbol \sigma, \boldsymbol \sigma)$.
Since $\Pi_\mathbf{c}^n=\Pi_{\boldsymbol\sigma}^n$, we also have that the wing block-rows (resp. wing block-columns) of $\mathcal{L}'(\lambda)$ and the wing block-rows (resp. wing block-columns) of $\mathcal{L}(\lambda)$ are  located at the same positions. 
Moreover, the wing block-rows (resp. wing block-columns) of $\mathcal{L}'(\lambda)$ and $\mathcal{L}(\lambda)$ other than those in  the $(k-x)$th and $(k-x+1)$th positions are equal.
Furthermore, the $(k-x+1)$th block-row (resp. block-column) of $\mathcal{L}(\lambda)\Pi_{\mathbf{c}}^n$ (resp. of $(\Pi_\mathbf{c}^n)^\mathcal{B}\mathcal{L}(\lambda)$) is of the form $-e_{i}^T\otimes I_n+\lambda e_{i+1}^T\otimes I_n$ (resp. $-e_{i}\otimes I_n+\lambda e_{i+1}\otimes I_n$), for some $1\leq i\leq s$, because, by the induction hypothesis, the $(k-x)$th block-row (resp. block-column) of $\mathcal{L}'(\lambda)\Pi_{\boldsymbol\sigma}^n$ is of this form.
Meanwhile, the $(k-x+1)$th block-row (resp. block-column) of $\mathcal{L}(\lambda)\Pi_\mathbf{c}^n$ (resp. of $(\Pi_\mathbf{c}^n)^\mathcal{B}\mathcal{L}(\lambda)$) is clearly not of this form generically.
Parts (a) and (b) follow from the preceding argument, together with
\[
\mathrm{heads}(\mathbf{t}'_w,\mathbf{w}_{k-1},\mathbf{c}_w, \rev(\mathbf{t}'_w),x)=\mathrm{heads}(\mathbf{t}'_w,\mathbf{w}_{k-1},\mathbf{c}_w, \rev(\mathbf{t}'_w))\cup\{x\},
\]
which follows from Remark \ref{rem: Type I and II heads},
and Lemma \ref{lemma:txSIP}. 
To prove part (c), just notice that pre- and post-multiplication by the matrix $M_x(Z_x)$ do not affect the first block-row and the first block-column of $\mathcal{L}'(\lambda)$, since $x\leq k-2$.

Assume, finally, that $x$ is a Type I index relative to  $(\mathbf{t}'_w,\mathbf{w}_{k-1},\mathbf{c}_{k-1}, \rev(\mathbf{t}'_w))$.
By Proposition \ref{prop:SIPheads} and Lemma \ref{lemma:txSIP}, the tuple $(\mathbf{t}'_w,\mathbf{w}_{k-1},\mathbf{c}_w, \rev(\mathbf{t}'_w),x-1)$ does not  satisfy the SIP.
Thus, by the inductive hypothesis, either the $(k-x+1)$th block-row (resp. block-column) of $\mathcal{L}'(\lambda)$ is one of its body block-rows (resp. body  block-columns) relative to $(\boldsymbol\sigma, \boldsymbol\sigma)$ or  the $(k-x+1)$th block-row (resp. block-column) of $\mathcal{L}'(\lambda)\Pi_{\boldsymbol\sigma}^n$ (resp. $(\Pi_{\boldsymbol\sigma}^n)^\mathcal{B}\mathcal{L}'(\lambda)$) is a wing block-row (resp. wing block-column) that is not of the form $-e_{i}^T \otimes I_n+\lambda e_{i+1}^T \otimes I_n$ (resp. $-e_{i}\otimes I_n+\lambda e_{i+1}\otimes I_n$), for some $1\leq i\leq s$. 
The proof that \eqref{eq:aux_GFPR1} holds in the case that the $(k-x+1)$th block-row (resp. block-column) is a wing block-row (resp. wing block-column) is very similar to the proof for the Type II index case in the paragraphs above.
So assume that the $(k-x+1)$th block-row (resp. block-column) of $L^\prime(\lambda)$ is a body block-row (resp. body block-column) relative to $(\boldsymbol{\sigma}, \boldsymbol{\sigma})$.
Since the $(k-x)$th block-row (resp. block-column) of $\mathcal{L}'(\lambda)$ is a wing block-row (resp. block-column) of $\mathcal{L}'(\lambda)$ relative to $(\boldsymbol{\sigma}, \boldsymbol{\sigma})$,  we can write 
\[
\mathcal{L}'(\lambda)=\begin{blockarray}{ccccc}
 & k-x & k-x+1 &  \\
\begin{block}{[c|cc|c]c}
  L_{1}'(\lambda) & \ell_1'(\lambda)^\mathcal{B} & \ell_2'(\lambda)^\mathcal{B} & L_{2}'(\lambda)^\mathcal{B} &  \\ \cline{1-4}
  \ell_1'(\lambda) & \phantom{\Big{(}}0_n\phantom{\Big{(}} & \ell_5'(\lambda) & \ell_3'(\lambda) & k-x \\
  \ell_2'(\lambda) & \phantom{\Big{(}}\ell_5'(\lambda)\phantom{\Big{(}} & \ell_6'(\lambda) & \ell_4'(\lambda) & k-x+1 \\ \cline{1-4}
  L_{2}'(\lambda) & \phantom{\Big{(}}\ell_3'(\lambda)^\mathcal{B}\phantom{\Big{(}} & \ell_4'(\lambda)^\mathcal{B} & L_2(\lambda) &  \\ 
\end{blockarray}.
 \]
 We recall again that  pre- and post-multiplying the pencil $\mathcal{L}'(\lambda)$ by $M_x(Z_x)$ affects only the $(k-x)$th and $(k-x+1)$th block-rows (resp. block-columns).
Therefore, the pencil $\mathcal{L}(\lambda)$ can be partitioned as
\[
\begin{blockarray}{ccccc}
 & k-x & k-x+1 &  \\
\begin{block}{[c|cc|c]c}
  L_{1}'(\lambda) & ( Z_x \ell_1'(\lambda)+ \ell_2'(\lambda))^\mathcal{B} & \ell_1'(\lambda)^\mathcal{B} & L_{2}'(\lambda)^\mathcal{B} &  \\ \cline{1-4}
  Z_x \ell_1'(\lambda)+ \ell_2'(\lambda) & \ell_5'(\lambda)Z_x+Z_x\ell_5'(\lambda)+\ell_6'(\lambda) & \phantom{\Big{(}}\ell_5'(\lambda)\phantom{\Big{(}} & Z_x\ell_3'(\lambda)+\ell_4'(\lambda) & k-x \\
  \ell_1'(\lambda) & \ell_5'(\lambda)^\mathcal{B} &\phantom{\Big{(}} 0 \phantom{\Big{(}}& \ell_3'(\lambda) & k-x+1 \\ \cline{1-4}
  L_{2}'(\lambda) & (Z_x\ell_3'(\lambda)+\ell_4'(\lambda))^\mathcal{B} & \phantom{\Big{(}}\ell_3'(\lambda)^\mathcal{B}\phantom{\Big{(}} &  L_2(\lambda) &  \\ 
\end{blockarray},
 \]
 where we have used the fact that the only non-zero block-entries of $\ell_1'(\lambda)$ and $\ell_3'(\lambda)$ are equal to $I_n$ or $\lambda I_n$ (this follows from parts (a) and (b) of the induction hypothesis).
Then, setting
\[
\mathbf{d} := (1:k-x-1,k-x+1,k-x,k-x+2:k),
\]
 we easily obtain that the pencil $(\Pi^n_\mathbf{d})^\mathcal{B}\mathcal{L}(\lambda)\Pi_\mathbf{d}^n$ can be partitioned as follows
 \[
\begin{blockarray}{ccccc}
 & k-x & k-x+1 &  \\
\begin{block}{[c|cc|c]c}
  L_{1}'(\lambda) & \ell_1'(\lambda)^\mathcal{B}& ( Z_x \ell_1'(\lambda)+ \ell_2'(\lambda))^\mathcal{B}  & L_{2}'(\lambda)^\mathcal{B} &  \\ \cline{1-4}
    \ell_1'(\lambda) & \phantom{\Big{(}} 0 \phantom{\Big{(}} & \ell_5'(\lambda)^\mathcal{B} & \ell_3'(\lambda) & k-x \\
  Z_x \ell_1'(\lambda)+ \ell_2'(\lambda) &  \phantom{\Big{(}}\ell_5'(\lambda)\phantom{\Big{(}} & \ell_5'(\lambda)Z_x+Z_x\ell_5'(\lambda)+\ell_6'(\lambda)& Z_x\ell_3'(\lambda)+\ell_4'(\lambda) & k-x+1 \\
 \cline{1-4}
  L_{2}'(\lambda) & \phantom{\Big{(}}\ell_3'(\lambda)^\mathcal{B}\phantom{\Big{(}} &  (Z_x\ell_3'(\lambda)+\ell_4'(\lambda))^\mathcal{B}&  L_2(\lambda) &  \\ 
\end{blockarray}.
 \]
Thus,
 \[
(\Pi_{\boldsymbol\sigma}^n)^\mathcal{B}(\Pi_\mathbf{d}^n)^\mathcal{B}\mathcal{L}(\lambda)\Pi_\mathbf{d}^n\Pi_{\boldsymbol\sigma}^n= 
\left[\begin{array}{c|c}
M(\lambda) & K_s(\lambda)^T(B^\prime)^\mathcal{B} \\ \hline
B^\prime K_s(\lambda) & 0
\end{array}\right]\in\langle\mathcal{O}_1^P\rangle,
\] 
where $M(\lambda)=M'(\lambda)+DK_s(\lambda)+K_s(\lambda)D^\mathcal{B}$, for some matrix $D$.
Therefore, \eqref{eq:aux_GFPR1} holds with $\Pi_\mathbf{c}^n=\Pi_\mathbf{d}^n\Pi_{\boldsymbol\sigma}^n$ and $B=B'$. 
 If $\mathcal{Z}_w$ is a nonsingular assignment for $\mathbf{t}_w$, then $B$ and $B^\mathcal{B}$ are nonsingular matrices, because $B^\prime$ and $(B^\prime)^\mathcal{B}$ are nonsingular, and  $(\Pi_{\mathbf{c}}^n)^\mathcal{B}\mathcal{L}(\lambda)\Pi_{\mathbf{c}}^n$ is a strong linearization of $P(\lambda)$ by Theorem \ref{thm:first family}.

We finish by showing that parts (a), (b) and (c) hold.
Recall that the $(k-x)$th  block-row (resp. block-column) of $\mathcal{L}'(\lambda)$ is one of its wing block-rows (resp. wing block-columns) relative to $(\boldsymbol\sigma, \boldsymbol\sigma)$, and that the $(k-x+1)$th block-row (resp. block-column) of $\mathcal{L}'(\lambda)$ is either one of its body block-rows (resp. body block-column)  or it is one of its wing block-rows (resp. wing block-columns) but not as those in part (b) (resp. part (a)).
Since $\Pi_\mathbf{c}^n=\Pi_\mathbf{d}^n\Pi_{\boldsymbol\sigma}^n$, we also have that the wing block-rows (resp. wing block-columns) of $\mathcal{L}'(\lambda)$  relative to $(\boldsymbol\sigma, \boldsymbol\sigma)$ and the wing block-rows (resp. wing block-columns) of $\mathcal{L}(\lambda)$ relative to $(\mathbf{c}, \mathbf{c})$ other than those in positions $(k-x)$th and $(k-x+1)$th are equal and located at the same positions. 
Moreover, the wing block-row (resp. wing block-column) of $\mathcal{L}'(\lambda)$   in the $(k-x)$th position equals the wing block-row (resp. wing block-column) of $\mathcal{L}(\lambda)$ in the $(k-x+1)$th position.
Then, parts (a) and (b) follow from the preceding facts, together with
\[
\mathrm{heads}(\mathbf{t}'_w,\mathbf{w}_{k-1},\mathbf{c}_{k-1}, \rev(\mathbf{t}'_w),x)=(\mathrm{heads}(\mathbf{t}'_w,\mathbf{w}_{k-1},\mathbf{c}_{k-1}, \rev(\mathbf{t}'_w))\cup\{x\})\setminus \{x-1\},
\]
which follows from Remark \ref{rem: Type I and II heads}, and Lemma \ref{lemma:txSIP}. 
To prove part (c), just notice again that pre- and post-multiplication by the matrix $M_x(Z_x)$ do not affect the first block-row and the first block-column of $\mathcal{L}'(\lambda)$, since $x\leq k-2$.
\end{proof}

\begin{remark}\label{remark:permutation1}
We note that part (c) in Theorem \ref{thm:aux_GFPR1} implies that the block-permutation $\Pi_{\mathbf{c}}^n$ is of the form $I_n\oplus \Pi_{\mathbf{\widetilde{c}}}^n$, for some permutation  $\mathbf{\widetilde{c}}$ of the set $\{1:k-1\}$.
\end{remark}

As a consequence of the previous theorem, we obtain Theorem \ref{thm:aux_GFPR2}, which shows a structural result for another subclass of block-symmetric GFPR.
In order to prove Theorem \ref{thm:aux_GFPR2}, we will make use of the following immediate property of elementary matrices:
\begin{equation}\label{eq:aux elementary matrix}
R_k M_{-i}(B)R_k= M_{k-i}(B), \quad \mbox{for $i=1:k$ and arbitrary $B$},
\end{equation}
where $R_k$ is the block sip-matrix \eqref{eq:sip matrix}.
\begin{theorem}\label{thm:aux_GFPR2}
Let $P(\lambda)=\sum_{i=0}^k A_i\lambda^i\in\mathbb{F}[\lambda]^{n\times n}$ with $k$ odd, let $s=(k-1)/2$, and let 
 \[
L_P(0,\emptyset,\mathbf{t}_v,\emptyset,\mathcal{Z}_v)= M_{\mathbf{t}_v}(\mathcal{Z}_v)(\lambda M^P_{\mathbf{v}_0}-M^P_{0})M^P_{-k+\mathbf{c}_{k-1}}M_{\rev(\mathbf{t}_v)}(\rev(\mathcal{Z}_v))
\]
be a block-symmetric GFPR associated with $P(\lambda)$.

Then, there exists a permutation $\Pi_\mathbf{c}^n$ such that
\begin{equation}\label{eq:aux_GFPR2}
 (\Pi_\mathbf{c}^n)^\mathcal{B} L_P(0,\emptyset,\mathbf{t}_v,\emptyset,\mathcal{Z}_v) \Pi_\mathbf{c}^n = 
\left[\begin{array}{c|c}
0 &  BK_s(\lambda) \\ \hline
K_s(\lambda)^TB^\mathcal{B}  & \phantom{\Big{(}} M(\lambda;P)+CK_s(\lambda)+K_s(\lambda)^TC^\mathcal{B} \phantom{\Big{(}}
\end{array}\right],
\end{equation}
for some matrices $B$ and $C$.
Moreover, the following statements hold:
\begin{itemize}
\item[\rm (a)] The last block-column of \eqref{eq:aux_GFPR2} is the last block-row of $(\Pi_\mathbf{c}^n)^\mathcal{B}L_P(0,\emptyset,\mathbf{t}_v,\emptyset,\mathcal{Z}_v)$.
\item[\rm (b)] The last block-row of \eqref{eq:aux_GFPR2} is the last block-column of  $L_P(0,\emptyset,\mathbf{t}_v,\emptyset,\mathcal{Z}_v)\Pi_\mathbf{c}^n$.
\item[\rm (c)] The block-entry of  $(\Pi_\mathbf{c}^n)^\mathcal{B} L_P(0,\emptyset,\mathbf{t}_v,\emptyset,\mathcal{Z}_v) \Pi_\mathbf{c}^n$ in position $(k,k)$ equals $\lambda A_1+A_0$.
\end{itemize}
Furthermore, if $\mathcal{Z}_v$ is a nonsingular matrix assignment for $\mathbf{t}_v$, then $B$ and $B^\mathcal{B}$ are nonsingular.
\end{theorem}
\begin{proof}
For simplicity, let  $L_P(\lambda):=L_P(0,\emptyset,\mathbf{t}_v,\emptyset,\mathcal{Z}_v)$.
Let $\widehat{P}(\lambda):=-\rev P(\lambda)$, and let us consider the pencil $\widehat{L}(\lambda):=R_k \rev(-L_P(\lambda))R_k $.
By using \eqref{eq:aux elementary matrix} together with $R_k^{-1}=R_k$ and taking into account that the indices in the  tuples $\mathbf{t}_v, \mathbf{v}_0, -k+\mathbf{c}_{k-1}$ are in $\{-k:-1\}$, we obtain without much difficulty
\begin{align*}
\widehat{L}(\lambda)=& 
R_kM_{\mathbf{t}_v}(\mathcal{Z}_v)R_k(\lambda R_kM^P_{0}R_k-R_k M^P_{\mathbf{v}_0}R_k)R_kM^P_{-k+\mathbf{c}_{k-1}}R_kR_kM_{\rev(\mathbf{t}_v)}(\rev(\mathcal{Z}_v))R_k= \\&
M_{\mathbf{t}_v+k}(\mathcal{Z}_v)(\lambda M_{-k}^{\widehat{P}}-M_{\mathbf{v}_0+k}^{\widehat{P}})M_{\mathbf{c}_{k-1}}^{\widehat{P}}M_{\rev(\mathbf{t}_v)+k}(\rev(\mathcal{Z}_v)),
\end{align*}
which is the block-symmetric GFPR  $L_{\widehat{P}}(k-1, \mathbf{t}_v+k, \emptyset, \mathcal{Z}_v, \emptyset)$ associated with $\widehat{P}(\lambda)$.
 Note that this pencil is of the kind  of  GFPR considered in Theorem \ref{thm:aux_GFPR1}, because if $(\mathbf{t}_v,\mathbf{v}_0,-k+\mathbf{c}_{k-1},\rev(\mathbf{t}_v))$ satisfies the SIP so does $(\mathbf{t}_v+k,\mathbf{v}_0+k,\mathbf{c}_{k-1},\rev(\mathbf{t}_v))+k)$.
Hence, by Theorem \ref{thm:aux_GFPR1}, there exists a permutation $\mathbf{c}^\prime$ of $\{1:k\}$ such that
\[
(\Pi_{\mathbf{c}^\prime}^n)^\mathcal{B}\widehat{L}(\lambda)\Pi_{\mathbf{c}^\prime}^n =
\left[\begin{array}{c|c} 
M(\lambda;\widehat{P})+ C^\prime K_s(\lambda)+K_s(\lambda)^T(C^\prime)^\mathcal{B} & K_s(\lambda)^T (B^\prime)^\mathcal{B} \\ \hline
B^\prime K_s(\lambda) & 0
\end{array}\right],
\]
for some matrices  $B'$ and $C'$. 
Let $\Pi_\mathbf{c}^n:=R_k\Pi_{\mathbf{c}^\prime}^nR_k$.
Taking into account that $L_P(\lambda)= \allowbreak- \rev(R_k \widehat{L}(\lambda) R_k)$, we, then, have
\begin{align*}
&(\Pi_{\mathbf{c}}^n)^\mathcal{B}L_P(\lambda)\Pi_{\mathbf{c}}^n=-\rev( R_k (\Pi_{\mathbf{c}^\prime}^n)^\mathcal{B}\widehat{L}(\lambda) \Pi_{\mathbf{c}^\prime}^n R_k ) = \\
&-\rev \left( R_k \left[\begin{array}{c|c} 
M(\lambda;\widehat{P})+ C^\prime K_s(\lambda)+K_s(\lambda)^T(C^\prime)^\mathcal{B} & K_s(\lambda)^T (B^\prime)^\mathcal{B} \\ \hline
B^\prime K_s(\lambda) & 0
\end{array}\right] R_k \right)=\\&
\left[ \begin{array}{c|c} 
0 & -\rev( R_s B^\prime K_s(\lambda)R_{s+1})\\ \hline
-\rev(R_s B^\prime K_s(\lambda)R_{s+1})^\mathcal{B} & \phantom{\Big{(}} -\rev(R_{s+1}(M(\lambda;\widehat{P})+ C^\prime K_s(\lambda)+K_s(\lambda)^T(C^\prime)^\mathcal{B})R_{s+1}) \phantom{\Big{)}}
\end{array}\right].
\end{align*}
Then, to prove the  first claim of the theorem, it suffices to notice that the following two equalities hold:
\[
-\rev( R_s B^\prime K_s(\lambda)R_{s+1}) = R_s B^\prime (-\rev( K_s(\lambda)R_{s+1}))=  BK_s(\lambda),
\]
where $B:=R_s B^\prime R_s$, and, 
\begin{align*}
-&\rev(R_{s+1}(M(\lambda;\widehat{P})+ C^\prime K_s(\lambda)+K_s(\lambda)^T(C^\prime)^\mathcal{B})R_{s+1}) \\
&=-\rev(R_{s+1}M(\lambda;\widehat{P})R_{s+1})
+ R_{s+1}C^\prime(-\rev(K_s(\lambda)R_{s+1})) \\
& \hspace{6cm}+(-\rev(K_s(\lambda)R_{s+1}))^\mathcal{B} (R_{s+1}C^\prime)^{\mathcal{B}}\\
&= M(\lambda;P)+CK_s(\lambda)+K_s(\lambda)^TC^\mathcal{B},
\end{align*}
where $C:=R_{s+1}C'R_s$. 

Now we prove parts (a), (b), and (c). 
Let $M^\prime(\lambda) := M(\lambda;\widehat{P})+ C^\prime K_s(\lambda)+K_s(\lambda)^T(C^\prime)^\mathcal{B}$ and $M(\lambda):= M(\lambda;P)+CK_s(\lambda)+K_s(\lambda)^TC^\mathcal{B}$.
Notice, first, that part (c) in Theorem \ref{thm:aux_GFPR1} implies that the first block-row  and the first block-column of the pencil $\widehat{L}(\lambda)$ are, respectively,  the first body block-row and the first  body block-column  of the pencil $ \widehat{L}(\lambda)$ relative to $(\mathbf{c'}, \mathbf{c'})$,  and that the block-entry in position $(1,1)$ of $M^\prime(\lambda)$ is $- \lambda A_{0}-A_1$. 
Since $M(\lambda)=-\rev(R_{s+1}M^\prime (\lambda)R_{s+1})$, the block-entry in position $(s+1, s+1)$ of $M(\lambda)$ is $\lambda A_1 + A_0$, which is part (c).  
Moreover, since $L_P(\lambda) = - \rev(R_k \widehat{L}(\lambda) R_k)$ and $( \Pi^n_\mathbf{c})^{\mathcal{B}}L_P(\lambda) \Pi^n_\mathbf{c} = - \rev (R_k (\Pi^n_{\mathbf{c}^\prime})^{\mathcal{B}} \widehat{L}(\lambda) \Pi^n_{\mathbf{c}^\prime} R_k)$, we deduce from part (c) in Theorem  \ref{thm:aux_GFPR1}  that the last block-column and the last block-row of $( \Pi^n_\mathbf{c})^{\mathcal{B}}L_P(\lambda) \Pi^n_\mathbf{c}$ are, respectively, the last body block-column and the last body block-row of $( \Pi^n_\mathbf{c})^{\mathcal{B}}L_P(\lambda)$ and $L_P(\lambda)\Pi^n_\mathbf{c}$.
Thus, claims (a) and (b) follow.

 Finally, notice that if $\mathcal{Z}_v$ is a nonsingular matrix assignment for $\mathbf{t}_v$, then $\rev(\mathcal{Z}_v)$ is a nonsingular matrix assignment for $\rev(\mathbf{t}_v)$.
Hence, if $\mathcal{Z}_v$ is a nonsingular matrix assignment for $\mathbf{t}_v$, then $B^\prime$ and $(B^\prime)^\mathcal{B}$ are nonsingular matrices by Theorem \ref{thm:aux_GFPR1}.
Thus,  $B=R_sB^\prime R_s$ and $B^\mathcal{B}=R_s(B^\prime)^\mathcal{B}R_s$ are nonsingular matrices, because $R_s$ is nonsingular.
\end{proof}

\begin{remark}\label{remark:permutation2}
We note that  Theorem \ref{thm:aux_GFPR2} implies that  $\Pi_{\mathbf{c}}^n$ is of the form $\Pi_{\mathbf{\widetilde{c}}}^n\oplus I_n$, for some permutation $\mathbf{\widetilde{c}}$ of the set $\{1:k-1\}$.
\end{remark}

 \subsubsection{Auxiliary results for the even degree case}

We first prove that  Theorem \ref{thm:main_GFPR} holds for the block-symmetric GFPR called the simple FPR with parameter $k-1$ in \cite{Her-GFPR}.
\begin{theorem}\label{aux;simple2}
Let $P(\lambda)=\sum_{i=0}^k A_i\lambda^i\in\mathbb{F}[\lambda]^{n\times n}$ of even degree $k$, let $P_{k-1}(\lambda)= \sum_{i=0}^{k-1} A_{i+1}\lambda^i$, let $s=(k-2)/2$, and let $F_k(\lambda)=(\lambda M^P_{-k}-M^P_{\mathbf{w}_{k-1}})M^P_{\mathbf{c}_{k-1}}$.
Let $\mathbf{c}$ be the permutation of $\{1:k\}$ given by $(1,2,4,\hdots,k,3,5,\hdots,k-1)$. 
Then,
\begin{equation}\label{eq:main_FPR_even}
(\Pi_{\mathbf{c}}^n)^\mathcal{B} F_k(\lambda) \Pi_{\mathbf{c}}^n = \left[\begin{array}{c|c:c}
   M(\lambda;P_{k-1})+BK_s(\lambda)+K_s(\lambda)B^\mathcal{B} & \begin{matrix} 0 \\ A_0 \end{matrix}  & K_s(\lambda)^T  \\ \hdashline
   \begin{matrix} 0 & A_0 \end{matrix}  & \lambda A_0 & 0\\ \hline
  K_s(\lambda)& 0 &  0
   \end{array}\right]\in\langle \mathcal{E}_1^P\rangle,
\end{equation}
for some matrix $B$.
Moreover, the following statements hold:
\begin{itemize}
\item[\rm (a)] The wing block-columns of  $(\Pi_{\mathbf{c}}^n)^{\mathcal{B}} F_{k}(\lambda)$ relative to $(\mathbf{id}, \mathbf{c})$ that are  of the form $-e_i \otimes I_n + \lambda e_{i+1} \otimes I_n$, for $1\leq i \leq s$, are  located in  positions $k-j$, where $j\in \{0:k-2\}$ and $(\mathbf{w}_{k-1},\mathbf{c}_{k-1}, j)$ satisfies the SIP.

\item[\rm (b)] The wing block-rows of $F_{k}(\lambda)\Pi_{\mathbf{c}}^n$ relative to $(\mathbf{c}, \mathbf{id})$ that are of the form $-e_i^T \otimes I_n + \lambda e_{i+1}^T \otimes I_n$, for $1\leq i \leq s$,  and are  located in  positions $k-j$, where $j\in \{0:k-2\}$ and $(\mathbf{w}_{k-1},\mathbf{c}_{k-1}, j)$ satisfies the SIP.

\item[\rm (c)]  The first block-row and the first block-column of $F_{k}(\lambda)$ are, respectively, the first body block-row and the first body block-column of $F_{k}(\lambda)$ relative to $(\mathbf{c}, \mathbf{c})$.
 Moreover, the block-entry of $(\Pi_{\mathbf{c}}^n)^\mathcal{B} F_k(\lambda) \Pi_{\mathbf{c}}^n$ in position $(1,1)$ equals $\lambda A_k + A_{k-1}$.
 
 \item[\rm (d)] The exceptional block-column of  $(\Pi_{\mathbf{c}}^n)^{\mathcal{B}} F_{k}(\lambda)$ relative to $(\mathbf{id}, \mathbf{c})$ is located in position  $k$, and the exceptional block-row of $F_k(\lambda)\Pi_{\mathbf{c}}^n$ relative to $(\mathbf{c}, \mathbf{id})$ is located in position $k$.
\end{itemize}
Furthermore, if $A_0$ is nonsingular, the pencil $(\Pi_{\mathbf{c}}^n)^\mathcal{B} F_k(\lambda) \Pi_{\mathbf{c}}^n$ is a strong linearization of $P(\lambda)$.
\end{theorem}
\begin{proof} 
We begin by recalling the block-structure of $F_k(\lambda)$ when $k$ is even \cite[Section 8]{Her-GFPR}.
Notice that $F_4(\lambda)$ is partitioned to show that it is an extended block Kronecker pencil.
\[
F_4(\lambda) = \left[ \begin{array}{cccc} \lambda A_4+A_3 & A_2 & -I_n & 0 \\ A_2 & -\lambda A_2+A_1 & \lambda I_n & A_0 \\ -I_n & \lambda I_n & 0 & 0 \\ 0 & A_0 & 0 & -\lambda A_0 \end{array} \right].
\]
For $k \geq 6$ even, we have that $F_k(\lambda)$ is of the form
\begin{equation}\label{Feven}
	\begin{bmatrix}
		\lambda A_k + A_{k-1} & A_{k-2} & -I_n &&&&&&\\
		A_{k-2} & A_{k-3} - \lambda A_{k-2} & \lambda I_n & A_{k-4} & -I_n &&&&\\
		-I_n & \lambda I_n & 0 & 0 &&&&&\\
		& A_{k-4} & 0 & \ddots & \ddots &&&&&\\
		& -I_n && \ddots & A_3 - \lambda A_4 & \lambda I_n & A_2 & -I_n &&\\
		&&&& \lambda I_n & 0 & 0 &&\\
		&&&& A_2 & 0 & A_1 - \lambda A_2 & \lambda I_n & A_0\\
		&&&& -I_n && \lambda I_n & 0 & 0\\
		&&&&&& A_0 & 0 & -\lambda A_0\\
	\end{bmatrix}.
\end{equation}
where the empty spaces denote zero blocks.

For $k=2$ or $k=4$, the theorem can be easily checked.
So let us assume $k\geq 6$.  
From the explicit block-structure of the pencil $F_k(\lambda)$ it is not difficult to check that  parts (c) and (d) hold and $( \Pi_c^n)^\mathcal{B} F_k(\lambda) \Pi_c^n$ is of the form \eqref{eq:main_FPR_even}, where
\[
B=\left[ \begin{array}{cccc} 0 & 0 & \cdots & 0 \\ -A_{k-2} & 0 & \cdots & 0 \\ 0 & -A_{k-4} & \cdots & 0\\ \vdots & \vdots & \ddots & \vdots\\ 0 & 0 & \cdots & -A_2 \end{array} \right].
\]

Parts (a) and (b) follow from checking directly that the wing block-columns of the form  $-e_i \otimes I_n + \lambda e_{i+1} \otimes I_n$ and the wing block-rows of $F_k(\lambda)$ relative to $(\mathbf{c}, \mathbf{c})$ are in positions $k-j\in\{3,5,\hdots,k-1\}$, which, in turn, by Lemmas \ref{wcSIP} and \ref{lemma:txSIP}, correspond to those values of $j\in \{0:k-2\}$ such that $(\mathbf{w}_{k-1},\mathbf{c}_{k-1}, j)$ satisfies the SIP.
Note that $\mathbf{csf}(\mathbf{w}_{k-1}, \mathbf{c}_{k-1})=(k-2:k-1, k-4:k-2, \ldots, 0:2, 0) $. 

Finally, from Theorem \ref{thm:third family}, it follows that the pencil \eqref{eq:main_FPR_even} is a strong linearization of $P(\lambda)$ when $A_0$ is nonsingular.
\end{proof}

As in the odd case, we next give the main result for some particular block-symmetric GFPR . 
We also prove  some structural information concerning  the block-rows and block-columns of these particular GFPR.
\begin{theorem}\label{thm:aux_GFPR3}
Let $P(\lambda)=\sum_{i=0}^k A_i\lambda^i\in\mathbb{F}[\lambda]^{n\times n}$ of even degree $k$, let $s=(k-2)/2$,  and let $L_P(k-1,\mathbf{t}_w,\emptyset,\mathcal{Z}_w,\emptyset)$ be the block-symmetric GFPR associated with $P(\lambda)$ given by
\[
M_{\mathbf{t}_w}(\mathcal{Z}_w)(\lambda M^P_{-k}-M^P_{\mathbf{w}_{k-1}})M^P_{\mathbf{c}_{k-1}}M_{\rev(\mathbf{t}_w)}(\rev (\mathcal{Z}_w)).
\]
Then, there exists a block-permutation $\mathbf{c}$ of $\{1:k\}$ such that
 \begin{equation}
(\Pi_{\mathbf{c}}^n)^\mathcal{B} L_P(k-1,\mathbf{t}_w,\emptyset,\mathcal{Z}_w,\emptyset) \Pi_{\mathbf{c}}^n 
\in\langle \mathcal{E}_1^P\rangle
\end{equation}
is as in \eqref{eq:third_family}.

Moreover, the following statements hold:
\begin{itemize}
\item[\rm (a)] The wing block-columns of  $(\Pi_{\mathbf{c}}^n)^{\mathcal{B}} L_P(k-1,\mathbf{t}_w,\emptyset,\mathcal{Z}_w,\emptyset)$ relative to $(\mathbf{\mathrm{id}},\mathbf{c})$ that are  of the form $-e_i \otimes I_n + \lambda e_{i+1} \otimes I_n$, for $1\leq i \leq s$, are  located in  positions $k-j$, where $j\in \{0:k-2\}$ and $(\mathbf{t}_w,\mathbf{w}_{k-1},\mathbf{c}_{k-1},\rev(\mathbf{t}_w), j)$ satisfies the SIP.

\item[\rm (b)] The wing block-rows of  $L_P(k-1,\mathbf{t}_w,\emptyset,\mathcal{Z}_w,\emptyset) \Pi_{\mathbf{c}}^n$ relative to $(\mathbf{c}, \mathbf{\mathrm{id}})$  that are of the form $-e_i^T \otimes I_n + \lambda e_{i+1}^T \otimes I_n$, for $1\leq i \leq s$, are  located in  positions $k-j$, where $j\in \{0:k-2\}$ and $(\mathbf{t}_w,\mathbf{w}_{k-1},\mathbf{c}_{k-1},\rev(\mathbf{t}_w), j)$ satisfies the SIP.
\item[\rm (c)]  The first block-row and the first block-column of   $L_P(k-1,\mathbf{t}_w,\emptyset,\mathcal{Z}_w,\emptyset)$ are, respectively, the first body block-row and the first body block-column of $L_P(k-1,\mathbf{t}_w,\emptyset,\mathcal{Z}_w,\emptyset)$ relative to $(\mathbf{c},\mathbf{c})$.
 Moreover, the block-entry of  $(\Pi_{\mathbf{c}}^n)^\mathcal{B} L_P(k-1,\mathbf{t}_w,\emptyset,\allowbreak\mathcal{Z}_w,\emptyset) \Pi_{\mathbf{c}}^n$ in position $(1,1)$ equals $\lambda A_k + A_{k-1}$.
\end{itemize}

Furthermore, if $\mathcal{Z}_w$ is a nonsingular matrix assignment for $\mathbf{t}_w$, then $C$, $C^{\mathcal{B}}$, $D$ and $D^{\mathcal{B}}$ in \eqref{eq:third_family} are nonsingular.
If, additionally, $A_0$ is nonsingular,  $(\Pi_{\mathbf{c}}^n)^\mathcal{B} L_P(k-1,\mathbf{t}_w,\emptyset,\mathcal{Z}_w,\emptyset) \Pi_{\mathbf{c}}^n$ is a strong linearization of $P(\lambda)$. 
\end{theorem}
\begin{proof}
When the tuple $\mathbf{t}_w$ is empty, the results follow from Theorem \ref{aux;simple2}.
When the tuple $\mathbf{t}_w$ is not empty, the results follow by an induction argument on the number of indices in $\mathbf{t}_w$ almost identical to the one used for proving Theorem \ref{thm:aux_GFPR1}, so we omit it. 
\end{proof}
\begin{remark}\label{remark:permutation3}
We note that part (c) in Theorem \ref{thm:aux_GFPR3} implies that the block-permutation $\Pi_{\mathbf{c}}^n$ is of the form $I_n\oplus \Pi_{\mathbf{\widetilde{c}}}^n$, for some permutation  $\mathbf{\widetilde{c}}$ of the set $\{1:k-1\}$.
\end{remark}

As a consequence of Theorem \ref{thm:aux_GFPR3}, we obtain  Theorem \ref{thm:aux_GFPR4}, which is a structural result for another subclass of block-symmetric GFPR.

\begin{theorem}\label{thm:aux_GFPR4}
Let $P(\lambda)=\sum_{i=0}^k A_i\lambda^i\in\mathbb{F}[\lambda]^{n\times n}$ of even degree $k$, let $s=(k-2)/2$, and let  
\[
L_P(0,\emptyset,\mathbf{t}_{v},\emptyset,\mathcal{Z}_{v})= M_{\mathbf{t}_{v}}(\mathcal{Z}_{v})(\lambda M^P_{\mathbf{v_0}}-M_{0}^P)M^P_{-k+\mathbf{c}_{k-1}}M_{\rev(\mathbf{t}_v)}(\rev(\mathcal{Z}_{v})),
\]
be a block-symmetric GFPR associated with $P(\lambda)$.
 Then, there exists a block-permutation matrix $\Pi_\mathbf{c}^n$ such that  $(\Pi_\mathbf{c}^n)^\mathcal{B} L_P(0,\emptyset,\mathbf{t}_{v},\emptyset,\mathcal{Z}_{v})\Pi_\mathbf{c}^n$ is the pencil
\begin{equation}\label{eq:aux_GFPR4} 
\left[\begin{array}{c|c|c}
0 & 0 & DK_s(\lambda) \\ \hline
0 & \phantom{\Big{(}} -A_k \phantom{\Big{)}} & \begin{bmatrix} \lambda A_k & 0 \end{bmatrix}+CK_s(\lambda) \\ \hline
K_s(\lambda)^TD^{\mathcal{B}} & \begin{bmatrix} \lambda A_k \\ 0 \end{bmatrix} + K_s(\lambda)^TC^{\mathcal{B}} & M(\lambda;P^{k-1})+BK_s(\lambda)+K_s(\lambda)^TB^{\mathcal{B}}
\end{array}\right],
\end{equation}
where $P^{k-1}(\lambda)$ is defined in \eqref{eq:Pd},  for some matrices $B$, $C$, and $D$.
Moreover, the following statements hold:
\begin{itemize}
\item[\rm (a)] The last block-column of \eqref{eq:aux_GFPR4} is the last block-row of $(\Pi_\mathbf{c}^n)^\mathcal{B}L_P(0,\emptyset,\mathbf{t}_{v},\emptyset,\mathcal{Z}_{v})$.
\item[\rm (b)] The last block-row of \eqref{eq:aux_GFPR4} is the last block-column of $L_P(0,\emptyset,\mathbf{t}_{v},\emptyset,\mathcal{Z}_{v})\Pi_\mathbf{c}^n$.
\item[\rm (c)] The block-entry of $(\Pi_\mathbf{c}^n)^\mathcal{B}L_P(0,\emptyset,\mathbf{t}_{v},\emptyset,\mathcal{Z}_{v})\Pi_\mathbf{c}^n$ in position $(k,k)$ equals $\lambda A_1+A_0$.
\end{itemize}
Furthermore, if $\mathcal{Z}_v$ is a nonsingular matrix assignment for $\mathbf{t}_v$, then $C$, $C^{\mathcal{B}}$, $D$ and $D^{\mathcal{B}}$ are nonsingular.
If, additionally, $A_k$ is nonsingular, $(\Pi_\mathbf{c}^n)^\mathcal{B}L_P(0,\emptyset,\mathbf{t}_{v},\emptyset,\mathcal{Z}_{v})\Pi_\mathbf{c}^n$ is a strong linearization of $P(\lambda)$.
\end{theorem}
\begin{proof}
The proof follows almost identically that of Theorem \ref{thm:aux_GFPR2}, so it is only outlined.
By using the reversal operation and the sip matrix $R_k$, together with \eqref{eq:aux elementary matrix}, as it was done in the proof of Theorem \ref{thm:aux_GFPR2},  the pencil $L_P(0,\emptyset,\mathbf{t}_{v},\emptyset,\mathcal{Z}_{v})$ can be transformed into one of the block-symmetric GFPR considered in Theorem \ref{thm:aux_GFPR3}.
Then, applying the results in  Theorem \ref{thm:aux_GFPR3} to this new pencil and reversing the  operations performed on $L_P(0,\emptyset,\mathbf{t}_{v},\emptyset,\mathcal{Z}_{v})$,  the desired results can be obtained.
\end{proof}

\begin{remark}\label{remark:permutation4}
We note that  Theorem \ref{thm:aux_GFPR4} implies that  $\Pi_{\mathbf{c}}^n$ is of the form $\Pi_{\mathbf{\widetilde{c}}}^n\oplus I_n$, for some permutation $\mathbf{\widetilde{c}}$ of the set $\{1:k-1\}$.
\end{remark}

\subsubsection{Proof of Theorem \ref{thm:main_GFPR}}

The main tools for proving Theorem \ref{thm:main_GFPR} are Theorems \ref{thm:aux_GFPR1}, \ref{thm:aux_GFPR2}, \ref{thm:aux_GFPR3} and \ref{thm:aux_GFPR4}, together with Lemma \ref{lem:splitting}.

\begin{lemma}\label{lem:splitting}
Let $P(\lambda)=\sum_{i=0}^k A_i\lambda^i\in\mathbb{F}[\lambda]^{n\times n}$ be a matrix polynomial of degree $k$, and let  $h\in \{0:k-1\}$.
Let $L_P(h,\mathbf{t}_w,\mathbf{t}_v,\mathcal{Z}_w,\mathcal{Z}_v)$ be a block-symmetric GFPR.
Then, this pencil can be partitioned as
\begin{equation}\label{partition}
L_P(h,\mathbf{t}_w,\mathbf{t}_v,\mathcal{Z}_w,\mathcal{Z}_v) = \left[ \begin{array}{c|c|c} D_\mathbf{v}(\lambda) & y_\mathbf{v}(\lambda) & 0 \\ \hline x_\mathbf{v}(\lambda) & \lambda A_{h+1}+A_h & x_\mathbf{w}(\lambda) \\ \hline 0 & y_\mathbf{w}(\lambda) & D_\mathbf{w}(\lambda)   \end{array} \right ],
\end{equation}
where $ D_\mathbf{w}(\lambda) \in \mathbb{F}[\lambda]^{nh\times nh}$, $D_\mathbf{v}(\lambda)\in \mathbb{F}[\lambda]^{n(k-h-1) \times n(k-h-1)}$, $x_\mathbf{v}(\lambda)\in \mathbb{F}[\lambda]^{n \times n(k-h-1)}$, $x_\mathbf{w}(\lambda) \in \mathbb{F}[\lambda]^{n \times nh}$, $y_\mathbf{v}(\lambda)\in \mathbb{F}[\lambda]^{n(k-h-1)\times n}$ and $y_\mathbf{w}(\lambda) \in \mathbb{F}[\lambda]^{nh \times n}$,
and where the pencils
\[
F(\lambda) := \left[ \begin{array}{cc} \lambda A_{h+1}+A_h & x_\mathbf{w}(\lambda) \\ y_\mathbf{w}(\lambda) & D_\mathbf{w}(\lambda)\end{array} \right] \quad \mbox{and} \quad G(\lambda) := \left[ \begin{array}{cc} D_\mathbf{v}(\lambda) & y_\mathbf{v}(\lambda) \\ x_\mathbf{v}(\lambda) & \lambda A_{h+1}+A_h \end{array} \right]
 \]
are block-symmetric GFPR associated with $Q(\lambda):= P^{h+1}(\lambda)= \lambda^{h+1} A_{h+1} + \lambda^h A_h + \cdots + \lambda A_1 + A_0$ and $Z(\lambda):=P_{k-h}(\lambda)= \lambda^{k-h} A_{k} + \lambda^{k-h-1}  A_{k-1} + \cdots +  \lambda A_{h+1}+ A_h$, respectively.
More precisely, we have
\[
F(\lambda) = M_{\mathbf{t}_w}(\mathcal{Z}_w) ( \lambda M^Q_{-h-1} - M^Q_{\mathbf{w}_h})M^Q_{\mathbf{c}_h} M_{\rev(\mathbf{t}_w)}(\rev(\mathcal{Z}_w))
\]
and
\[
G(\lambda) = M_{\mathbf{t}_v}(\mathcal{Z}_v) ( \lambda M^Z_{\mathbf{v}_h} - M^Z_{0}) M^Z_{\mathbf{c}_{-k+h}} M_{\rev(\mathbf{t}_v)}(\rev(\mathcal{Z}_v)).
\]
\end{lemma}
\begin{proof}
The result follows by combining the partition in \cite[Lemma 4.34]{canonical_Fiedler} with Theorems \ref{thm:aux_GFPR2} and \ref{thm:aux_GFPR4}, which imply that the upper-left block-entry and the bottom-right block-entry of, respectively, $F(\lambda)$ and $G(\lambda)$ are both equal to $\lambda A_{h+1}+A_h$.
\end{proof}

We are finally in a position to prove Theorem \ref{thm:main_GFPR}.

\begin{proof}{\rm (\bf of Theorem \ref{thm:main_GFPR})}
We have to distinguish four cases, namely, (i) $k$ odd and $h$ even; (ii) $k$ and $h$ odd;  (iii) $k$ even and $h$ odd; and (iv) $k$ and $h$ even.
The main ideas and steps for proving Theorem \ref{thm:main_GFPR} in each of these four cases are the same.
For this reason, we only prove the most difficult case, which turns out to be case (ii).
The  proofs for the remaining cases are just outlined at the end, leaving the details to the interested reader.

Let us assume that $k$ and $h$ are odd.
For simplicity, instead of writing $L_P(h,\mathbf{t}_w,\mathbf{t}_v,\mathcal{Z}_w,\mathcal{Z}_v)$, we write $L_p(\lambda)$.
The first goal, then, is to show that there exists a block-permutation matrix $\Pi_{\mathbf{c}}^n$ such that 
\begin{equation}\label{eq:goal}
(\Pi_{\mathbf{c}}^n)^{\mathcal{B}} L_P(\lambda)\Pi_{\mathbf{c}}^n\in
\langle \mathcal{O}_2^P \rangle.
\end{equation}

By Lemma \ref{lem:splitting} applied  to the pencil $L_P(\lambda)$,
we can partition $L_p(\lambda)$ as in \eqref{partition},
where the pencil
\[
F(\lambda) := \left[ \begin{array}{cc} \lambda A_{h+1}+A_h & x_\mathbf{w}(\lambda) \\ y_\mathbf{w}(\lambda) & D_\mathbf{w}(\lambda)\end{array} \right]=
M_{\mathbf{t}_w}(\mathcal{Z}_w) ( \lambda M^Q_{-h-1} - M^Q_{\mathbf{w}_h})M_{\mathbf{c}_w} M_{\rev(\mathbf{t}_w)}(\rev(\mathcal{Z}_w))
\]
is a block-symmetric GFPR associated with the matrix polynomial $Q(\lambda):= \lambda^{h+1} A_{h+1} + \lambda^h A_h + \cdots + \lambda A_1 + A_0$, and where the pencil
\[
G(\lambda) := \left[ \begin{array}{cc} D_\mathbf{v}(\lambda) & y_\mathbf{v}(\lambda) \\ x_\mathbf{v}(\lambda) & \lambda A_{h+1}+A_h \end{array} \right]=
M_{\mathbf{t}_v}(\mathcal{Z}_v) ( \lambda M^Z_{\mathbf{v}_h} - M^Z_{0}) M_{\mathbf{c}_v} M_{\rev(\mathbf{t}_v)}(\rev(\mathcal{Z}_v))
\]
is a block-symmetric GFPR associated with the matrix polynomial $Z(\lambda):= \lambda^{k-h} A_{k} + \lambda^{k-h-1}  A_{k-1} + \cdots +  \lambda A_{h+1}+ A_h$.
Notice that $F(\lambda)$ is one of the block-symmetric GFPR considered in Theorem \ref{thm:aux_GFPR3}, while  $G(\lambda)$ is one of the block-symmetric GFPR considered in Theorem \ref{thm:aux_GFPR4}. 

Let $s_1=(h-1)/2$ and $s_2=(k-h-2)/2$.
From Theorem \ref{thm:aux_GFPR3}, together with Remark  \ref{remark:permutation3} and Definition \ref{def:third_family}, we obtain that there exists a block-permutation matrix $I_n\oplus \Pi_{\mathbf{c}_1}^n$ such that $(I_n\oplus \Pi_{\mathbf{c}_1}^n)^\mathcal{B}F(\lambda) (I_n\oplus \Pi_{\mathbf{c}_1}^n)=$
\[
\left[\begin{array}{c|c:c}
   M(\lambda;Q_h)+BK_{s_1}(\lambda)+K_{s_1}(\lambda)B^\mathcal{B} & \begin{bmatrix} 0 \\ A_0 \end{bmatrix} + K_{s_1}(\lambda)^T C^\mathcal{B} & K_{s_1}(\lambda)^T D^\mathcal{B} \\ \hdashline
   \begin{bmatrix} 0 & A_0 \end{bmatrix} + CK_{s_1}(\lambda) & -\phantom{\Big{(}} \lambda A_0 \phantom{\Big{)}} & 0\\ \hline
  DK_{s_1}(\lambda)& 0 &  0
   \end{array}\right],
\]
for some matrices $B$, $C$ and $D$.
Additionally, from Theorem \ref{thm:aux_GFPR4}, together with Remark \ref{remark:permutation4}, we obtain that there exists a block-permutation matrix $\Pi_{\mathbf{c}_2}^n\oplus I_n$ such that $(\Pi_{\mathbf{c}_2}^n\oplus I_n)^\mathcal{B}G(\lambda) (\Pi_{\mathbf{c}_2}^n\oplus I_n)=$
\[
\left[\begin{array}{c|c|c}
0 & 0 & D^\prime K_{s_2}(\lambda) \\ \hline
0 & \phantom{\Big{(}} -A_k  \phantom{\Big{)}} & \begin{bmatrix} \lambda A_k & 0 \end{bmatrix}+C^\prime K_{s_2}(\lambda) \\ \hline
K_{s_2}(\lambda)^T(D^\prime)^{\mathcal{B}} & \begin{bmatrix} \lambda A_k \\ 0 \end{bmatrix} + K_{s_2}(\lambda)^T(C^\prime)^{\mathcal{B}} & M(\lambda;Z^{k-h-1})+B^\prime K_{s_2}(\lambda)+K_{s_2}(\lambda)^T(B^\prime)^{\mathcal{B}}
\end{array}\right],
\]
for some matrices $B^\prime$, $C^\prime$, and $D^\prime$.

Next, let us introduce the notation 
\[
\left[\begin{array}{c:c}
T_s(\lambda) & t_s(\lambda)
\end{array}\right] := \left[\begin{array}{c|ccc:c}
-I_n & \lambda I_n & 0 & 0 \\
 & -I_n & \lambda I_n \\
 & & \ddots & \ddots \\
  & & & -I_n & \lambda I_n
\end{array}\right]=: \left[\begin{array}{c|c}
r_s(\lambda) & R_s(\lambda)
\end{array}\right],
\]
where $t_s(\lambda)$ and $r_s(\lambda)$ are of size $sn\times n$. We also introduce the following  notation 

\[
M(\lambda;Q_h)+BK_{s_1}(\lambda)+K_{s_1}(\lambda)B^\mathcal{B}=: 
\begin{bmatrix}
\lambda A_{h+1}+A_{h} & m_1^Q(\lambda) \\
 m_1^Q(\lambda)^{\mathcal{B}} & \widehat{M}^Q(\lambda)
\end{bmatrix}
\]
and 
\[
 M(\lambda;Z^{k-h-1})+B^\prime K_{s_2}(\lambda)+K_{s_2}(\lambda)^T(B^\prime)^{\mathcal{B}}=:
\begin{bmatrix}
\widehat{M}^Z(\lambda) & m_1^Z(\lambda) \\
m_1^Z(\lambda)^{\mathcal{B}} & \lambda A_{h+1}+A_{h}
\end{bmatrix}.
\]
Then, omitting the dependence on $\lambda$ for lack of space,  notice that the pencil $(\Pi_{\mathbf{c}_2}^n\oplus I_n\oplus \Pi_{\mathbf{c}_1})^\mathcal{B}L_P(\lambda) (\Pi_{\mathbf{c}_2}^n\oplus I_n\oplus \Pi_{\mathbf{c}_1})=$
{\small
\setlength{\arraycolsep}{2.5pt}
 \[
\begin{bmatrix}
0 & 0 & D^\prime T_{s_2} & D^\prime t_{s_2} & 0 & 0 & 0 \\
0 & -A_k & \begin{bmatrix} \lambda A_k & 0 \end{bmatrix} + C^\prime T_{s_2} & C^\prime t_{s_2} & 0 & 0 & 0\\
(D^\prime T_{s_2})^\mathcal{B} & \begin{bmatrix} \lambda A_k \\ 0 \end{bmatrix} + (C^\prime T_{s_2})^\mathcal{B} & \widehat{M}^Z & m_1^Z & 0 & 0 & 0 \\
(D^\prime t_{s_2})^\mathcal{B} & (C^\prime t_{s_2})^\mathcal{B} & (m_1^Z)^\mathcal{B} & \lambda A_{h+1}+A_{h} & m_1^Q & (Cr_{s_1})^\mathcal{B} & (Dr_{s_1})^\mathcal{B} \\
0 & 0 & 0 & (m_1^Q)^\mathcal{B} & \widehat{M}^Q & \begin{bmatrix} 0 \\ A_0 \end{bmatrix}+ (CR_{s_1})^\mathcal{B} & (DR_{s_1})^\mathcal{B} \\
0 & 0 & 0 & 0 & \begin{bmatrix} 0 & A_0 \end{bmatrix}+ Cr_{S_1} & -\lambda A_0 & 0 \\
0 & 0 & 0 & Dr_{s_1} & DR_{s_1} & 0 & 0
\end{bmatrix},
\]}%
is block-permutationally congruent to
{\small
\setlength{\arraycolsep}{2.5pt}
\[
\left[\begin{array}{c:ccc:c:cc}
 -A_k & \begin{bmatrix} \lambda A_k & 0 \end{bmatrix} + C^\prime T_{s_2} & C^\prime t_{s_2} & 0 & 0 & 0 & 0 \\ \hdashline
 \begin{bmatrix} \lambda A_k \\ 0 \end{bmatrix} + (C^\prime T_{s_2})^\mathcal{B} & \widehat{M}^Z & m_1^Z & 0 & 0 & (D^\prime T_{s_2})^\mathcal{B} & 0 \\
(C^\prime t_{s_2})^\mathcal{B} & (m_1^Z)^\mathcal{B} & \lambda A_{h+1}+A_{h} & m_1^Q & (Cr_{s_1})^\mathcal{B} & (D^\prime t_{s_2})^\mathcal{B} & (Dr_{s_1})^\mathcal{B} \\
0 & 0 & (m_1^Q)^\mathcal{B} & \widehat{M}^Q & \begin{bmatrix} 0 \\ A_0 \end{bmatrix}+ (CR_{s_1})^\mathcal{B} & 0 & (DR_{s_1})^\mathcal{B} \\ \hdashline
0 & 0 & Cr_{s_1} & \begin{bmatrix} 0 & A_0 \end{bmatrix} + CR_{s_1} & -\lambda A_0 & 0 & 0  \\ \hdashline
0 & D^\prime T_{s_2} & D^\prime t_{s_2} & 0 & 0 & 0 & 0\\
0 & 0 & Dr_{s_1} & DR_{s_1} & 0 & 0 & 0
 \end{array}\right].
\]}%

To finish the proof, it suffices to check that the above pencil belongs to $\langle \mathcal{O}_2^P \rangle$.
Hence, we have to analyze the different blocks  highlighted by the dash lines.
Let $s=(k-1)/2$. 
First, notice
\begin{equation}\label{eq:auxD}
\begin{bmatrix}
D^\prime T_{s_2}(\lambda) & D^\prime t_{s_2}(\lambda) & 0 \\
0 & Dr_{s_1}(\lambda) & DR_{s_1}(\lambda)
\end{bmatrix}=
\begin{bmatrix}
D^\prime & 0 \\ 0 & D
\end{bmatrix}
K_{s-1}(\lambda)=: \widetilde{D}K_{s-1}(\lambda).
\end{equation}
Second, notice
\begin{align*}
\begin{bmatrix}
0 & Cr_{s_1} & \begin{bmatrix} 0 & A_0 \end{bmatrix} + CR_{s_1}\end{bmatrix} =& 
\begin{bmatrix} 0 & 0 & \begin{bmatrix} 0 & A_0 \end{bmatrix}\end{bmatrix}+\begin{bmatrix} 0 & C \end{bmatrix}K_{s-1}(\lambda) =:\\ & \begin{bmatrix} 0 & A_0 \end{bmatrix} + \widetilde{C}K_{s-1}(\lambda).
\end{align*}
Third, notice
\begin{align*}
\begin{bmatrix}
\begin{bmatrix} \lambda A_k & 0 \end{bmatrix} + C^\prime T_{s_2}(\lambda) & C^\prime t_{s_2}(\lambda) & 0 
\end{bmatrix} = &
\begin{bmatrix}
\begin{bmatrix} \lambda A_k & 0 \end{bmatrix} & 0 & 0 
\end{bmatrix}+
\begin{bmatrix} C^\prime & 0 \end{bmatrix} K_{s-1}(\lambda) =:\\ & 
\begin{bmatrix} \lambda A_k & 0 \end{bmatrix} + \widetilde{C}^\prime K_{s-1}(\lambda).
\end{align*}
Finally, writing $B = \left[  \begin{smallmatrix} b \\ \widehat{B} \end{smallmatrix} \right]$  and $B^\prime = \left[  \begin{smallmatrix} \widehat{B}^\prime \\ b^\prime \end{smallmatrix} \right]$, where $b$ and $b^\prime$ are, respectively, the first and last block-rows of $B$ and $B^\prime$, it is not difficult to check that
\begin{align*}
\begin{bmatrix}
\widetilde{M}^Z(\lambda) & m_1^Z(\lambda) & 0 \\
m_1^Z(\lambda)^\mathcal{B} & \lambda A_{h+1}+A_h & m_1^Q(\lambda) \\
0 & m_1^Q(\lambda)^\mathcal{B} & \widehat{M}^Q(\lambda)
\end{bmatrix} =& M(\lambda;P_{k-1}^{k-1})+\begin{bmatrix}
\widehat{B}^\prime & 0 \\
b^\prime & b \\
0 & B
\end{bmatrix}K_{s-1}(\lambda)+ \\& \hspace{2.5cm}
K_{s-1}(\lambda)^T \begin{bmatrix}
(\widehat{B}^\prime)^\mathcal{B} & (b^\prime)^\mathcal{B} & 0 \\ 0 & b^\mathcal{B} & B^\mathcal{B} 
\end{bmatrix} =: \\ &
M(\lambda;P_{k-1}^{k-1})+\widetilde{B}K_{s-1}(\lambda)+K_{s-1}(\lambda)^T(\widetilde{B})^\mathcal{B}.
\end{align*}
Thus, we have proven that there exists a block-permutation matrix $\Pi_{\mathbf{c}}^n$ such that the permuted GFPR  $(\Pi_{\mathbf{c}}^n)^\mathcal{B}L_p(\lambda) \Pi_{\mathbf{c}}^n$ is of the form 
\[
  \left[
        \begin{array}{c:c:c|c}
  -A_k & \begin{bmatrix} \lambda A_k & 0 \end{bmatrix} + \widetilde{C}^\prime K_{s-1}&0& 0\\
            \hdashline
            \begin{bmatrix} \lambda A_k \\ 0 \end{bmatrix} + K_{s-1}^T(\widetilde{C}^\prime )^\mathcal{B} & M(\lambda;P_{k-1}^{k-1}) + \widetilde{B}K_{s-1}+K_{s-1}^T\widetilde{B}^\mathcal{B}& \begin{bmatrix} 0 \\ A_0 \end{bmatrix} +  K_{s-1}^T \widetilde{C}^\mathcal{B} & K_{s-1}^T \widetilde{D}^\mathcal{B} \\
            \hdashline
            0&\begin{bmatrix} 0 & A_0 \end{bmatrix} + \widetilde{C}K_{s-1} & \phantom{\Big{(}} -\lambda A_0 \phantom{\Big{)}} & 0 \\  \hline
            0 & \phantom{\Big{(}} \widetilde{D}K_{s-1} \phantom{\Big{)}} & 0 & 0
        \end{array}
    \right],
\]
which belongs to $\langle \mathcal{O}_2^P \rangle$. 
Furthermore, if statement (i)  holds, then the matrices $D$, $D^{\mathcal{B}}$, $D^\prime$  and $(D')^{\mathcal{B}}$ in \eqref{eq:auxD} are nonsingular by Theorems \ref{thm:aux_GFPR3} and \ref{thm:aux_GFPR3}.
Hence, in this situation, the matrices $\widetilde{D}$ and $\widetilde{D}^{\mathcal{B}}$  are nonsingular.
 Then, taking into consideration the statements (ii) and (iii), by Theorem \ref{thm:second family}, the pencil $(\Pi_{\mathbf{c}}^n)^\mathcal{B}L_p(\lambda) \Pi_{\mathbf{c}}^n$ is a strong linearization of $P(\lambda)$.
\smallskip

For proving Theorem \ref{thm:main_GFPR} in cases (a), (c) and (d), one may proceed as follow. 
First, one apply Lemma \ref{lem:splitting} to the block-symmetric GFPR in order to ``split'' the pencil into two simpler block-symmetric GFPR $F(\lambda)$ and $G(\lambda)$, as we have done at the beginning of the proof for case (b).  
Depending on the parity of $k$ and $h$, these simpler GFPR are as one of those considered in Theorems \ref{thm:aux_GFPR1}, \ref{thm:aux_GFPR2}, \ref{thm:aux_GFPR3} or \ref{thm:aux_GFPR4}.
Applying the corresponding theorems to $F(\lambda)$ and $G(\lambda)$, and a simple block-permutation, one easily obtain the desired result.
\end{proof}

\end{document}